\documentclass[11pt, a4paper,leqno]{amsart}
\usepackage{amsmath,amsthm,amscd,amssymb,amsfonts, amsbsy}
\usepackage{latexsym}
\usepackage{txfonts}
\usepackage{exscale}
\usepackage{bbm}

\usepackage[margin=1.35in]{geometry}

\usepackage[colorlinks,citecolor=red, linkcolor=red, pagebackref,hypertexnames=false]{hyperref}

\usepackage{pgf}
\usepackage{color}

\calclayout
\allowdisplaybreaks

\theoremstyle{plain}
\newtheorem{theorem}[equation]{Theorem}
\newtheorem{lemma}[equation]{Lemma}
\newtheorem{corollary}[equation]{Corollary}
\newtheorem{proposition}[equation]{Proposition}

\theoremstyle{definition}
\newtheorem{definition}[equation]{Definition}

\theoremstyle{remark}
\newtheorem{remark}[equation]{Remark}

\numberwithin{equation}{section}

\newcommand{\eps}{\varepsilon}

\newcommand{\dist}{\operatorname{dist}}

\newcommand{\dv}{\operatorname{div}}

\newcommand{\re}{\mathbb{R}}
\newcommand{\rn}{\mathbb{R}^n}

\newcommand{\reu}{\mathbb{R}^{n+1}_+}
\newcommand{\ree}{\mathbb{R}^{n+1}}
\newcommand{\N}{\mathbb{N}}

\newcommand{\dd}{\mathbb{D}}

\newcommand{\om}{\Omega}

\newcommand{\F}{\mathcal{F}}

\newcommand{\LL}{\mathcal{L}}

\newcommand{\W}{\mathcal{W}}

\newcommand{\sbf}{{\bf S}}

\newcommand{\pom}{\partial\Omega}

\newcommand{\hm}{\omega}

\newcommand{\vp}{\varphi}

\newcommand{\ttm}{{\tt m}}

\renewcommand{\P}{\mathcal{P}}

\newcommand{\Bb}{{\bf B}}

\renewcommand{\emptyset}{\mbox{\textup{\O}}}

\DeclareMathOperator{\supp}{supp}

\DeclareMathOperator{\diam}{diam}

\DeclareMathOperator{\interior}{\tt int}

\DeclareMathOperator*{\esssup}{ess\,sup}

\title{The Dirichlet Problem for elliptic equations with singular drift terms}

\begin{document}

\author{S. Hofmann}
\address{
Department of Mathematics
\\
University of Missouri
\\
Columbia, MO 65211, USA} 
\email{hofmanns@missouri.edu}

\thanks{The author was partially supported by NSF grant  
DMS-2349846. }

\date{\today}



\begin{abstract}
We establish $L^p$ solvability
of the Dirichlet problem, for some finite $p$,  
in a 1-sided chord-arc domain $\Omega$ 
(i.e., a uniform domain with Ahlfors-David regular boundary),
for elliptic equations of the form
\[
Lu=-\dv(A\nabla u) \,+\, \Bb\cdot \nabla u=:L_0 u\,+\, \Bb\cdot \nabla u=0\,,
\]
given that the analogous result holds (typically with a different value of $p$) for the homogeneous second order operator $L_0$.  Essentially, we 
assume that $|\Bb(X)|\lesssim \dist(X,\pom)^{-1}$, and that
$|\Bb(X)|^2\dist(X,\pom) dX$ is a Carleson measure in $\Omega$.

\end{abstract}

\maketitle

\tableofcontents

\section{Introduction, history,  and statement of main result}\label{Sintro}

In this work, we treat the Dirichlet problem with data in $L^p$ (denoted $(D)_p$), for the equation $Lu=0$, where
$L$ is an elliptic operator with a drift term, i.e.,
\begin{equation}\label{Ldef}
Lu:=-\dv(A\nabla u) \,+\, \Bb\cdot \nabla u\,=:\,L_0 u\,+\, \Bb\cdot \nabla u\,,
\end{equation}
in the setting of a {\bf 1-sided chord arc domain}, i.e., a uniform domain $\Omega\subset \ree$
with ($n$-dimensional) Ahlfors-David regular boundary (all notation and terminology will be defined in the sequel).  Our results are perturbative in nature:  we assume that the Dirichlet problem is solvable for the purely second order %, uniformly elliptic 
operator $L_0$, with data in $L^p$ for some finite $p$, and show that the analogous
solvability result holds (for a possibly different $p$) for the operator $L$. 
Writing $\delta(X):= \dist(X,\pom)$,
we assume in particular that for some $M_0<\infty$, $\Bb$ satisfies
\begin{equation}\label{driftsize}
|\Bb(X)|\,\leq\, \sqrt{M_0} \,\delta(X)^{-1}\,,\qquad \text{a.e. } X\in \Omega\,,
\end{equation} 
and that $d\mu(X):=|\Bb(X)|^2\delta(X) dX$
 is a Carleson measure in $\Omega$, i.e.,
\begin{equation}\label{driftCarleson}
\|\Bb\|^2_{T^{\infty}_2(\Omega)}\,:=\,
\sup_B % {z\in \pom,\, 0<r<\diam(\Omega)}
r_B^{-n}\iint_{\Omega\cap B} |\Bb(X)|^2\,\delta(X) \,dX\,\leq M_0
 \,<\,\infty\,,
\end{equation}
where the supremum runs over all balls $B\subset \ree$ centered on $\pom$, and $r_B$ is the radius of $B$.  
More precisely, we consider a slightly stronger version of \eqref{driftCarleson}, in which $\Bb$ is replaced by its essential supremum on a Whitney ball, i.e.,
% let (as usual) $B(X,r)$ denote the Euclidean ball centered at $X$ of radius $r$, and 
for $X\in \Omega$, set  % let $r_X:= \delta(X)/4$, and set
\begin{equation}\label{bstardef}
{\Bb_*}(X):= \esssup_{Y:\, |X-Y|<\delta(X)/4} |\Bb(Y)|\,,
\end{equation}
and we suppose that
\begin{equation}\label{driftmax}
M_1:=\|\Bb_*\|^2_{T^{\infty}_2(\Omega)}\,:=\,
\sup_B 
r_B^{-n}\iint_{\Omega\cap B} \big(\Bb_*(X)\big)^2\,\delta(X) \,dX
 \,<\,\infty\,,
\end{equation}
We observe that \eqref{driftmax} implies each of \eqref{driftsize} 
and \eqref{driftCarleson} (with $M_0\lesssim M_1$).

The % coefficients of the 
principal part of the operator is assumed to be uniformly elliptic, i.e., $A$ % =A(X)$ 
is an $(n+1)\times(n+1)$ matrix of (real) bounded measurable
 coefficients, not necessarily symmetric, defined on
$\Omega\subset \ree$, satisfying the 
uniform ellipticity condition
\begin{equation}
\label{uniellip} \lambda|\xi|^{2}\leq\,\langle A(X)\xi,\xi\rangle
:= \sum_{i,j=1}^{n+1}A_{ij}(X)\xi_{j}\xi_{i}, \quad
  \Vert A\Vert_{L^{\infty}(\Omega)}\leq\Lambda,
\end{equation}
 for some $\lambda>0$, $\Lambda<\infty$, and for all $\xi\in\mathbb{R}^{n+1}$, and a.e.~$X\in\Omega$.

 Our main result is as follows.
 \begin{theorem}\label{Tmain} Let $\Omega\subset\ree$, $n\geq 2$, 
 be a 1-sided chord arc domain (CAD), and let $L_0$ and $L$ be defined as in \eqref{Ldef}, where the coefficients of the drift term, and of the principal part, satisfy (respectively) \eqref{driftmax} % \eqref{driftsize}-\eqref{driftCarleson}, 
 and \eqref{uniellip}.  Suppose that for some $p_0<\infty$, the 
 Dirichlet problem $(D)_{p_0}$ is solvable for the equation $L_0u=0$ in $\Omega$.  
 Then there is a finite $p$ (possibly larger than $p_0$), such that $(D)_p$ is 
 solvable for $Lu=0$ in $\Omega$.  
 \end{theorem}
 
 Of course, this result is quantitative, i.e., 
 the exponent $p$, as well as the constants that will appear in our estimates for 
 solutions of $(D)_p$, will depend only on the ``allowable parameters": 
 dimension, the constants in \eqref{driftmax} 
 and \eqref{uniellip}, the 1-sided chord arc constants
 for $\Omega$ and its boundary, the exponent $p_0$, and the constants in the quantitative estimates for solutions of $(D)_{p_0}$ for $L_0$.
 
 As an easy corollary of Theorem \ref{Tmain}, we shall also
  prove the following.  We
 refer the reader to
 \cite{AHMMT} for the definition of the {\it weak local John property}, and to
 \cite{DS1,DS2} (or to \cite{AHMMT}) for the definition of
  {\it uniform rectifiability}, but
 note that uniform rectifiability entails in particular that 
 $\pom$ satisfies the Ahlfors-David regularity (ADR) condition (see Definition
 \ref{def1.ADR}).  See also Remark \ref{remarkcompare}
 below. 

 % Terminology will be defined in the sequel (see Section \ref{Sdefinitions}).
 
  \begin{theorem}\label{T2} Let $\Omega\subset\ree$, $n\geq 2$, 
 be an open set which satisfies the corkscrew condition (see Definition
 \ref{def1.cork}), and the
 weak local John condition, and which has 
 a uniformly rectifiable boundary.
Let $L_0$ and $L$ be defined as in \eqref{Ldef}, where the coefficients of the drift term, and of the principal part, satisfy (respectively) \eqref{driftmax} 
 and \eqref{uniellip}.  Assume further that the coefficient matrix $A$ satisfies the 
 DKP condition (see Definition \ref{def.DKP}).  
 Then there is a finite $p$ such that $(D)_p$ is 
 solvable for $Lu=0$ in $\Omega$.  
 \end{theorem}
 
 Again, this result is quantitative.
 
 \begin{remark}\label{remarkcompare}
 To compare Theorem \ref{T2} to Theorem \ref{Tmain}, note that in the former,
we have replaced the Harnack Chain condition with 
a weaker connectivity condition, namely the weak local John property.  On the other hand, we have imposed a stronger geometric
hypothesis of another sort (uniform recifiability of $\pom$), and we have 
restricted the class of coefficient matrices.   We further remark that 
open sets satisfying the weak local John condition,
with uniformly rectifiable boundary, are a natural class of domains 
in which to consider solvability of 
the $L^p$ Dirichlet problem with $p<\infty$, 
in the sense that these geometric conditions characterize such solvability in the case of the Laplacian;  
see \cite{AHMMT}.  Finally, we note that the proof of Theorem \ref{T2}
will not make explicit use of uniform rectifiability or the 
weak local John property
{\em per se}, 
but rather will rely on another property 
equivalent to the combination of those two, namely,
that $\Omega$ satisfies an ``interior big pieces of chord-arc domains" condition
(see Definitions \ref{def1.cad} and \ref{def1.ibp}; see 
also \cite{AHMMT} for the proof of the stated equivalence).
\end{remark}

 Let us recall some related work.
Theorem \ref{Tmain} is an extension
of a previous result of the same type, valid in Lipschitz domains, which
was proved by J.~L.~Lewis and the present author 
in \cite{HL} (see also \cite{KP}, for related results).  The purpose of the present work is to extend the result of \cite{HL} 
to the setting of a 1-sided chord-arc domain.
In a Lipschitz domain, one can (after a suitable localization)
exploit the use of a 
pullback mechanism to reduce matters to the case that $\Omega$ 
is the half-space $\reu$.  Such an approach is clearly unavailable in the much more general setting of a 1-sided chord arc domain, the boundary of which
need not be given locally as a graph.  In \cite{HL}, the authors 
also proved a parabolic version of
Theorem \ref{Tmain}, valid in ``regular Lip(1,1/2) domains" (see, e.g., \cite{BHMN} for the definition) again using a pullback argument.  Indeed, the parabolic case was the main focus of the work \cite{HL}.  The proof of the elliptic version of the result in \cite{HL} 
exploited the fact that, 
given a certain local ampleness
property of elliptic-harmonic measure (see \eqref{eq2.Bourgain1} below), 
one can establish doubling of the elliptic-harmonic measure\footnote{For
parabolic measure, doubling
remains an open question, even in the presence of
the parabolic version of \eqref{eq2.Bourgain1}.} associated to an operator with a drift. We observe 
that in the absence of the ampleness property \eqref{eq2.Bourgain1}, 
it is not clear whether 
the doubling property holds. % for operators with a drift.
The proof of doubling given in \cite{HL} relies on first establishing
a pointwise (local) upper bound for the Green function (in turn deduced from 
\eqref{eq2.Bourgain1}), and this Green function bound
can fail, in general, for operators with a drift\footnote{Indeed, 
a counter-example has been presented in \cite{P}.  On the other hand,
the pointwise local upper bound for the Green function
is valid (as we shall show in the sequel) when $M_0$ is sufficiently
small in \eqref{driftsize}, and also in the presence of \eqref {eq2.Bourgain1},
as shown in \cite{HL}; the latter fact may be extended to our setting
(see Section \ref{appdoubl}).}.

As in \cite{HL}, we make essential
use of the method of ``extrapolation of Carleson measures"
(see the introduction to Section
\ref{section2key} for a brief discussion of this method and its history).
% a technique introduced in \cite{LM}.   
Nevertheless, our implementation of this method
is different to that of \cite{HL}: necessarily so, in order 
to treat non-graph boundaries to which the aforementioned
pullback mechanism does not apply.
We remark that we obtain the conclusion of
 Theorem \ref{Tmain} (i.e., $L^p$ solvability of the Dirichlet problem) without
 relying on the doubling property of elliptic measure
 for the operator with a drift.  
 On the other hand, 
 our arguments yield \eqref {eq2.Bourgain1} {\em a posteriori}, and
 doubling then follows, as in \cite[Ch.~III, Lemma 4.3]{HL} (see 
Section \ref{appdoubl}).

We conclude this brief historical review by mentioning several other fairly recent 
works in which the authors have studied 
$L^p$ boundary value problems for operators with a drift (or with lower order terms
more generally).  In the 
papers \cite{BHLMP} and \cite{S1}, the authors obtain $L^2$ solvability results for 
boundary value problems,
using layer potential methods, 
under different (but somewhat related) conditions on the drift 
coefficient than the one considered here and in \cite{HL}.   
In \cite{BHLMP}, for example,
$\Omega=\reu=\{(x,t)\in\rn\times (0,\infty)\}$, and the drift coefficient $\Bb$ is 
$t$-independent and belongs to $L^n(\rn)$ (with small norm).  Note that 
for such $\Bb$, the measure $d\mu(x,t) := |\Bb(x)|^2 \,t\, dx  dt$ defines a
Carleson measure in the half-space, with norm on the order of
$\|\Bb\|_{L^n(\rn)}^2$, as one may verify routinely by using H\"older's inequality
 in the horizontal variable $x$, with exponent $n/2$.

In work that is more closely related to that of \cite{HL} and the present paper, J. Feneuil and B. Poggi \cite{FP}
have studied the Dirichlet problem in a class of 
``PDE friendly" domains that includes 1-sided chord arc domains, for drifts that satisfy a Carleson measure 
condition of the sort considered here and in \cite{HL}, under additional
structural assumptions on the drift, namely, that the drifts 
arise by rewriting certain perturbations of the second order part of the operator. More precisely, the results of \cite{FP} apply to an elliptic operator $L$ that may be written as
$L = -\dv A\nabla\ +\Bb\cdot\nabla =-\dv A_1 \nabla$, 
for some coefficient matrix $A_1$ which in turn may be viewed
as a perturbation of $A$ in some suitable sense.

\smallskip

\noindent{\bf Outline of the paper}.
The paper is organized as follows.  In Section \ref{Sdefinitions}, we introduce our notation and terminology, and we record some previously known geometric facts.  In Section \ref{Ssmall1},
we present a number of preliminary results\footnote{E.g., existence of solutions and elliptic measure, continuity of solutions up to the boundary, Green function estimates, etc.}, 
well-known for the homogeneous second order
$L_0=-\dv A\nabla$, which we extend to the operator $L=L_0+\Bb\cdot\nabla$ under certain conditions (e.g., sometimes, but not always, smallness in the pointwise size condition \eqref{driftsize}).  We also
give an approximation result, which will allow us to work with operators for which the drift coefficient $\Bb$ has been truncated near the boundary, and is thus (qualitatively) bounded.
Of course, all of our quantitative estimates will be independent of such truncations.
There is a considerable amount of material to review in Section \ref{Ssmall1}, 
but on the other hand, most of it is fairly routine and standard, so the expert reader 
could probably skim this section rapidly, at least on a first reading.
Sections \ref{Ssmall2}, \ref{section2key}, and  \ref{Sgeneral} are the heart of the matter.
In Section \ref{Ssmall2}, we present a small constant perturbation result, in the spirit
of \cite[Theorem 2.5]{FKP}, in which we show that the elliptic measures for $L$ and $L_0$,
denoted by $\hm$ and $\hm_0$,
are mutually absolutely continuous in the sense of $A_\infty$, given that $\Bb$ is
small in the sense of a certain $\hm_0$-adapted Carleson measure condition.
In Section \ref{section2key}, we prove two lemmas that will play a key role in removing the smallness assumption in Section \ref{Ssmall2}.  
In Section \ref{Sgeneral}, we 
% carry out the latter main step in 
remove the small constant restriction 
via the method of ``extrapolation of Carleson measures." 
% the technique mentioned above. 
Finally, in Section \ref{appdoubl}, under the assumptions of 
Theorem \ref{Tmain},
we establish the doubling property of elliptic measure associated
to $L= L_0 +\Bb\cdot\nabla$, following the argument in 
\cite[Ch.~III, Lemma 4.3]{HL}.  We then use Theorem \ref{Tmain}
and the doubling property to deduce Theorem \ref{T2}
as an easy corollary.

\section{Notation, terminology, and some known geometric results}\label{Sdefinitions}

% Lip$_c$, $Y^{1,2}$ etc. indicator function $1_E$,surface measure, $(D)_p$, 
% continuous (D), % CS, HC, 
% uniform/1-sided NTA, ADR, 1-sided chord arc,
% elliptic measure, % $A_\infty$, weak-$A_\infty$ (equiv to (D)p), 
% Bourgain, H\"older at boundary, 
% Benewitz-Lewis result, 
% dyadic boundary cubes, Whitney regions and cubes, 
% trees/sawtooths (always define trees to be at least semi-coherent), 
% CFMS, comparison principle (BHP), 
% trees (define them always to be semi-coherent at least) and their maximal (top) cubes.

In the sequel, we shall assume that $\Omega\subset \ree$ is an open set (typically with additional properties to be specified below), we let $B=B(X,r)$ denote the 
Euclidean ball in $\ree$ of radius $r$, centered at $X$, and for $X\in \pom$, we let
\[
\Delta = B\cap \pom
\]
denote the corresponding surface ball on the boundary.  
% For $\kappa>0$, as usual we let
% $\kappa B := B(X,\kappa r)$ denote the $\kappa$-fold concentric dilate of the ball $B$, and 
% similarly $\kappa\Delta := \kappa B\cap \pom$.  We will also sometimes 
% write $B_r,\Delta_r$ to denote, respectively, a ball or surface ball of radius $r$, 
% whose center has been left implicit.
We list further notation as follows.

\begin{list}{$\bullet$}{\leftmargin=0.4cm  \itemsep=0.2cm}

\item We use the letters $c,C$ to denote harmless positive constants, not necessarily
the same at each occurrence, which depend only on the ``allowable parameters'': dimension, and the constants appearing in the hypotheses of the main theorem.  We shall also
sometimes write $a\lesssim b$ and $a \approx b$ to mean, respectively,
that $a \leq C b$ and $0< c \leq a/b\leq C$, where the constants $c$ and $C$ are as above, unless
explicitly noted to the contrary.  At times, we shall designate by $M$ a particular constant whose value will remain unchanged throughout the proof of a given lemma or proposition, but
which may have a different value during the proof of a different lemma or proposition.

\item Given a domain $\Omega \subset \ree$, we shall typically
use lower case letters $x,y,z$, etc., to denote points on $\partial \Omega$, and capital letters
$X,Y,Z$, etc., to denote generic points in $\ree$ (especially those in $\ree\setminus \partial\Omega$).

% \item The open $(n+1)$-dimensional Euclidean ball of radius $r$ will be denoted
% $B(X,r)$ (or sometimes $B(x,r)$, when the center $x$ lies on $\partial \Omega$). 
% A ``surface ball'' is denoted $\Delta(x,r):= B(x,r) \cap\partial\Omega,$ for $x\in\pom$.

\item Given a Euclidean ball $B$ or surface ball $\Delta$, its radius will be denoted
$r_B$ or $r_\Delta$, respectively.

\item Given a Euclidean or surface ball $B= B(X,r)$ or $\Delta = \Delta(x,r)$, its concentric
dilate by a factor of $\kappa >0$ will be denoted
by $\kappa B := B(X,\kappa r)$ or $\kappa \Delta := \Delta(x,\kappa r).$

\item For $X \in \Omega$, we set $\delta(X):= \dist(X,\partial\Omega)$.

\item We let $\mathcal{H}^n$ denote $n$-dimensional Hausdorff measure, and let
$\sigma := \mathcal{H}^n\big|_{\partial\Omega}$ denote the surface measure on $\partial \Omega$.

\item For a Borel set $E\subset \ree$, we let $1_E$ denote the usual
indicator function of $E$, i.e. $1_E(x) = 1$ if $x\in E$, and $1_E(x)= 0$ if $x\notin E$.

\item For a Borel set $E\subset \ree$,  we let $\interior(E)$ denote the interior of $E$.
If $E\subset \partial\Omega$, then $\interior(E)$ will denote the relative interior, i.e., the largest relatively open set in $\partial\Omega$ contained in $E$.  Thus, 
 $\partial E := \overline{E} \setminus {\rm int}(E)$ will denote the boundary of a set
$E\subset \partial\Omega$. 

\item For a Borel set $E$, we denote by $C(E)$ (resp. Lip$(E)$) the space of continuous (resp. Lipschitz continuous) functions on
$E$, and by $C_c(E)$ (resp. Lip$_c(E)$) the subspace of $C(E)$ (resp. Lip$(E)$)
with compact support in $E$. 
% and by $C_b(E)$ the
% space of bounded continuous functions on $E$.  If $E$ is unbounded, we denote by
% $C_0(E)$ the space of continuous functions on $E$ converging to $0$ at infinity.

\item For a Borel subset $E\subset\partial\Omega$, we
set $\fint_E f d\sigma := \sigma(E)^{-1} \int_E f d\sigma$.

\item We shall use the letter $I$ (and sometimes $J$)
to denote a closed $(n+1)$-dimensional Euclidean cube with sides
parallel to the co-ordinate axes, and we let $\ell(I)$ denote the side length of $I$.
We use $Q$ to denote a dyadic ``cube''
on $\partial \Omega$.  The
latter exist, given that $\partial \Omega$ is ADR  (cf. \cite{DS1}, \cite{Ch}), and enjoy certain properties
which we enumerate in Lemma \ref{lemmaCh} below.

\begin{definition}\label{def1.ADR}
({\bf Ahlfors-David regular, or ADR}). We say that a closed set $E \subset \ree$ is $n$-dimensional ADR (or simply ADR) (``Ahlfors-David regular'') if
there is some uniform constant $C$ such that
\begin{equation} \label{eq1.ADR}
\frac1C\, r^n \leq \mathcal{H}^n(E\cap B(x,r)) \leq C r^n,\,\,\,\forall r\in(0,R_0),x \in E,\end{equation}
where $R_0$ is the diameter of $E$ (which may be infinite).   When $E=\partial \Omega$,
the boundary of a domain $\Omega$, we shall sometimes for convenience simply
say that ``$\Omega$ has the ADR property'' to mean that $\partial \Omega$ is ADR.
\end{definition}

\begin{definition}\label{defdouble} {\bf (Doubling)} We say that a measure $\ttm$ defined on
$\pom$ is {\em doubling} (with doubling constant $N_{db}$), if
for every surface ball $\Delta\subset \pom$,
we have
\[
\ttm(2\Delta)\,\leq\, N_{db} \,\ttm(\Delta)\,.
\]
\end{definition}

\begin{lemma}\label{lemmaCh}\textup{({\bf Existence and properties of the ``dyadic grid''})
\cite{Ch,HK}.}
Suppose that $E\subset \ree$ is an ADR set.  Then there exist
constants $ a_0>0,\, s>0$ and $C_1<\infty$, depending only on $n$ and the
ADR constant, such that for each $k \in \mathbb{Z},$
there is a collection of Borel sets (``cubes'')
$$
\dd_k:=\{Q_{j}^k\subset E: j\in \mathfrak{I}_k\},$$ where
$\mathfrak{I}_k$ denotes some (possibly finite) index set depending on $k$, satisfying

\begin{list}{$(\theenumi)$}{\usecounter{enumi}\leftmargin=.8cm
\labelwidth=.8cm\itemsep=0.2cm\topsep=.1cm
\renewcommand{\theenumi}{\roman{enumi}}}

\item $E=\cup_{j}Q_{j}^k\,\,$ for each
$k\in{\mathbb Z}$.

\item If $m\geq k$ then either $Q_{i}^{m}\subset Q_{j}^{k}$ or
$Q_{i}^{m}\cap Q_{j}^{k}=\emptyset$.

\item For each $(j,k)$ and each $m<k$, there is a unique
$i$ such that $Q_{j}^k\subset Q_{i}^m$.

\item $\diam\big(Q_{j}^k\big)\leq C_1 2^{-k}$.

\item Each $Q_{j}^k$ contains some ``surface ball'' $\Delta \big(x^k_{j},a_02^{-k}\big):=
B\big(x^k_{j},a_02^{-k}\big)\cap E$.

\item $\mathcal{H}^n\big(\big\{x\in Q^k_j:{\rm dist}(x,E\setminus Q^k_j)\leq \vartheta \,2^{-k}\big\}\big)\leq
C_1\,\vartheta^s\,\mathcal{H}^n\big(Q^k_j\big),$ for all $k,j$ and for all $\vartheta\in (0,a_0)$.
\end{list}
\end{lemma}

A few remarks are in order concerning this lemma.

\begin{list}{$\bullet$}{\leftmargin=0.4cm  \itemsep=0.2cm}

\item In the setting of a general space of homogeneous type, this lemma has been proved by Christ
\cite{Ch} (see also \cite{HK}), with the
dyadic parameter $1/2$ replaced by some constant $\delta \in (0,1)$.
In fact, one may always take $\delta = 1/2$ (see  \cite[Proof of Proposition 2.12]{HMMM}).
In the presence of the Ahlfors-David
property, the result already appears in \cite{DS1,DS2}. Some predecessors of this construction have appeared in \cite{David88} and \cite{David91}.

\item  For our purposes, we may ignore those
$k\in \mathbb{Z}$ such that $2^{-k} \gtrsim {\rm diam}(E)$, in the case that the latter is finite.

\item  We shall denote by  $\mathbb{D}=\mathbb{D}(E)$ the collection of all relevant
$Q^k_j$, i.e., $$\mathbb{D} := \cup_{k} \mathbb{D}_k,$$
where, if $\diam (E)$ is finite, the union runs
over those $k$ such that $2^{-k} \lesssim  {\rm diam}(E)$.
%Eventually, in considering domains for which the Corkscrew condition holds, we shall restrict our
%attention to $k$ such that $2^{-k}<(K_0)^{-1}\diam(\partial\Omega)$, in the case that the latter is finite,
%where $K_0$ is some large constant to be chosen depending upon the constants
%in the Corkscrew condition.

\item Properties $(iv)$ and $(v)$ imply that there are uniform constants $c,C$
such that for each cube $Q\in\mathbb{D}_k$,
there is a point $x_Q\in Q$, a Euclidean ball $B(x_Q,r_k)$ and a surface ball
$\Delta(x_Q,r_k):= B(x_Q,r_k)\cap E$ with
$r_k=c 2^{-k} \approx {\rm diam}(Q)$
and \begin{equation}\label{cube-ball}
\Delta(x_Q,r_k)\subset Q \subset \Delta(x_Q,Cr_k)\,.
\end{equation}
We shall refer to the point $x_Q$ as the ``center'' of $Q$.

% We shall denote this ball and surface ball by
 %  \begin{equation}\label{cube-ball2}
% B_Q:= B(x_Q,r) \,,\qquad\Delta_Q:= \Delta(x_Q,r),\end{equation}
% and 

% \item Let us now specialize to the case that  $E=\pom$, with $\Omega$ satisfying 
% the Corkscrew condition.  Given $Q\in \mathbb{D}(\partial\Omega)$,
% we shall sometimes refer to a ``Corkscrew point relative to $Q$'', which we denote by
% $X_Q$, and which we define to be the corkscrew point $X_B$ relative to the ball
% $B:=B_Q$ in \eqref{cube-ball2} (see also Definition \ref{def1.cork}).  We note that
% \begin{equation}\label{eq1.cork}
% \delta(X_Q) \approx \dist(X_Q,Q) \approx \diam(Q).
% \end{equation}

\item For a dyadic cube $Q\in \mathbb{D}_k$, we shall
set $\ell(Q) = 2^{-k}$, and we shall refer to this quantity as the ``length''
of $Q$.  Evidently, $\ell(Q)\approx \diam(Q).$
\end{list}

% \begin{list}{$\bullet$}{\leftmargin=0.4cm  \itemsep=0.2cm}

% \item For a dyadic cube $Q \in \mathbb{D}$, we let $k(Q)$ denote the ``dyadic generation''
% to which $Q$ belongs, i.e., we set  $k = k(Q)$ if
% $Q\in \mathbb{D}_k$; thus, $\ell(Q) =2^{-k(Q)}$.

% \end{list}
\end{list}

For future reference, with $E=\pom$, we set
\begin{equation}\label{eq3.4a}
\dd_Q:=\left\{Q'\in\dd(\pom):Q'\subset Q\right\}\,.
\end{equation}

% \steve{\Bl Put this with M. Christ cube definition}
% and given $k\ge 1$,
% \begin{equation}\label{eq3.4aaa}
% \dd_Q^k:=\left\{Q'\in\dd(\pom):Q'\subset Q, \ \ell(Q')=2^{-k}\,\ell(Q)\right\}\,,
% \end{equation}

\begin{definition} ({\bf Tree (or ``stopping time tree")}).  \label{def1.tree}
A tree $\sbf\subset \dd$ is a collection of cubes such that
\begin{enumerate}
\item $\sbf$ has a unique maximal (or ``top") cube $Q(\sbf)$, i.e.,
$Q\subset Q(\sbf)$ for all $Q\in\sbf$.
\item $\sbf$ is {\em semi-coherent}, i.e., if $Q\in \sbf$, and
$Q\subset Q'\subset Q(\sbf)$, then $Q'\in\sbf$.
\end{enumerate}
\end{definition}

\begin{definition} ({\bf Corkscrew condition}).  \label{def1.cork}
Following
\cite{JK}, we say that an open set $\Omega\subset \ree$
satisfies the ``Corkscrew condition'' if for some uniform constant $c\in (0,1)$ and
for every ball $B=B(x,r),$ with $x\in \partial\Omega$ and
$0<r<\diam(\partial\Omega)$, there is a ball
$B(X_B,cr)\subset B(x,r)\cap\Omega$.  The point $X_B\subset \Omega$ is called
a ``Corkscrew point'' relative to $B.$  
% We note that  we may allow $r<C\diam(\pom)$ for any fixed $C$, 
% simply by adjusting the constant $c$.
\end{definition}

\begin{remark}\label{remark1.2}
We note that, on the other hand, every $X\in\Omega$, with $\delta(X)<\diam(\pom)$,
may be viewed as a Corkscrew point,
relative to some ball $B$ centered on $\pom$.
Indeed, set $r=2 \delta(X)$,  % with $K>1$, 
fix $\hat{x}\in\pom$ such that $|X-\hat{x}|=\delta(X)$, and let
$B:=B(\hat{x},r)$.
\end{remark}

\begin{definition}({\bf Harnack Chain condition}).  \label{def1.hc} Again following \cite{JK}, we say
that $\Omega$ satisfies the Harnack Chain condition if there are uniform positive constants
$c,C$ such that
for every $\rho >0,\, K\geq 1$, and pair of points
$X,X' \in \Omega$ with $\delta(X),\,\delta(X') \geq\rho$ and $|X-X'|<K\,\rho$, 
there is a chain of open balls
$B_1,\dots,B_N \subset \Omega$, $N\leq C(K)$,
with $X\in B_1,\, X'\in B_N,$ $B_k\cap B_{k+1}\neq \emptyset$
and $c\diam (B_k) \leq \dist (B_k,\partial\Omega)\leq C\diam (B_k).$  The chain of balls is called
a ``Harnack Chain''.
\end{definition}

\begin{definition}({\bf 1-sided NTA, aka ``uniform" domain}). \label{def1.nta} 
% Again following \cite{JK}, 
We say that a
domain $\Omega\subset \ree$ is 1-sided NTA ({\it Non-tangentially accessible})
(aka ``uniform") if it satisfies the Corkscrew and
Harnack Chain conditions.
\end{definition}

\begin{definition}({\bf 1-sided CAD and CAD}). \label{def1.cad}  
We say that a connected open set $\om \subset \ree$
is a 1-sided CAD ({\it 1-sided Chord-arc domain}), if it is a uniform (1-sided NTA) 
domain, and if $\pom$ is ADR.  If {\it in addition}, 
$\Omega_{ext}:=\ree\setminus \overline{\Omega}$ also satisfies the 
corkscrew condition, then we say that  $\Omega$ is a CAD 
({\it Chord-arc domain}).
\end{definition}

% Given a class {\bf A} of domains, defined by 
% certain quantitative properties (which we refer to as the {\bf A}-constants),
% with $\delta(X):= \diam(X,\pom)<\diam(\pom)$, \in {\bf A}

\begin{definition}({\bf Interior Big Pieces of Chord-arc Domains (IBPCAD)}). \label{def1.ibp} 
An open set $\om \subset \ree$ with ADR boundary
has {\it Interior Big Pieces of Chord-arc Domains} (IBPCAD) if 
there exist positive constants $\eta$ and $C$, %and $R\geq 2$, 
such that for each $X\in \Omega$, 
with $\delta(X)<\diam(\pom)$,
there is a chord-arc subdomain $\Omega_X\subset \Omega$ satisfying
\begin{itemize}
\item $X\in \Omega_X $.
%\item $X\in \Omega_X$.
\item $\dist(X,\pom_X) \geq \eta \delta(X)$.
\item $\diam(\Omega_X) \leq C\delta(X)$.
\item $\sigma(\pom_X\cap \Delta_X) \geq \,\eta \sigma(\Delta_X) \,\approx\, \eta\delta(X)^n$, where 
$\Delta_X:= B(X,10\delta(X)) \cap \pom$.
\item The chord-arc constants of the subdomains $\Omega_X$ are uniform in $X$.
\end{itemize}
\end{definition}

\begin{definition}({\bf DKP condition}). \label{def.DKP}  
A coefficient matrix $A$ is said to satisfy the Dahlberg-Kenig-Pipher (``DKP")
condition in $\Omega$ if $A$ is locally Lipschitz in $\Omega$, with
 $\delta(X) |\nabla A(X)| \in L^\infty$, and $d\mu(X):=
 |\nabla A(X)|^2 \delta(X) dX$ is a Carleson measure in
 $\Omega$.  See \cite{KP}.
\end{definition}

%We also recall that in \cite{JK}, the authors introduce
%the notion of  a  ``non-tangentially accessible'' (NTA) domain:
%$\Omega$ is said to be NTA if
%it satisfies the Corkscrew
%and Harnack Chain conditions
%(``interior Corkscrew and Harnack Chain''), and also if
%the exterior domain, $\Omega_{ext}:= \ree\setminus \overline{\Omega}$,
%satisfies the Corkscrew condition
%(``exterior Corkscrew'').  In the present paper we impose no hypothesis
%on the exterior domain.

\begin{definition}\label{Y12and120} ($Y^{1,2}$, $W^{1,2}$ {\bf and} $Y^{1,2}_0$, 
$W^{1,2}_0$).
As usual, $W^{1,2}(\Omega):=\{f\in L^{2}(\Omega): \nabla f \in L^2(\Omega)\}$,
endowed with the norm $\|f\|_{W^{1,2}(\Omega)}:= \|f\|_{L^2(\Omega)} +
\|\nabla f\|_{L^2(\Omega)}$, and $W^{1,2}_0(\Omega)$ is the completion of
$C_0^\infty(\Omega)$ in the $W^{1,2}(\Omega)$ norm.

Also as usual, in $\re^d$ set $2^*:= 2d/(d-2)$, thus for us, $d=n+1$ and
$2^*:= 2(n+1)/(n-1)$.  Then 
$Y^{1,2}(\Omega):=\{f\in L^{2^*}(\Omega): \nabla f \in L^2(\Omega)\}$, with norm
\[
\|f\|_{Y^{1,2}(\Omega)} \,:= \,\|f\|_{L^{2^*}(\Omega)}\, +\, 
\|\nabla f\|_{L^{2}(\Omega)}\,.
\]
We let $Y^{1,2}_0(\Omega)$ denote the completion of $C_0^\infty(\Omega)$ in the 
$Y^{1,2}(\Omega)$ norm.  
\end{definition}

\begin{definition}\label{conentmaxdef} {\bf (Non-tangential approach regions and maximal functions)}.
Given $z\in \pom$, we define a non-tangential approach region (``cone") with vertex at $z$ by
\begin{equation}\label{eqconedef}
\Gamma(z) := \{Y\in\Omega: |z-Y|< 10\delta(Y)\}\,,
\end{equation}
and for a measurable function $H$ defined in $\Omega$, let
\begin{equation}\label{eqntmaxdef}
N_*H(z):= \sup_{Y\in\Gamma(z)} |H(Y)|\,,
\end{equation}
denote the non-tangential maximal function of $H$.
\end{definition}

\begin{definition}\label{DLp}
{\bf (Dirichlet problem with data in $L^p$)}. 
 \begin{center}
$(D)_p\left\{ \begin{array}{rl}Lu&=0  \textrm{ in } \Omega  \\
u\vert_{\pom}&=f\in L^p(\pom)
\\N_* u &\in L^p(\pom)\,,
\end{array} \right.$
\end{center} 
where $u\vert_{\pom}=f$ is understood in the sense of nontangential convergence. 
\end{definition}

\begin{definition}\label{Dc}
{\bf Continuous Dirichlet problem}.
 \begin{center} 
$(D)\left\{ \begin{array}{rl}Lu&\!\!=0  \textrm{ in } \Omega  \\
u\vert_{\pom}&\!\!=f \in C_c(\pom)
\\ u&\!\!\in C(\overline{\Omega})\,.
\end{array} \right.$
\end{center} 
\end{definition}

\begin{remark} As is well known, solvability of the continuous Dirichlet problem, in conjunction with the weak maximum principle, implies the existence of elliptic measure via the Riesz representation theorem.
\end{remark}

\begin{definition}\label{deflocalAinfty} ({\bf $A_\infty$ and weak-$A_\infty$ for elliptic measure}).  Let $\ttm$ be a doubling measure on $\pom$.
We say that elliptic measure $\hm$  is (locally) in weak-$A_\infty$ on $\pom$, with respect to
$\ttm$, and we write $\hm\in$ weak-$A_\infty(\ttm)$,
if there are  uniform positive constants $C$ and $s$
such that for every ball $B=B(x,r)$ centered on $\pom$, 
with radius  $r<\diam(\pom)/4$, and associated surface ball $\Delta=B\cap\pom$,
\begin{equation}\label{eq1.localwainfty}
\hm^X (E) \leq C \left(\frac{\ttm(E)}{\ttm(\Delta)}\right)^{\! s}\,\hm^X (2\Delta)\,,
\qquad \forall \, X\in\om\setminus 4B\,, \,\,\forall \mbox{ Borel } E\subset \Delta\,;
\end{equation}
equivalently,
if for every ball $B$ and surface ball $\Delta=B\cap\pom$ as above,
and for each  
point $X\in\om\setminus 4B$, $\hm^X\in$ weak-$A_\infty(\ttm,\Delta)$ with uniformly controlled weak-$A_\infty$ constants, i.e., \eqref{eq1.localwainfty} holds uniformly, with $\Delta$ replaced by $\Delta'=B'\cap\pom$, and with $E\subset \Delta'$, for all $B'\subset B$ and all 
$X\in \Omega\setminus 4B$.

If, {\em in addition}, $\hm^X$ satisfies the doubling property (Definition \ref{defdouble}), 
for $X\in \Omega\setminus 4B$,
locally for all $\Delta'=B'\cap\pom$, with $B'\subset B$,
then we say that
elliptic measure $\hm$  is (locally) in $A_\infty$ on $\pom$ with respect to
$\ttm$, and we write $\hm\in A_\infty(\ttm)$.

In the special case that $\ttm=\sigma$, we will sometimes simply write
$\hm\in$ weak-$A_\infty$, or $\hm\in A_\infty$.
\end{definition}

\begin{remark}\label{Dp=Ainfty} 
It is well known that solvability of $(D)_p$ for some finite $p$ is equivalent to
the property that $\hm^X$ belongs (locally) 
to weak-$A_\infty$ (or to $A_\infty$, in the case that
$\hm^X$ is doubling) with respect to $\sigma$, 
in the sense of Definition \ref{deflocalAinfty}, and 
in fact the Poisson kernel $k^X:=d\hm^X/d\sigma$ 
satisfies a reverse H\"older condition with exponent
$q=p/(p-1)$.  See, e.g., \cite{H,HLe}.
\end{remark}

\subsection{Some geometric results}
We record some known geometric facts that will be useful in the sequel.

Let us first make a standard Whitney decomposition of $\Omega$.
Let $\mathcal{W}=\W(\om)$ denote a collection
of (closed) dyadic Whitney cubes of $\om$, so that the cubes in $\mathcal{W}$
form a pairwise non-overlapping covering of $\om$, which satisfy
\begin{equation}\label{Whintey-4I}
4 \diam(I)\leq
\dist(4I,\pom)\leq \dist(I,\pom) \leq 40\diam(I)\,,\qquad \forall\, I\in \mathcal{W}\,\end{equation}
(just dyadically divide the standard Whitney cubes, as constructed in  \cite[Chapter VI]{St},
into cubes with side length 1/8 as large)
and also
$$(1/4)\diam(I_1)\leq\diam(I_2)\leq 4\diam(I_1)\,,$$
whenever $I_1$ and $I_2$ touch.

We fix a small parameter $\tau_0>0$, so that
for any $I\in \W$, and any $\tau \in (0,\tau_0]$,
the concentric dilate
\begin{equation}\label{whitney1}
I^*(\tau):= (1+\tau) I
\end{equation} 
still satisfies the Whitney property
\begin{equation}\label{whitney}
\diam I\approx \diam I^*(\tau) \approx \dist\left(I^*(\tau), E\right) \approx \dist(I,E)\,, \quad 0<\tau\leq \tau_0\,.
\end{equation}

 We recall that in \cite[Section 3, and Appendix A]{HM-I}, there is a 
 construction of Whitney regions, ``Carleson boxes" (and Carleson tents), and ``sawtooth" subdomains corresponding to a tree $\sbf$, with the following properties. 
 % We let $\W$ be a fixed collection of closed Whitney cubes for $\Omega$.
 \begin{proposition}\label{sawtoothprop}\cite[Section 3, and Appendix A]{HM-I}. 
 Let $\Omega$ be a 1-sided chord arc domain.  Then given a tree $\sbf\subset \dd(\pom)$, there is a ``sawtooth" subdomain
 $\Omega_\sbf\subset \Omega$, satisfying the following properties:
\begin{enumerate}
\item $\Omega_\sbf$ is a 1-sided chord arc domain.
\smallskip
\item $\Omega_\sbf = \interior\left(\cup_{Q\in \sbf} \,U_Q\right)$, where $U_Q$ is a
 ``Whitney region" relative to $Q$.
\smallskip
\item The Whitney regions are of the form $U_Q=\cup_{I\in\W(Q)}I^*$,
where $I^*$ is the concentrically fattened version of $I$ defined above, and 
$\W(Q)\subset \W$, with
\[
\ell(Q)\approx\ell(I)\approx \dist(Q,I)\approx \dist(Q,I^*)\approx \dist(\pom,I^*)\,,
\quad \forall\, I\in\W(Q)\,.
\] 
% \smallskip
\item For the Whitney regions, we have
\begin{equation}\label{UQprop}
\dist(U_Q,Q) \approx \dist(U_Q,\pom)\approx \diam(Q)\approx \diam(U_Q)
 \approx |U_Q|^{1/(n+1)}\,.
\end{equation}
\item The Whitney regions enjoy the bounded overlap 
property 
$$\sum_{Q\in\dd(\pom)} 1_{U_Q} (X) \lesssim 1\,,\qquad \forall\, X\in\Omega\,.$$ 
\item In the special case that $\sbf = \dd_{Q_0}$ for some $Q_0\in\dd$, we use the notation
$R_{Q_0}:= \interior\left(\cup_{Q\subset Q_0} U_{Q}\right)$, and we refer to $R_{Q_0}$ as the ``Carleson box" relative to $Q_0$.  For each $Q\in\dd$, the Carleson box $R_Q$ has the property
that  there is a ball $B_Q=B(x_Q, r_Q)$, with $r_Q\approx \ell(Q)$, such that
$R_Q\supset 2B_Q\cap\Omega$, where
$2B_Q= B(x_Q, 2r_Q)$ is the concentric double of $B_Q$. % where $B_Q=B(x_Q, r_Q)$, with $r_Q\approx \ell(Q)$.
\item By (2) (or (6)), and (4), it follows that
\[
\diam(\Omega_{\sbf}) \lesssim \diam (Q(\sbf))\,.
\]
In particular, $\diam(R_Q) \lesssim \diam(Q)$.
\item Given a surface ball $\Delta=B\cap\pom$, where $B=B(x,r)$ is centered on $\pom$, there
is a ``Carleson Tent" $T_\Delta$ associated to $\Delta$, such that 
\[
B\cap\Omega\subset 
(5/4)B\cap\Omega \subset T_{\Delta} \subset \kappa B\cap \Omega\,,
\] 
where $(5/4)B = B(x,5r/4)$, and $\kappa B = B(x,\kappa r)$ for some sufficiently large, uniform constant $\kappa$.   Moreover,
$T_\Delta$ is a 1-sided  chord arc domain.
Finally, if $B'\subset B$, and $\Delta'=B'\cap\pom$, then 
$T_{\Delta'}\subset T_{\Delta}$. 
\end{enumerate}
Furthermore, all of the implicit and explicit 
constants in properties (1), (3), (4), (5), (6), (7) and (8) 
depend only on dimension and on the 1-sided chord arc constants for $\Omega$.
\end{proposition}

\begin{remark}\label{remarknotation}
We caution the reader that our notation differs from that of \cite{HM-I}.  In the latter reference, the subdomains $\Omega_\sbf$ are denoted $\Omega_{\F,Q}$, where, translating to the notation of the present paper, we have $Q=Q(\sbf)$, the maximal cube in $\sbf$, and
$\F=\{Q_j\}_j$ denotes the collection of sub-cubes of $Q=Q(\sbf)$ that are maximal with respect to the property that $Q_j\notin\sbf$.
\end{remark}

We recall some further results from \cite{HM-I} that will be useful in the sequel.
% Again we caution the reader that our notation is not completely 
% consistent with that of \cite{HM-I}.  

\begin{proposition}\label{prop:sawtooth-contain}\cite[Proposition 6.1]{HM-I}.  
Suppose that $\Omega$ is a
1-sided chord-arc domain.
Fix $Q_0\in \dd$, let $\sbf\subset \dd_{Q_0}$ be a tree, and let
$\F=\{Q_j\}_j$ denote the collection of sub-cubes of $Q_0$ that are maximal 
with respect to the property that $Q_j\notin\sbf$. Then
\begin{equation}\label{eq5.0}
Q_0\setminus \left(\cup_\F Q_j\right)
\subset Q_0\cap\partial\Omega_{\sbf}
\subset\partial\Omega\cap\partial\Omega_{\sbf}
\subset \overline{Q_0} \setminus \left(\cup_\F \,\,{\rm int}\!\left(Q_j\right)\right)\,.
\end{equation}
\end{proposition}

\begin{remark}\label{remark2.15}
Note that \eqref{eq5.0} implies in particular that 
\begin{multline}\label{eq2.16}
Q_0\cap\partial\Omega_{\sbf} = \left[Q_0
\setminus \left(\cup_\F \,\,{\rm int}\!\left(Q_j\right)\right)\right]\cap\pom_\sbf
\\[4pt]
= \left[Q_0\setminus \left(\cup_\F Q_j\right)\right] \bigcup 
\left(\pom_\sbf\cap \left(\cup_{\F}\,  \left[Q_j\setminus \interior (Q_j)\right]\right)\right)\,.
% \\[4pt]
% =:\,\left[Q_0\setminus \left(\cup_\F Q_j\right)\right] \bigcup Z\,,
\end{multline}
% where $Z:= \pom_\sbf\cap \left(\cup_{\F}\,  
% \left[Q_j\setminus \interior (Q_j)\right]\right)$ has surface measure zero on
% $\partial\Omega_\sbf$, by the next lemma applied to the domain $\Omega_\sbf$, 
% with $\mu = \sigma_\sbf:= \mathcal{H}^n\lfloor_{\,\pom_\sbf}$.
\end{remark}

 \begin{proposition}\label{prop5.0a}\cite[Proposition 6.3]{HM-I}.
   Suppose that $\partial\Omega$ is ADR, and that $\ttm$ is a doubling measure on
$\partial\Omega$ (Definition \ref{defdouble}).
Then $\partial Q:=\overline{Q}\setminus {\rm int}(Q)$
 has $\ttm$-measure 0, for every $Q\in\dd$.  
 \end{proposition}

% \begin{remark}\label{remarkdouble}
% In particular, if $\mu$ is a doubling measure, then
% the sets in \eqref{eq5.0} have the same $\mu$ measure.
% \end{remark}

\begin{proposition}\cite[Proposition 6.7]{HM-I}\label{prop:Pj}.  Suppose that
$\Omega$ is a 1-sided chord-arc domain.
Fix $Q_0\in \dd$, let $\sbf\subset \dd_{Q_0}$ be a tree, and let
$\F=\{Q_j\}_j$ denote the collection of sub-cubes of $Q_0$ that are maximal 
with respect to the property that $Q_j\notin\sbf$.  Then for each
$Q_j\in\F$, there is a surface ball $P_j\subset \partial\Omega_{\sbf}$, satisfying
\begin{equation}\label{eqpj}
\diam(P_j)\approx \dist(P_j,Q_j)\approx \dist(P_j,\partial\Omega) \approx \ell(Q_j),
\end{equation}
where the uniform implicit constants depend only on the 1-sided chord-arc 
constants for $\Omega$ (in particular, on the parameters fixed in the construction of the Whitney regions $U_Q$, and the sawtooth regions $\Omega_{\sbf}$). 
\end{proposition}

Let $x^\star_j\in\pom_\sbf$ denote the center of the surface 
ball $P_j$ in Proposition
\ref{prop:Pj}, and let $t_j\approx \ell(Q_j)$ be its radius.

\begin{proposition}\cite[Proposition 6.12]{HM-I}. \label{prop:Pj2}
Let $\Omega, Q_0, \sbf, \F$ and $P_j$ be as in Proposition \ref{prop:Pj}.
For $Q_j\in\F$, let $MP_j:=B(x_j^\star,Mt_j)\cap \pom_\sbf$
be the $M$-fold concentric dilate of the surface ball $P_j$ on $\pom_\sbf$.
Then for $M$ large enough, % for each $Q\in \dd_{\F,Q_0}$, 
there is a surface ball 
$\Delta_{\star}:=B(x_\star, t_\star)\cap \pom_\sbf$, 
with $t_\star\approx\ell(Q_0)$, and $x_\star\in \partial\Omega_{\sbf}$, 
satisfying
\begin{equation}\label{eqDstar}
\Delta_{\star}\subset \big(Q_0\cap\partial\Omega_{\sbf}\big)
\bigcup\left(\cup_{Q_j\in\F} MP_j\right)\,,
\end{equation}
where $M$ and the various implicit constants 
depend only on the 1-sided chord-arc 
constants for $\Omega$ (in particular, on the parameters 
fixed in the construction of the Whitney regions $U_Q$, and the 
sawtooth regions $\Omega_{\sbf}$).
\end{proposition}

We now observe that, in light of Remark \ref{remark2.15}, we may deduce the following
subtle (but eventually useful) self-improvement of Proposition \ref{prop:Pj2}.

\begin{corollary}\label{corPj}
Let $\Omega, Q_0, \sbf, \F$, $\Delta_{\star}$, $P_j$ and $MP_j$ be as in Proposition 
\ref{prop:Pj2}.  Then for a possibly larger choice of
$M$ (but still uniform, and depending only on allowable parameters), 
we have the following improvement of \eqref{eqDstar}:
\begin{equation}\label{eqDstar2}
\Delta_{\star}\subset \big(Q_0\setminus \left(\cup_\F Q_j\right)\big)
\bigcup\left(\cup_{Q_j\in\F} MP_j\right)\,.
\end{equation}
\end{corollary}
\begin{proof}
Using \eqref{eq2.16}, we need only observe that for each $j$,
\[
\pom_\sbf\cap  \left[Q_j\setminus \interior (Q_j)\right] 
\subset \pom_\sbf\cap Q_j\subset MP_j\,,
\]
by \eqref{eqpj} and the definition of $MP_j$, provided that we fix $M$ large enough.
\end{proof}

% \begin{lemma}\label{lemmaSN} {\Rd $S<N$ estimate, to be written}
% \end{lemma}

\begin{lemma}[Local Hardy Inequality]\label{localHardy}
Let $\Omega\subset\re^{n+1}$ be a 1-sided $\mathrm{CAD}$, and consider any ball
$B=B(x,r)$, with $x\in \pom$.  Set $\Omega_{r}:= \Omega\cap B(x,r)$, and 
$\Omega_{M r}:= \Omega\cap B(x,M r)$, for $M>0$.  Then there is a positive 
constant $\kappa_0<\infty$, depending only on the 1-sided CAD constants, 
such that for any 
$u\in W^{1,2}(\Omega_{2\kappa_0 r})$ % \cap C(\overline{\Omega_{\kappa_0 r}})$,
with vanishing trace on $2\kappa_0\Delta:= B(x,2\kappa_0 r)\cap\pom$,
\begin{equation}\label{eqHardyloc}
\iint_{\Omega_r} \left(\frac{u(Y)}{\delta(Y)}\right)^2 dY\, \lesssim\,
\iint_{\Omega_{\kappa_0r}}  |\nabla u(Y)|^2 dY\,,
\end{equation}
where the implicit constant depends only on dimension and the 1-sided CAD constants.
\end{lemma}

\begin{remark}\label{remarkHardy}
A global version of this result, in which $\Omega_r$ and $\Omega_{\kappa_0r}$ 
are replaced by $\Omega$, and $u\in W^{1,2}_0(\Omega)$,
appears in \cite{A}, and for that result, the corkscrew and Harnack chain
conditions are not required, rather only ``uniform thickness" of the complement of $\Omega$
(in particular, Ahlfors-David regularity of the boundary suffices).  A more general result appears in \cite{L}.  Lemma \ref{localHardy} will be a routine consequence of these global results.
\end{remark}

\begin{proof}[Proof of Lemma \ref{localHardy}]
 Let $\vp$ be a smooth bump function adapted to the ball $B=B(x,r)$,
i.e., $\vp\in C_0^\infty(2B)$, $\vp\equiv 1$ on $B$, $0\leq \vp\leq 1$, 
and $|\nabla \vp|\lesssim r^{-1}$, and similarly, let $\eta$ be
a smooth bump function adapted to the ball $B(x,\kappa_0r)$, 
with $\eta\equiv 1$ on $B(x,\kappa_0r)$, and $\supp(\eta) \subset B(x,2\kappa_0r)$.
Then $u=u\eta$ on $\Omega_{\kappa_0 r}$, so replacing
$u$ by $u\eta$, we may assume that
$u\in W^{1,2}_0(\Omega)$, and hence by density, 
we may suppose that $u\in C_0^\infty(\Omega)$.

Observe further that in particular, $u\vp\in W^{1,2}_0(\Omega)$.
Thus, by the global result of \cite{A} (or \cite{L}), we see that
\begin{multline*}
\iint_{\Omega_r} \left(\frac{u(Y)}{\delta(Y)}\right)^2 dY\,\leq\,
\iint_{\Omega} \left(\frac{u(Y)\vp(Y)}{\delta(Y)}\right)^2 dY
\\[4pt]
\lesssim\,
\iint_{\Omega} |\nabla u|^2\vp^2 dY\,+\, \iint_{\Omega} |\nabla \vp|^2 u^2 dY\,=:\, I+II\,.
\end{multline*}
Term $I$ clearly satisfies the desired bound in \eqref{eqHardyloc}.  

To handle term $II$, 
first note that by the properties of $\vp$,
\[
II \lesssim r^{-2} \iint_{\Omega_{2r}}  u^2 dY\,.
\]
We choose a collection $\mathcal{C}\subset \dd(\pom)$,
of uniformly bounded cardinality depending only on the 1-sided CAD constants, 
with $\ell(Q)\approx r$ for all $Q\in\mathcal{C}$, and such that
\[
\Omega_{2r}\subset \bigcup_{Q\in\mathcal{C}} R_Q\,,
\]
where $R_Q$ is the ``Carleson box" associated to $Q$, as in 
Proposition \ref{sawtoothprop} (6).  
Our goal is now to prove that 
\begin{equation}\label{eqHardyloc2}
r^{-2} \iint_{R_Q}  u^2 dY\,\lesssim\, 
\iint_{R^*_Q}|\nabla u|^2 dY\,,
\end{equation}
uniformly for each $Q\in\mathcal{C}$, where $R^*_Q =
\interior\left(\cup_{Q'\subset Q} U^*_{Q'}\right)$, and $U^*_{Q'}$ is a fattened version
of the Whitney region $U_{Q'}$,
which still retains the same Whitney properties:  in particular,
$\diam(U_{Q'}^*)\approx \dist(U_{Q'}^*,\pom)\approx \ell(Q')$.  Indeed,
observe that by construction, there is a uniform constant
$\kappa_0$ such that
$R^*_Q\subset \Omega_{\kappa_0r}$ for each $Q\in\mathcal{C}$, thus, 
\eqref{eqHardyloc2} implies the desired bound in \eqref{eqHardyloc}.

We now turn to the proof of \eqref{eqHardyloc2}.  
For any $Q\in \dd$, and any  $f\in L^1(U_Q^*)$,
let $[f]_Q$ and $[f]^*_Q$ denote, respectively, the mean value of $f$ on $U_Q$, and on the fattened Whitney region $U_Q^*$, i.e.,
\[
[f]_Q:= |U_Q|^{-1}\iint_{U_Q} f \,dY\,,\qquad [f]^*_Q:= |U^*_Q|^{-1}\iint_{U^*_Q} f \,dY
\]
Fix an arbitrary
$Q_0\in\mathcal{C}$, and observe that by definition of the Carleson regions, 
\[
r^{-2} \iint_{R_{Q_0}}  u^2 dY\,\leq\, r^{-2}\sum_{Q\subset Q_0} \iint_{U_Q} u^2 dY
\lesssim \, III + IV\,,
\]
where
\[
III\,= \, r^{-2}\sum_{Q\subset Q_0} \iint_{U_Q} \big|u-[u]_Q\big|^2 dY\,,
\]
and
\[
IV\,=\,
r^{-2}\sum_{Q\subset Q_0} 
\iint_{U_Q} \big|\fint_Q\sum_{Q'\subset Q}
1_{Q'}(x)\left([u]_{Q'} - [u]_{Q_{child}'}\right) d\sigma(x)\big|^2 dY\,,
\]
where for $Q'\subset Q$ with $Q'\ni x$, we let $Q'_{child}$ denote the child of
$Q'$ that contains $x$, and where we have replaced $[u]_{Q} = [u]_{Q}-u(x)$ by a telescoping sum,
using the fact that we have reduced to the case that $u \in C_0^\infty(\Omega)$.

By Poincar\'e's inequality (see, e.g., \cite[Section 4]{HM-I})
\[
\iint_{U_Q} \big|u-[u]_{Q}\big|^2 dY \, \lesssim \, \ell(Q)^2
\iint_{U^*_Q} |\nabla u|^2 dY\,,
\]
where as above $U^*_Q$ is a fattened version of $U_Q$.  Since the fattened Whitney regions retain the bounded overlap property, we easily obtain \eqref{eqHardyloc2} for term $III$.

Similarly (again see \cite[Section 4]{HM-I}),
\[
\big|[u]_{Q'} - [u]_{Q_{child}'}\big| \,\lesssim \,\ell(Q') [|\nabla u|]_{Q'}^*\,,
\]
and therefore
\begin{multline*} % \label{ivbound}
IV\, \lesssim \,r^{-2}\sum_{Q\subset Q_0} |U_Q| \,\left(
\sum_{Q'\subset Q}\frac{\sigma(Q')}{\sigma(Q)}
\,\ell(Q')\,|U^*_{Q'}|^{-1} \iint_{U^*_{Q'}} |\nabla u| \,dY\right)^2
\\[4pt]
\lesssim \, r^{-2}\sum_{Q\subset Q_0}\ell(Q)^{1-n} 
\, \left(\iint_{R^*_Q}|\nabla u|\, dY\right)^2\,,
\end{multline*}
where we have used that $|U_Q|\approx |U^*_Q|\approx \ell(Q)^{n+1}$, and $\sigma(Q)\approx \ell(Q)^n$, for each $Q\in \dd(\pom)$.
 Moreover $|R_Q^*| \approx \ell(Q)^{n+1}$, so by Cauchy-Schwarz, we see in turn that
\[
IV \,\lesssim\, r^{-2}\sum_{Q\subset Q_0}\ell(Q)^{2} 
\, \iint_{R^*_Q}|\nabla u|^2 dY\,
\approx\,
\sum_{k=0}^\infty2^{-2k}\sum_{Q\in \dd_k(Q_0)}
\, \iint_{R^*_Q}|\nabla u|^2 dY\,,
\] 
where $\dd_k(Q_0)$ is the collection of subcubes of $Q_0$ with 
$\ell(Q)=2^{-k}\ell(Q_0)\approx 2^{-k} r$.
Note that by definition, $R_Q^*$ is contained in a union of fat Whitney regions $U^*_{Q'}$ with
$Q'\subset Q$, thus, for distinct $Q_1,Q_2\in \dd_k(Q_0)$, 
any pair of corresponding sub-cubes $Q_1'\subset Q_1$ and $Q_2'\subset Q_2$ are disjoint.  
Since the regions $U^*_{Q'}$ have bounded overlaps, we find that
\[
\sum_{Q\in \dd_k(Q_0)}
\, \iint_{R^*_Q}|\nabla u|^2 dY \lesssim \iint_{R^*_{Q_0}}|\nabla u|^2 dY\,,
\] 
uniformly for each $k$, so we may sum the geometric series to obtain the desired bound for $IV$.
\end{proof}

% \section{Elliptic measure: existence and 
% a criterion for quantitative absolute continuity}\label{Shm}

\section{Preliminaries}
\label{Ssmall1}

In this section,
for operators of the form $L=-\dv A\nabla +\Bb\cdot\nabla=L_0 +\Bb\cdot\nabla$, 
we verify that certain standard estimates 
for solutions of the homogeneous second order equation 
$L_0u=0$, continue to hold for solutions of $Lu=0$, assuming that $\Bb$ satisfies
\eqref{driftsize} (but not necessarily the Carleson measure condition \eqref{driftCarleson}).
For some (but not all) of these results, we shall require the ``small constant" 
version of \eqref{driftsize}, namely
 \begin{equation} \label{smallptwise}
|\Bb(X)|\,\leq \, \frac{\eps_0}{\delta(X)}\,,\qquad \text{a.e. } X\in \Omega\,,
 \end{equation}
where $\eps_0$ will be 
sufficiently small, depending only on dimension, the 
ellipticity parameters $\lambda$ and $\Lambda$, and the 
1-sided chord-arc constants for $\Omega$.

In addition, we obtain solvability of the Dirichlet problem with continuous data, and 
consequent existence of elliptic measure, under certain conditions. % and we give 
% a criterion for quantitative absolute continuity of elliptic measure with respect to 
% surface measure.  

Some of the results in this section hold for both $L=-\dv A\nabla + \Bb\cdot\nabla$ 
and its adjoint $L^*$, and for those results, we shall consider more generally
the operator
\begin{equation}\label{Lboth}
\mathcal{L} u:= -\dv(A\nabla u + \Bb_1 u) +\Bb_2 \cdot\nabla u\,.
\end{equation}
Taking $\Bb_1=0,\, \Bb_2=\Bb$, we recover the operator $L$, and 
taking $\Bb_2=0,\,\Bb_1=\Bb$ (and replacing
$A$ by its transpose $A^{\tt \!T}$), we recover $L^*$. 

We begin with some standard {\em a priori} estimates.

\subsection{A priori estimates}\label{apriori}

\begin{lemma}[De Giorgi-Nash-Moser estimates and Harnack's inequality]\label{Moser}
Let $\Omega\subset\re^{n+1}$ be an open set, 
let $\LL$ be defined as in \eqref{Lboth}, and suppose that \eqref{driftsize} holds for
each of $\Bb_1$ and $\Bb_2$.
Consider a ball $B=B(X,r)$ such that $3B \subset \Omega$.  Let
$u\in W^{1,2}(2B)$ be a weak solution of $\LL u=0$ in $2B$.  Then
\begin{equation}\label{eqMoser}
\sup_B |u|^2\,\lesssim \, |2B|^{-1}\iint_{2B} |u|^2 dY\,,
\end{equation}
and 
\begin{equation}\label{eqHolder}
|u(Y)-u(Z)|\lesssim |Y-Z|^\alpha \sup_{2B} |u|\,,\qquad \forall\, Y,Z\in B\,.
\end{equation}
If in addition, $u$ is non-negative in $2B$, then
\begin{equation}\label{eqHarnack}
\sup_{B} u \,\lesssim\, \inf_B u\,.
\end{equation}
Moreover, the implicit constants and the H\"older exponent $\alpha$ depend only on dimension, ellipticity of $A$, and the constant in \eqref{driftsize}.
\end{lemma}

\begin{proof} 
Since $3B\subset \Omega$,
we have $r\leq \delta(Y)$ in $2B$, and therefore
$|\Bb_1(Y)| + |\Bb_2(Y)|  \lesssim r^{-1}$, for a.e. $Y\in 2B$.
By scale invariance, we may assume $r=1$, and thus we have reduced to the case
that $\|\Bb\|_{L^\infty(2B)} \lesssim 1$.  In this form, the results stated may all be found in 
\cite[Chapter 8]{GT}.
\end{proof}

\begin{remark}
Note that the smallness condition \eqref{smallptwise} is not needed in Lemma \ref{Moser},
nor will it be needed in the next result,
Lemma \ref{Caccinterior}.
\end{remark}

\begin{lemma}[interior Caccioppoli's inequality]\label{Caccinterior}
Let $\Omega\subset\re^{n+1}$ be an open set, 
let $\LL$ be defined as in \eqref{Lboth}, and suppose that \eqref{driftsize} holds for
each of $\Bb_1$ and $\Bb_2$.
Consider a ball $B=B(X,r)$ such that $3B \subset \Omega$.  Let
$u$ be a $W^{1,2}$ solution of $\LL u=0$ in $2B$.  Then
\[
\iint_{B}|\nabla u|^2 dY \,\lesssim\, r^{-2} \iint_{2B} u^2 dY\,,
\]
where the implicit constant depends only on ellipticity of $A$ and the constant in 
\eqref{driftsize}.
\end{lemma}

\begin{proof}
The proof is standard.  
Let $\vp$ be a smooth bump function adapted to the ball $B=B(X,r)$,
i.e., $\vp\in C_0^\infty(2B)$, $\vp\equiv 1$ on $B$, $0\leq \vp\leq 1$, 
and $|\nabla \vp|\lesssim r^{-1}$, so that 
\begin{multline*}
\iint_{B}|\nabla u|^2 dY \,\leq \,\iint_{\Omega}|\nabla u|^2\vp^2 dY\,
\lesssim \, \iint_{\Omega} A\nabla u \cdot \nabla u \,\vp^2dY
\\[4pt]
= \,\iint_{\Omega} A\nabla u \cdot \nabla (u \vp^2)dY\,-\,
2\iint_{\Omega} A\nabla u\cdot \nabla \vp\, u \vp\, dY\, =:\, I + II\,.
\end{multline*}
As usual, for a small number $\gamma$ at our disposal,
\[
|II| \,\lesssim\, \gamma \!\iint_{\Omega} |\nabla u|^2 \vp^2 dY \,+\,
\frac1{\gamma} \iint_{\Omega} |\nabla \vp|^2 u^2 dY 
% = {\tt small}(\gamma) + O\!\left(\frac1{r^2}\iint_{2B} |\nabla  u|^2 dY\right).
\]
Choosing $\gamma>0$ small enough, we may hide the small term.  With $\gamma$ now fixed, 
the second term satisfies the desired bound, since $|\nabla \vp|^2 \lesssim r^{-2}$.

Since $3B\subset \Omega$, we have
$r\leq \delta(Y)$ for all $Y\in 2B$, hence by \eqref{driftsize},
\begin{equation*} % \label{Bbbounds}
|\Bb_1(Y)| + |\Bb_2(Y)| \, \lesssim \, r^{-1}\,,\qquad \text{a.e. } Y\in 2B\,.
\end{equation*}
Thus, since $\LL u=0$ in $2B$,
\begin{multline*}
|I| = \big|\iint_\Omega \Bb_2\cdot\nabla u\, u\vp^2 dY \,
+\, \iint_\Omega u \,\Bb_1\cdot\nabla(u\vp^2) dY\,\big|
\\[4pt]
\lesssim \frac1{r} \iint_\Omega |\nabla u| \,|u|\vp^2 dY\,
\lesssim \gamma \!\iint_{\Omega} |\nabla u|^2 \vp^2 dY \,+\,
\frac1{\gamma}\,\frac1{r^2}\iint_{2B} u^2 dY\,,
\end{multline*}
where again $\gamma>0$ is at our disposal.  Again we may hide the small term with
$\gamma$ fixed small enough.
\end{proof}

\begin{lemma}[Caccioppoli's inequality at the boundary]\label{Cacc}
Let $\Omega\subset\re^{n+1}$ be a 1-sided $\mathrm{CAD}$, and consider any ball
$B=B(x,r)$, with $x\in \pom$. 
Let $\LL$ be defined as in \eqref{Lboth}, and suppose that \eqref{smallptwise} holds for
each of $\Bb_1$ and $\Bb_2$, with 
$\eps_0$ sufficiently small, depending only on $n$, ADR, and the ellipticity of $A$.  
Set $\Omega_{r}:= \Omega\cap B(x,r)$, and
$\Omega_{2r}:= \Omega\cap B(x,2r)$, and let $u\in W^{1,2}(\Omega_{2r})$
be a solution of $\LL u=0$ in 
$\Omega_{2r}$, with vanishing trace on $2\Delta:= B(x,2r) \cap\pom$.  Then
\begin{equation}\label{eqCacc}
\iint_{\Omega_r} |\nabla u|^2 dY\,\lesssim \,r^{-2} \iint_{\Omega_{2r}}u^2 dY\,.
\end{equation}
where the implicit constants depend only on the allowable parameters.
\end{lemma}
\begin{proof}
The proof is again standard.   
As above, choose $\vp\in C_0^\infty(2B)$, with $\vp\equiv 1$ on $B$, $0\leq \vp\leq 1$, 
and $|\nabla \vp|\lesssim r^{-1}$.  % Observe that $u\vp^2\in W^{1,2}_0(\Omega_{2r})$.
We then have
\begin{multline*}
\iint_{\Omega_r} |\nabla u|^2 dY\,\leq\,
\iint_{\Omega} |\nabla u|^2 \vp^2 dY\,
 \lesssim \,\iint_{\Omega} A\nabla u\cdot \nabla u \,\vp^2 dY
\\[4pt] 
=\,  \iint_{\Omega} A\nabla u\cdot \nabla \!\left(u \vp^2\right) dY\,-\,
2\iint_{\Omega} A\nabla u\cdot \nabla \vp\, u \vp\, dY\,=:\, I + II\,.
\end{multline*}
As usual, for a small number $\gamma$ at our disposal,
\[
|II| \lesssim \gamma \iint_{\Omega} |\nabla u|^2 \vp^2 dY \,+\,
\frac1{\gamma}\iint_{\Omega} |\nabla \vp|^2 u^2 dY\,.
\]
Choosing $\gamma>0$ small enough, we may hide the small term.  With $\gamma$ now fixed, 
the second term satisfies the desired bound in \eqref{eqCacc}.

Since $\LL u=0$ in $\Omega_{2r}$, and $u\vp^2\in W^{1,2}_0(\Omega_{2r})$, 
using \eqref{smallptwise}, we also have
\begin{multline*}
|I| = \big|\!\iint_\Omega \Bb_2\cdot\nabla u\, u\vp^2 dY 
+ \iint_\Omega u \,\Bb_1\cdot\nabla(u\vp^2) dY\,\big|
\lesssim  \eps_0 \iint_\Omega \left(|\nabla u|\vp+|\nabla\vp|\,|u|\right)\, 
\frac{|u|\vp}{\delta} \, dY
\\[4pt]
\lesssim \eps_0\left( \iint_\Omega |\nabla u|^2\vp^2 dY+
\iint_\Omega |\nabla \vp|^2 u^2 dY+
\iint_\Omega \left(\frac{u\vp}{\delta}\right)^2 dY\right) 
\\[4pt] =: \,\eps_0\big(I_1+I_2+I_3\big)\,.
\end{multline*}
It is enough to deal with $I_1$ and $I_2$, since,
by the global Hardy inequality in \cite{A} and \cite{L}, we have
\[
I_3 \lesssim \int_{\Omega}\big(|\nabla u|\vp \,+\,|\nabla \vp|\,|u| \big)^2dY\,
\approx\,I_1+I_2\,.
\]
In turn, as was the case for term $II$ treated above, $\eps_0I_1$ may be hidden, provided $\eps_0$ is small enough, and term $I_2$ yields the desired estimate in \eqref{eqCacc}.
\end{proof}

\begin{lemma}[H\"older continuity at the boundary and Carleson's estimate]\label{proppde}
Let $\Omega\subset\re^{n+1}$ be a 1-sided $\mathrm{CAD}$, and consider any ball
$B=B(x_0,r)$, with $x_0\in \pom$.  Set $\Omega_{r}:= \Omega\cap B(x_0,r)$, and
$\Omega_{2r}:= \Omega\cap B(x_0,2r)$.
Let $\LL$ be defined as in \eqref{Lboth}, and suppose that \eqref{smallptwise} holds for
each of $\Bb_1$ and $\Bb_2$, with 
$\eps_0$ sufficiently small.  If $u\in W^{1,2}(\Omega_{2r})$ is a non-negative 
solution to $\LL u=0$ in $\Omega_{2r}$, 
with vanishing trace on $2\Delta:= B(x_0,2r) \cap\pom$, 
then for some $\alpha>0$,
\begin{equation}\label{holderatbdry}
u(X) \, \lesssim \, \left(\frac{\delta(X)}{r}\right)^\alpha \sup_{ \Omega_{2r}} u\,,
\qquad \forall \, X\in \Omega_r\,,
\end{equation}
and if $X_B$ is a corkscrew point relative to $B=B(x,r)$, then
% if $u$ in addition is non-negative in $\Omega_{2r}$, then
\begin{equation}\label{carleson}
\sup_{Y\in \Omega_{r}} u(Y) \lesssim u(X_B)
\end{equation}
where the implicit constants and the exponent $\alpha$ depend only on ellipticity, 
dimension, and the 1-sided chord arc constants.
\end{lemma}

\begin{proof}
It suffices to prove \eqref{holderatbdry}.  Indeed, the fact that \eqref{holderatbdry} 
holds for every ball $B$ centered on $\pom$ yields \eqref{carleson}, by the proof in
\cite[Theorem 1.1]{CFMS}\footnote{see also \cite[Lemma 4.4]{JK}, 
where the same argument is used}, 
which carries over to the present setting.  The only 
ingredients required are Harnack's inequality, the Corkscrew and Harnack Chain conditions,
and \eqref{holderatbdry}, all of which are available to us.  We omit the details.

Turning to \eqref{holderatbdry}, we shall follow the argument in 
\cite[pp.~12-14]{HL}.    Let $\kappa$ and $\kappa_0$ be the constants  
in Proposition \ref{sawtoothprop} (8), and in Lemma \ref{localHardy}, respectively.
Fix $M\gg \kappa \kappa_0=:\kappa_1$, and
let $X\in \Omega_r$, with $\delta(X)\leq r/M$ (if $\delta(X)> r/M$, 
then there is nothing to prove).
Choose $\hat{x}\in\pom$ such that
$|X-\hat{x}|= \delta(X)$, and set $\rho=\delta(X)$. 
Then by \eqref{eqMoser} and Lemma \ref{localHardy}, we have
\begin{equation}\label{equrhobound}
u^2(X)\lesssim \rho^{-n-1} \iint_{B(X,\,\rho/4)} u^2 dY \lesssim
\rho^{-n-1} \iint_{\Omega_{\rho}(\hat{x})} u^2 dY
\lesssim  \rho^{1-n} \iint_{\Omega_{\kappa_0\rho}(\hat{x})} |\nabla u|^2 dY\,,
\end{equation}
where $\Omega_{\rho}(\hat{x}):= B(\hat{x},\rho)\cap\Omega$.  

For $0<t\leq r/M$, 
set $\Delta_t:= B(\hat{x},t)\cap\pom$, 
and let $T(t):=T_{\Delta_t}$ be the Carleson tent
associated to $\Delta_t$ as in Proposition \ref{sawtoothprop} (8), so that 
(with $M$ chosen large as above)
\begin{equation}\label{OtTt}
\Omega_{t}(\hat{x})\subset
\Omega_{5t/4}(\hat{x}) \subset T(t) \subset \Omega_{\kappa t}(\hat{x}) \subset
\Omega_{2\kappa t}(\hat{x})\subset
 \Omega_{3r/2}\,.
\end{equation}
By Proposition \ref{sawtoothprop} (8), $T(t)$ is a 1-sided 
CAD\footnote{Indeed, it is for this reason that we have gone to the trouble of replacing
$\Omega_t(\hat{x})$ by $T(t)$.} for each $t$.
For $t$ momentarily fixed,
we may solve the 
Dirichlet problem in $T(t/\kappa_1)$
 for $L_0u_0=0$, with 
 data $u_0=u$ (in the trace sense) on $\partial T(t/\kappa_1)$.  
 Set $w:=u-u_0$, so that
$w\in W^{1,2}_0(T(t/\kappa_1))$.  Using that $L_0u_0=0$ and $\LL u=0$, we see that
\begin{multline*}
\iint_{T(t/\kappa_1)}|\nabla w|^2 dY \,\approx\,
\iint_{T(t/\kappa_1)} A\nabla w\cdot \nabla w \,dY\,
= \iint_{T(t/\kappa_1)} A\nabla u\cdot \nabla w \,dY
\\[4pt]
 =-\iint_{T(t/\kappa_1)}  \Bb_2\cdot\nabla u\, w\, dY 
- \iint_{T(t/\kappa_1)}  u \,\Bb_1\cdot\nabla w\, dY
\\[4pt]
\lesssim \, \eps_0 \iint_{T(t/\kappa_1)} \left(|\nabla u| \,\frac{w}{\delta'} + |\nabla w|\, \frac{u}{\delta}\right) dY\\[4pt]
\lesssim \,\eps_0 \left(\iint_{T(t)} 
 |\nabla u|^2 dY+\iint_{T(t/\kappa_1)}|\nabla w|^2 dY\right)\,, 
\end{multline*}
where $\delta'(Y):= \dist(Y,\partial T(t))\leq \delta(Y)$, and we have used
\eqref{smallptwise} and either the 
global Hardy inequality of \cite{A} (or \cite{L}) in the domain $T(t/\kappa_1)$
(to handle the term $w/\delta'$), 
or else \eqref{OtTt}, the local Hardy inequality 
Lemma \ref{localHardy}, and then \eqref{OtTt} again, to handle the term
$u/\delta$.  To guide the reader through the latter step, we observe that
\eqref{OtTt} yields that
$T(t/\kappa_1) \subset \Omega_{t/\kappa_0}(\hat{x})$
(since $\kappa_1:=\kappa\kappa_0$), whence application of  
Lemma \ref{localHardy} yields an integral over $\Omega_t(\hat{x})\subset T(t)$.

Hiding the small term involving $w$, we obtain
\begin{equation}\label{wsmallbound}
\iint_{T(t/\kappa_1)}|\nabla w|^2 dY 
\,\lesssim \,\eps_0 \iint_{T(t)} |\nabla u|^2 dY\,.
\end{equation}

Given $s>0$ and $f\in W^{1,2}(\Omega_{2r})$, set
\[
\Phi(f,s)\,:=\, s^{1-n} \iint_{T(s)}|\nabla f|^2 dY\,,
\]
and observe that for $s<t/(4\kappa_1^2)$,
letting $X_t$ denote a corkscrew point
relative to the ball $B(\hat{x},t/(\kappa_0\kappa_1))$, and
using \eqref{OtTt} and the case $\Bb_1=0=\Bb_2$ of Lemma \ref{Cacc}, we have
\begin{multline*}
\Phi(u_0,s)\,\leq\, s^{1-n} \iint_{\Omega_{\kappa s}(\hat{x})}|\nabla u_0|^2 \,dY
\,\lesssim \,s^{-n-1} \iint_{\Omega_{2\kappa s}(\hat{x})}u_0^2 \,dY
\\[4pt]
\lesssim \, \left(\frac{s}{t}\right)^{2\alpha_0}\sup_{\Omega_{t/(\kappa_0\kappa_1)}(\hat{x})} u^2_0 \,\lesssim\, 
\left(\frac{s}{t}\right)^{2\alpha_0} u_0^2(X_t)
% \\[4pt]
\, \lesssim \, \left(\frac{s}{t}\right)^{2\alpha_0} t^{-n-1} \iint_{\Omega_{t/(\kappa_0\kappa_1)}(\hat{x})}u_0^2\, dY
\\[4pt]
\lesssim\, \left(\frac{s}{t}\right)^{2\alpha_0} t^{1-n} \iint_{\Omega_{t/\kappa_1}(\hat{x})}|\nabla u_0|^2\, dY
\,\lesssim \, \left(\frac{s}{t}\right)^{2\alpha_0}\Phi(u_0,t/\kappa_1)\,,
\end{multline*}
where in the second line of the display, we have used that Lemma \ref{proppde} 
(with some H\"older exponent $\alpha_0>0$) and 
\eqref{eqMoser} hold for $u_0$, and in the last line we have used 
Lemma \ref{localHardy}, applied in the 1-sided CAD $T(t/\kappa_1)$.

Combining the last estimate with \eqref{wsmallbound}, we deduce that for $s<t/(4\kappa_1^2)$,
\begin{multline*}
\Phi(u,s)\, \lesssim\, \Phi(u_0,s) \,+\, \Phi(w,s)
\,\lesssim \,\left(\frac{s}{t}\right)^{2\alpha_0}\Phi(u_0,t/\kappa_1) 
\,+\, \left(\frac{t}{s}\right)^{n-1}\Phi(w,t/\kappa_1)
\\[4pt]
\lesssim\, \left[\left(\frac{s}{t}\right)^{2\alpha_0} \,+\, \eps_0\left(\frac{t}{s}\right)^{n-1}\right]\Phi(u,t)\,.
\end{multline*}
Set $\theta= s/t$, so that first choosing $\theta<1/(4\kappa_1^2)$ sufficiently small, and then choosing $\eps_0$ depending on $\theta$, we see that for all $t\le r/M$
\[
\Phi(u,\theta t) \leq \frac12 \Phi(u, t)\,.
\]
Iterating the latter estimate starting with $t=r/M$, and
using \eqref{equrhobound} and \eqref{OtTt}, we find that for some $\alpha>0$,
\[
u^2(X)\lesssim \Phi(u,\kappa_0 \rho) \lesssim \left(\frac{\rho}{r}\right)^{2\alpha} \Phi(u, r/M)\,.
\]
% By \eqref{equrhobound} and \eqref{OtTt}, we have $u^2(X) \lesssim \Phi(u,\kappa_0 \rho)$,
Moreover, by \eqref{OtTt} and Lemma \ref{Cacc},
\[
 \Phi(u, r/M) \lesssim r^{-n-1}\iint_{\Omega_{2\kappa r/M}(\hat{x})} u^2\, dY 
 \lesssim \sup_{\Omega_{2r}} u^2\,.
\]
Since we had set $\rho =\delta(X)$, the lemma now follows.
\end{proof}

Our next task is to discuss existence of solutions.

\subsection{Existence of $Y^{1,2}$ solutions 
with data in Lip$_c(\pom)$, in the small constant case}\label{ssy12}

In this subsection, we construct $Y^{1,2}$
solutions to the Dirichlet problem with data in Lip$_c(\pom)$, in the small
constant case.  Specifically, we consider the operator
\begin{equation}\label{LdefShmsection}
L=-\dv A\nabla +\Bb\cdot\nabla\,,
\end{equation}
in the case that the constant $\eps_0$ in \eqref{smallptwise} % \eqref{driftsize} 
is small enough.  
% i.e., $|\Bb(X)|\leq \eps_0/\delta(X)$, where $\eps_0$ is sufficiently small.
The first step is to construct solutions to the Poisson-Dirichlet problem, with vanishing trace:
\begin{equation}\label{Poisson-DirichletY12}
(\mbox{PD})\,\,\left\{
\begin{array}{l}
L w=F\in Y^{-1,2}(\Omega),\\[5pt]
% N_*(u)\in L^p(\rn),\\[6pt]
w \in \,Y^{1,2}_0(\Omega),
\end{array}\right.
\end{equation}
We state this construction as a lemma, with estimates, for future reference in the sequel.

\begin{lemma}\label{lemmaPDbounds} Let $L$ be as in \eqref{LdefShmsection}, with
$|\Bb(X)|\leq \eps_0/\delta(X)$, for a.e.~$x\in\Omega$.  
If $\eps_0$ is small enough, depending only on the allowable 
parameters, then given $F\in Y^{-1,2}(\Omega)$, there is a unique solution 
$w\in Y^{1,2}_0(\Omega)$ to the problem {\em (PD)}.  Moreover,
\begin{equation}\label{wfbound3}
\|w\|_{L^{2^*}(\Omega)}\,\leq\,
\|w\|_{Y^{1,2}_0(\Omega)} \,\lesssim \,\|F\|_{Y^{-1,2}(\Omega)}\lesssim\, \|F\|_{L^{2_*}(\Omega)}\,,
\end{equation}
where the implicit constants depend only on allowable parameters.
\end{lemma}

\begin{proof}
The technique is standard.
Indeed, if $\eps_0$ is small enough, depending only on dimension, ellipticity of $A$,
and the 1-sided CAD constants, then
the bilinear form
\[
\mathfrak{B}(w,v) := \iint_\Omega \big(A \nabla w \cdot\nabla v
+ \Bb \cdot \nabla w \, v\big) dY
\]
is bounded and coercive on $Y^{1,2}_0(\Omega)$ (here, we use the global Hardy inequality
of \cite{A} or \cite{L} to deal with a term involving $v/\delta$).  
We omit the routine details.  Thus by Lax-Milgram, given $ F\in Y^{-1,2}(\Omega)$, % L^{2_*}(\Omega)$, 
we see that there is a unique $w\in Y^{1,2}_0(\Omega)$ such that
\begin{equation}\label{LM}
\mathfrak{B}(w,v) = % \iint_\Omega f(Y)\, v(Y)\, dY
\langle F,v\rangle \,,\qquad \forall \, v\in Y^{1,2}_0(\Omega)\,,
\end{equation}
where  $\langle \cdot,\cdot\rangle$ denotes the duality 
pairing of $Y^{1,2}_0(\Omega)$ and its dual
$Y^{-1,2}(\Omega)$.
In particular, if we set $v=w$, then
\begin{equation*}% \label{wfbound}
\|\nabla w\|^2_{L^2(\Omega)} \approx 
\iint_\Omega A\nabla w\cdot\nabla w\,dY = 
\mathfrak{B}(w,w) +O\left(\eps_0\|\nabla w\|^2_{L^2(\Omega)}\right)\,,
\end{equation*}
where in the last step we have used \eqref{smallptwise} and the global Hardy inequality.
For $\eps_0$ chosen small enough, we may hide the small term to obtain
\begin{equation*}% \label{wfbound2}
\|w\|^2_{Y^{1,2}_0(\Omega)}\,\lesssim\,\|\nabla w\|^2_{L^2(\Omega)} \,\lesssim \,\mathfrak{B}(w,w)\, \lesssim \,\|F\|_{Y^{-1,2}(\Omega)}
\,\|w\|_{Y^{1,2}_0(\Omega)}\,,
\end{equation*}
by \eqref{LM} and Sobolev embedding.  Consequently, again using Sobolev embedding, 
we obtain \eqref{wfbound3}.
\end{proof}

We now give the usual construction  (see, e.g., \cite[page 5]{K}) of the solution to the Dirichlet problem with compactly supported Lipschitz data.
Let $f\in \text{Lip}_c(\ree)$.  Then $Lf\in Y^{-1,2}(\Omega)$ in the weak sense 
(again, this observation uses the Hardy inequality of \cite{A}), so there is a unique
$w\in Y^{1,2}_0(\Omega)$ such that \eqref{LM} holds with $F=Lf$, i.e.,
$Lw = Lf$ in the weak sense.  Setting $u=f-w$, we see that
$u\in Y^{1,2}(\Omega)$, $Lu=0$ in $\Omega$,
and that $u|_{\pom} =f|_{\pom}$ in the trace sense, and moreover, 
given any $f_0\in \text{Lip}_c(\pom)$, the solution $u$ does not depend
on the particular $\text{Lip}_c$ extension of $f_0$ to $\ree$.

\subsection{Weak maximum principle and continuity of solutions with 
Lip$_c(\pom)$ data}\label{WMP}
Let us now observe that the solutions $u$ 
constructed as above satisfy the weak maximum principle.  
We follow the argument in \cite[pp 179-180]{GT}: given 
$f\in \text{Lip}_c(\ree)$, 
let $u$ be the solution constructed above,
and let $f_0\in \text{Lip}_c(\pom)$ denote the restriction of $f$ to $\pom$.
 Set $M:= \sup_{\pom} f_0^+$, and observe that we seek to show that
 $(u-M)^+=0$ in $\Omega$.  Note that $(u-M)^+\in Y^{1,2}_0(\Omega)$, thus, since 
 $\nabla u = \nabla (u-M)^+$ in the set $\{X\in \Omega:(u(X)-M)^+ \neq 0\}$, 
 and since $Lu=0$ in the weak sense, we have
 \begin{multline*}
 \iint_{\Omega} |\nabla (u-M)^+|^2 dY
 \approx  \iint_{\Omega} A\nabla u\cdot\nabla (u-M)^+ dY
 \\[4pt]
 = - \iint_{\Omega} \Bb\cdot \nabla u (u-M)^+  dY
 = - \iint_{\Omega} \Bb\cdot \nabla (u-M)^+ \, (u-M)^+  dY
 \\[4pt]
 \lesssim \eps_0  \iint_{\Omega} |\nabla (u-M)^+|^2 dY\,,
 \end{multline*}
 where in the last step we used the Hardy inequality.
 For $\eps_0$ small enough, $\nabla (u-M)^+=0$  in $\Omega$,
 and since the boundary trace of $(u-M)^+$ is zero, we must have
 $u\leq M$ in $\Omega$.  We remark that in that case that $\Omega$ is bounded,
 we could have set $M:= \sup_{\pom} f_0$; in the unbounded case, we defined
 $M:= \sup_{\pom} f_0^+$ to ensure that $M\geq 0$, otherwise, it is not clear that
$(u-M)^+ \in L^{2^*}(\Omega)$.
 
 % Our next goal is to 
 % treat the case of continuous data with compact support on $\pom$.
 Next, we discuss continuity up to the boundary for the solutions to the Dirichlet problem
 constructed in Subsection \ref{ssy12}.
The following lemma will be useful in several circumstances.
 
\begin{lemma}\label{localCalpha}
 Let $L$ be as in \eqref{LdefShmsection}, where $\Bb$ satisfies \eqref{driftsize}.
 Suppose further that the conclusion of Lemma \ref{proppde} holds for $L$,
 and that for each $f\in$ {\em Lip}$_c(\pom)$, the equation $Lu=0$ has a solution
 $u\in Y^{1,2}(\Omega)$ (or in $W^{1,2}(\Omega)$), 
 with data $u\lfloor_{\pom}=f$ (in the trace sense),
 for which the weak maximum principle holds.
 Then for every compact set $K\subset \overline{\Omega}$,
 \begin{equation}\label{eqlocalCalpha}
 |u(Y) -f(x)|\,\leq\, C_{K,f} \,|Y-x|^{\alpha/2}\,,\qquad Y \in K\,, x\in\pom \,,
 \end{equation}
 where $\alpha$ is the % smaller of the two 
 H\"older exponent in % \eqref{eqHolder} and
  \eqref{holderatbdry}, and where the constant $C_{K,f}$ may depend upon $K$ and $f$, but otherwise depends only on the constants in Lemma \ref{proppde}. % allowable parameters.
 \end{lemma}
 
\begin{remark}\label{remark_no_smallness}
In particular, Lemma \ref{localCalpha} applies in the small constant case, by the construction above and Lemma \ref{proppde}.  On the other hand,
note that we do {\em not} impose smallness of the constant in \eqref{driftsize}), and in the sequel,
we shall also use Lemma \ref{localCalpha} in the absence of smallness.  The same remark applies to Corollary \ref{corlocCalpha} below.
\end{remark}

 \begin{proof}[Proof of Lemma \ref{localCalpha}]
 Let $f \in$ Lip$_c(\pom)$, and let $u_f$ be the corresponding solution. 
Let $x_0\in\pom $, and $Y\in K$, and set $r=|Y-x_0|$.  Since we allow the constant
in \eqref{eqlocalCalpha} to depend on $K$ and on $f$, we may suppose that $r<1$.
Let $\vp\in C_0^\infty\big(B(x_0,2)\big)$ 
be a smooth bump function with $\vp\equiv 1$ on $B(x_0,1)$, and let
$\Phi_r\in C^\infty(\ree)$, with $0\leq \Phi_r\leq 1$, $\Phi_r\equiv 0$ on
$B(x_0, 2\sqrt{r})$, and $\Phi_r \equiv 1$ on 
$\ree\setminus B(x_0, 4\sqrt{r})$.

Set $\tilde{f}:= f-f(x_0)\,\vp$, define $f_r:= \tilde{f}\, \Phi_r$, and 
 let $u_{\tilde{f}}$ and $u_{f_r}$ denote the corresponding solutions.
Note that $\sup_{\pom}| \tilde{f}-f_r| \, \lesssim_f \sqrt{r}$,
therefore by the maximum principle,
\[
\sup_\Omega |u_{\tilde{f}}-u_{f_r}| \lesssim_f \sqrt{r}\,,
\]
Also by the maximum principle, we have 
$|u_{f_r}|\leq u_{|f_r|}\lesssim\|f\|_\infty$, and since $f_r$ vanishes on 
$\Delta(x_0, 2\sqrt{r})=B(x_0, 2\sqrt{r})\cap\pom$, as 
does the trace of the solution $v:=1-u_\vp$, we therefore have
\[
|u_\vp(Y)-1| \,+ \,|u_{f_r}(Y)|\,\lesssim_f \left(\frac{\delta(Y)}{\sqrt{r}}\right)^\alpha\,,\quad
Y\in B(x_0, \!\sqrt{r})\cap\Omega\,,
\]
by Lemma \ref{proppde}.  Since $\delta(Y) \leq |Y-x_0| = r$, we obtain
\[
|u_f(Y) - f(x_0)|\leq 
|u_{\tilde{f}}(Y)| +|f(x_0)\,(u_\vp(Y)-1)|
\lesssim_f r^{\alpha/2} + r^{1/2} \lesssim_f r^{\alpha/2}\,.
\]
 \end{proof}
 
 \begin{corollary}\label{corlocCalpha} Under the same hypotheses as in Lemma
 \ref{localCalpha}, we have
  \begin{equation}\label{eqcorlocCalpha}
 |u(Y) -u(X)|\,\leq\, C_{K,f} \,|Y-X|^{\alpha/4}\,,\qquad X,Y \in K\,,
 \end{equation}
 where $\alpha$ is the  smaller of the two 
 H\"older exponents in \eqref{eqHolder} and
  \eqref{holderatbdry}.
 \end{corollary}
 \begin{proof}
 Since the constants in \eqref{eqcorlocCalpha} are allowed to depend on
 $K$ and $f$, we may assume that $r:=|X-Y|<1/2$, and that $\delta(X),\delta(Y) <1$.
 
 Suppose first that $r\leq \frac12\max(\delta^2(X),\delta^2(Y))$.  In this case, 
$ \delta(X) \approx \delta(Y)$, so by
 \eqref{eqHolder}, we have
  \[
  |u(Y) -u(X)|\,\lesssim_f  \left(\frac{|X-Y|}{\delta(X)}\right)^\alpha \lesssim_f  r^{\alpha/2}\,.
  \]
  On the other hand, suppose that $r\geq \frac12\max(\delta^2(X),\delta^2(Y))$, and fix
  $\hat{x},\hat{y} \in\pom$ such that $\delta(X)=|X-\hat{x}|$ and 
  $\delta(Y)=|Y-\hat{y}|$.
  Then
  \[
   |u(Y) -u(X)| \leq    |u(Y) -f(\hat{y})| + |f(\hat{y})-f(\hat{x})| + |u(X) -f(\hat{x})|=:I+II+III\,.
  \]
  By Lemma \ref{localCalpha}, since $\delta(Y) +\delta(X) \lesssim r^{1/2}$ in the present case,
  \[
  I+III \lesssim_{K,f} r^{\alpha/4}\,.
  \]
  By the triangle inequality, $|\hat{y}-\hat{x}| \leq \delta(Y) +\delta(X) +|X-Y|\lesssim r^{1/2}$, 
  so that $II \lesssim_f r^{1/2}$.
 \end{proof}

 \subsection{Existence of elliptic measure in the small constant case, and a local 
 ampleness property} \label{schm}
 We establish solvability of the Dirichlet problem with continuous, compactly supported data
 on $\pom$, and consequently, we obtain
 existence of elliptic measure, for the operator $L$ defined in \eqref{LdefShmsection}, under the smallness assumption \eqref{smallptwise}.  To be precise, we have the following.

\begin{lemma}\label{hmexistssmall}
Let $L$ be as in \eqref{LdefShmsection}, and suppose that \eqref{smallptwise} holds
for $\Bb$.  If $\eps_0$ is small enough, depending only on allowable parameters, then
the continuous Dirichlet problem is solvable for the equation $Lu=0$ in $\Omega$, and
therefore elliptic measure exists for $L$ in $\Omega$.
\end{lemma}
 
 \begin{proof}  The proof is standard.
 Given a nonnegative function $f\in C_c(\pom)$, let $\{f_m\}_m$ be a sequence
 of nonnegative functions in Lip$_c(\pom)$, such that
 $f_m\to f$ uniformly on $\pom$.  Set $u_m:=u_{f_m}$, 
 the solution of the Dirichlet problem with data
$f_m$, as constructed in Subsection \ref{ssy12}.  Since these solutions satisfy the 
weak maximum principle (Subsection \ref{WMP}), we have
\[
\|u_m-u_j\|_{L^\infty(\Omega)}= \|u_{f_m-f_j}\|_{L^\infty(\Omega)}  \leq \|f_m-f_j\|_{L^\infty(\pom)}\,,
\]
thus, $u_m$ converges uniformly in $\Omega$, and therefore by Caccioppoli's inequality,
$u_m$ converges also in $W^{1,2}_{loc}(\Omega)$.  Let $u$ denote the limit.  Then
$Lu=0$ in the weak sense. Indeed, for $\Phi\in C_0^\infty(\Omega)$,
\[
0=\iint_\Omega \left[A\nabla u_m \cdot \nabla \Phi 
+ \Bb\cdot \nabla u_m \Phi\right]dY
\to \iint_\Omega \left[A\nabla u \cdot \nabla \Phi 
+ \Bb\cdot \nabla u \Phi\right]dY\,.
\]
Now let $x\in\pom$, $Y\in\Omega$, and observe that
\[
|u(Y) -f(x)|\leq
|u(Y)-u_m(Y)|+|u_m(Y) -f_m(x)| + |f_m(x)-f(x)| =:I+II+III\,.
\]
Let $\epsilon>0$. Since $u_m\to u$ uniformly in $\Omega$, and $f_m\to f$ uniformly on $\pom$,
we can fix $m=m(\epsilon)$ such that $I+III \leq \epsilon$, uniformly for all $Y\in \Omega$ and 
$x\in\pom$.  With this choice of $m$ now fixed, we find that $II\leq \epsilon$, 
either by application of Lemma \ref{localCalpha} in the case that
$Y,x$ lie in a suitable neighborhood of $\supp(f_m)$, or else by application of
\eqref{holderatbdry}, 
in the case that $x$ is away from $\supp(f_m)$ and $Y$ is near enough to $\pom$.

Thus, we have shown that $u\in C(\overline{\Omega})$ with $u\equiv f$ on $\pom$. Moreover,
$u$ satisfies the weak maximum principle
$\|u\|_{L^\infty(\Omega)}\leq \|f\|_{L^\infty(\pom)}$, as may be seen
by taking limits of the analogous inequality for each $u_m,f_m$.
As usual, it then follows that the mapping $f\to u(X)$ defines a bounded linear functional on
$C_c(\pom)$, for each $X\in \Omega$, so that
\[
u(X) = \int_{\pom} f\, d\hm^X\,,
\]
for some Radon measure $\hm^X$ (the elliptic measure for $L$), by Riesz representation.
\end{proof}

In the next lemma, we suppose that $\eps_0$ is small enough that
elliptic measure exists, by Lemma \ref{hmexistssmall}.

\begin{lemma}\label{Bourgainhm}  Let 
$\Omega$ be a 1-sided CAD, and set $L=-\dv A \nabla +\Bb\cdot\nabla$,
where $\Bb$ satisfies \eqref{smallptwise}.
Let $\hm$ denote elliptic measure for $L$ in $\Omega$. 
If $\eps_0$ is small enough, depending only on dimension, ellipticity of $A$,
and the 1-sided CAD constants, then there is a uniform constant
% there are uniform constants $c\in(0,1)$ and 
$c\in(0,1)$ % $C\in (1,\infty)$ 
with the same dependence,
such that for every $x \in \partial\Omega$, and every $r\in (0,\diam(\partial\Omega))$,
if $Y \in\Omega_{r/2}:= \Omega \cap B(x,r/2)$, then
\begin{equation}\label{eq2.Bourgain1}
\omega^{Y} (\Delta(x,r)) \geq c>0 \;.
\end{equation}
\end{lemma}

\begin{proof} Let $Y \in \Omega_{r/2}$.
Set $u(Y):= 1-\omega^{Y} (\Delta(x,r))$.  Suppose first that
$\delta(Y) \leq \eta r$, where $\eta>0$ is a sufficiently small number to be chosen.
Then by Lemma \ref{proppde},
$u(Y) \leq C \eta^\alpha \leq 1/2$, provided that $\eta$ is small enough.  With this choice of
$\eta$ now fixed, we see that 
$\omega^{Y} (\Delta(x,r)) \geq 1/2$, for all
$Y \in \Omega_{r/2} \cap \{\delta(Y)\leq \eta r\}$.  In particular, the latter bound holds if
$Y=Y_0$ is a Corkscrew point relative to the ball $B(x,\eta r)$.  Since 
$\delta(Y_0) \approx_\eta r$, we see that \eqref{eq2.Bourgain1} holds for all
$Y \in  \Omega_{r/2} \cap \{\delta(Y)> \eta r\}$, with $c=c(\eta)$, by Harnack's inequality.
\end{proof}

 \subsection{Existence of elliptic measure in the large constant case, and an
 approximation result}\label{approxhm}
 In this subsection, we show that there is an elliptic measure for an operator
 $L$ as in \eqref{LdefShmsection}, with $\Bb$ satisfying \eqref{driftsize}, 
provided that $L$ can be suitably approximated
 by operators 
 \[
 L_k:= -\dv A\nabla +\Bb_k \cdot \nabla = L_0 + \Bb_k \cdot \nabla\,,
 \]
 for which the hypotheses of Lemma \ref{localCalpha} (and thus also of Corollary 
 \ref{corlocCalpha}) hold for each $L_k$ with constants that are uniform in $k$.
 Specifically, we shall set 
 \begin{equation}\label{bbkdef}
 \Bb_k:= \Bb 1_{\Omega_k}\,,
 \end{equation}
 where $\Omega_k:=\{X\in \Omega: \delta(X)> 1/k\}$ if $\Omega$ is bounded,
 or $\Omega_k:=\{X\in \Omega: \delta(X)> 1/k\}\cap B(x_0,k)$, for some fixed $x_0\in\pom$, if
 $\Omega$ is unbounded.
 % in the case that $\Omega$ is bounded, and 
 % \[
%  \Omega_k:=\{X\in \Omega: \delta(X)> 1/k\} \cap B(x_0,k)
%  \]
%  if $\Omega$ is unbounded, where $x_0$ is 
 % some fixed point on $\pom$.  
 Then $\|\Bb_k\|_{L^\infty(\Omega)} \leq k\sqrt{M_0}$, by \eqref{driftsize}, and
 $L_k=L_0$ in the border strip $\{X\in\Omega:\delta(X)\leq 1/k\}$. 
 % and in $\Omega\setminus B(x_0,k)$.

 In the next pair of Lemmata, we let $\hm_k^X$ and $\hm^X$ denote the respective
elliptic measures, at the point $X$, for the operators $L_k$ and $L$ in $\Omega$.
 
 \begin{lemma}\label{hmkexists}
 Let $L$ be as in \eqref{LdefShmsection}, with $\Bb$ satisfying \eqref{driftsize}.
Define $L_k$ as above.  Then for each positive integer $k$, 
the continuous Dirichlet problem is solvable for the equation
$L_k u=0$ in $\Omega$, the solution satisfies the weak maximum principle, and consequently
elliptic measure $\hm_k$ exists for $L_k$.
 \end{lemma}

 \begin{lemma}\label{hmklimit}
Let $L$ be as in \eqref{LdefShmsection}, with $\Bb$ satisfying \eqref{driftsize}.
Define $L_k$ as above, and suppose that 
the hypotheses of Lemma \ref{localCalpha} 
hold for each $L_k$, and that Lemma \ref{proppde} holds for each $L_k$, 
with constants that are uniform in $k$.
Then elliptic measure $\hm$ exists for $L$, and a subsequence
$\hm_{k_j}^X$ converges weakly to $\hm^X$, for each $X \in \Omega$.
\end{lemma}

 \begin{proof}[Proof of Lemma \ref{hmkexists}]
 We follow a familiar strategy.
First consider $f\in$ Lip$_c(\ree)$.  Since $\Bb_k$ is (qualitatively) bounded (and compactly supported if $\Omega$ is unbounded),
  by \cite[Lemma 4.2]{KS} in the case that $\Omega$ is bounded
  (respectively, \cite[Theorem 5.3]{M} in the case that $\Omega$ is unbounded),
  there is a solution $w$ of the equation $L_kw=L_kf$, belonging to $W^{1,2}_0(\Omega)$
  (resp., $Y^{1,2}_0(\Omega)$), with constants that may depend upon $k$.   
  Set $u:= f-w$, so that $u\in W^{1,2}(\Omega)$
  (resp., $Y^{1,2}(\Omega)$)
   is a weak solution
 of $L_k u =0$, with $u\lfloor_{\pom}= f$ in the trace sense.
  Then by \cite[Lemma 8.1]{GT} (resp., \cite[Theorem 5.1]{M}), 
  the weak maximum principle holds:
 \begin{equation}\label{WMPR}
 \sup_{\Omega} |u| \leq \sup_{\pom} |f|\,,
 \quad f\in \text{Lip}_c(\ree)\,. 
 \end{equation}
 
   We now proceed as in the proof of
   Lemma \ref{hmexistssmall}:  given a nonnegative function 
   $f\in C_c(\pom)$, we let $\{f_m\}_m$ be a sequence in Lip$_c(\pom)$, such that
 $f_m\to f$ uniformly on $\pom$.  Set $u_m:=u_{f_m}$, 
 the solution of the Dirichlet problem with data
$f_m$, as constructed above.  Since these solutions satisfy the 
weak maximum principle \eqref{WMPR}, we have
\[
\|u_m-u_j\|_{L^\infty(\Omega)}= \|u_{f_m-f_j}\|_{L^\infty(\Omega)}  \leq \|f_m-f_j\|_{L^\infty(\pom)}\,,
\]
thus, $u_m$ converges uniformly in $\Omega$, and therefore by Caccioppoli's inequality,
$u_m$ converges also in $W^{1,2}_{loc}(\Omega)$.  The limit $u$ is therefore a solution
of $L_ku=0$ in $\Omega$.
 
Recall that by construction, $L_k=L_0$ when $\delta(X)\leq 1/k$, and therefore
Lemma \ref{proppde} and
Lemma \ref{localCalpha} apply (with constants that may depend upon $k$) to the pair
$u_m$ and $f_m$.  The rest of the proof now follows that of Lemma \ref{hmexistssmall},
essentially verbatim, 
so we leave the remaining details to the reader. 
\end{proof}

\begin{proof}[Proof of Lemma \ref{hmklimit}]
Let $f\in$ Lip$_c(\pom)$, and set $u_k$ be the $Y^{1,2}$ (or $W^{1,2}$) solution
of $L_ku_k=0$, with boundary data $f$.
By assumption, the hypotheses of Lemma \ref{proppde}
hold uniformly in $k$ for each $L_k$, thus by
Lemma \ref{localCalpha} and Corollary \ref{corlocCalpha},
 the family $\{u_k\}$ is equicontinuous on 
 each compact set $K\subset \overline{\Omega}$, and
 moreover, the solutions $u_k$ are uniformly bounded on $\Omega$, by the weak maximum principle.  Consequently, for each such $K$, there is a subsequence of $u_{k_j}$ converging
 uniformly on $K$, and in particular, on each bounded subdomain 
 $\Omega'$ compactly contained in $\Omega$.  Observe that $L_k=L$ in any 
 such $\Omega'$, for all $k$ large enough (depending on $\Omega'$).
 Hence, by Caccioppoli's inequality (Lemma
 \ref{Caccinterior}), the convergence holds also in $W^{1,2}(\Omega')$.
By choosing a countable collection of compact sets that exhaust
 $\overline{\Omega}$ (obviously, this step is unnecessary when 
 $\Omega$ is bounded), we find a subsequential limit $u$ defined on all of $\Omega$.
Note that
\[
\|u\|_{L^\infty(\Omega)} \leq \sup_k \|u_k\|_{L^\infty(\Omega)} \leq 
\|f\|_{L^\infty(\Omega)}\,,
\]
since the solutions $u_k$ satisfy the weak maximum principle.

Let $\Phi\in C_0^\infty(\Omega)$.   Then $\Bb_k=\Bb$, hence $L_k=L$, in the support of $\Phi$,
provided $k$ is large enough.  Thus $Lu_k=0$ in $\supp(\Phi)$, and 
by the $W^{1,2}$ convergence on compactly contained subdomains, we conclude
that $u$ is a weak solution of $Lu=0$ in $\Omega$.
 
 We now claim that $u\in C(\overline{\Omega})$, and that $u=f$ on $\pom$.  Continuity in 
 $\Omega$ is immediate from \eqref{eqHolder}, since $Lu=0$.  Let us now verify continuity up to the boundary.  Let $x\in \pom$, and let $Y\in\Omega \cap B(x,1)=:\Omega_1(x)$. 
By re-numbering, we may assume that $u_k\to u$ uniformly on a compact set $K_1$
containing $\Omega_1(x)$.
We then have
\[
|u(Y)-f(x)| \leq |u(Y)-u_k(Y)| +|u_k(Y) - f(x)| \leq |u(Y)-u_k(Y)| + C_{K_1,f} |Y-x|^{\alpha/2}\,,
\]
since Lemma \ref{proppde}, and hence the quantitative conclusion of Lemma \ref{localCalpha}, apply to $u_k$ with constants independent of $k$.  
Letting $k\to \infty$, we find that
\eqref{eqlocalCalpha} holds for $u$, thus in particular $u\in C(\overline{\Omega})$,
with $u\lfloor_{\pom}=f$, as claimed.

As usual, existence of the elliptic measure $\hm$ now follows from the 
Riesz representation theorem.  Finally, the claimed weak convergence is simply the statement that for $f\in C_c(\pom)$, and for each $X\in\Omega$, if $u_k$ and $u$ denote, respectively, the solutions to $L_ku_k=0$ and $Lu=0$ with data $f$, then for some subsequence,
\[
\int_{\pom} f\, d\hm_{k_j}^X = u_{k_j}(X) \to u(X) = \int_{\pom} f\, d\hm^X\,.
\]
\end{proof}

\subsection{The Green function: existence, estimates, and consequences}

\begin{lemma}[Green function existence, properties and estimates]
 \label{lemma2.green}
Let $\Omega\subset \ree$ be a bounded
1-sided CAD, and let $L=-\dv A \nabla +\Bb\cdot\nabla$.
Suppose that $\Bb$ % satisfies \eqref{driftsize}, and 
is {\tt qualitatively} bounded.  
Then there exists a unique Green function
$G(X,Y)$ defined on $\Omega\times\Omega\setminus \{X=Y\}$, satisfying
\begin{equation}\label{greeny12}
G(Z,\cdot), G(\cdot,Z) \in W^{1,2}(\Omega \setminus B(Z,r))\,\,\, \forall\, Z\in \Omega\,, \,r>0\,;
\end{equation}
\begin{equation}
\label{eq2.green-cont}
% G(Z,\cdot), G(\cdot,Z)\in C(\overline{\Omega}\setminus\{Z\}) \,\,\, \text{\em and}\,\,\,
G(Z,\cdot)\big|_{\pom}\equiv 0\equiv
G(\cdot,Z)\big|_{\pom}\,\,\, \text{in the trace sense}\,,\,\,\, \forall \,Z\in\Omega\,;
\end{equation}
\begin{equation}
\label{eq2.green3}
G(X,Y)\geq 0\,,\qquad \forall X,Y\in\Omega\,,\, X\neq Y\,.
\end{equation}
Moreover,
setting $\mathfrak{G}F(X):= \iint_\Omega G(X,Y) F(Y) dy$ 
and $\mathfrak{G}^*F(Y):=\iint_\Omega G(X,Y) F(X) dX$,
we have 
\begin{equation}\label{LG1}
\mathfrak{G}F,\, \mathfrak{G}^*\!F\, \in W^{1,2}_0(\Omega)\,,
\quad \forall \, F\in L^{\infty}(\Omega)\,;
\end{equation}
\begin{equation}\label{LG}
L\mathfrak{G}F = F\,\, \text{\em and } \,\, L^*\mathfrak{G}^*F= F\,,
\quad \forall \, F\in L^{\infty}(\Omega)\,,
\end{equation}
in the weak sense.
In addition, for every $\Phi \in$ {\em Lip}$_c(\ree)$,
 % C(\overline{\Omega})\cap Y^{1,2}(\Omega)$, 
and $X\in\Omega$, we have
\begin{equation}\label{eq2.14}
% \int_{\partial\Omega} \Phi\,d\omega^X 
u(X) \,-
\, \Phi(X)
=
-\iint_\Omega
\big[A^{\tt \!T}(Z)\nabla_Z G(X,Z) \,+\, \Bb(Z) G(X,Z)\big]\cdot\nabla\Phi(Z)\, dZ\,,
\end{equation}
where $u$ is the $W^{1,2}$ solution of $Lu=0$ with $u\lfloor_{\pom}=\Phi\lfloor_{\pom}$
in the trace sense, and $A^{\tt \!T}$ denotes the transpose of $A$.

Finally, if in addition \eqref{smallptwise} holds, with $\eps_0$ sufficiently small, depending only
on $n$, ellipticity of $A$, and the 1-sided CAD constants, then
\begin{equation}\label{eq2.green}
|X-Y|^{1-n} \lesssim G(X,Y) \lesssim |X-Y|^{1-n}\,,\qquad |X-Y|\leq \frac12 \delta(Y)\,,
\end{equation}
where the implicit constants also depend only on the allowable parameters.  In particular,
$\eps_0$ and the implicit constants in \eqref{eq2.green} do {\em \bf not} depend upon
the qualitative $L^\infty$ bound for $\Bb$.
\end{lemma}

% We will prove Lemma \ref{lemma2.green} and Lemma \ref{l2.10} simultaneously.

\begin{proof}[Proof of Lemma \ref{lemma2.green}] % and \ref{l2.10}]
By the qualitative assumption, the existence of a unique Green function,
satisfying properties \eqref{greeny12} - \eqref{eq2.14},
follows from the results in \cite{KS} 
(or the results of \cite{S}).\footnote{In the case of an unbounded domain, one 
could substitute the results of \cite{M} for those of \cite{KS,S}.  We have not done this here, since
we shall use the Green function only in bounded subdomains of our original domain.}
% For these results, % of \cite{KS} or \cite{S}, 
% the 1-sided CAD assumption is not needed.

On the other hand, property \eqref{eq2.green} is the only quantitative conclusion in the Lemma, and to prove it, we cannot use the pointwise bounds
   of \cite{KS} off the shelf, since those estimates depend quantitatively on impermissible quantities;
   however, we can and do use the aforementioned existence and qualitative properties of the Green function proved in \cite{KS}.

 Let us now give the proof of \eqref{eq2.green}. 
   We first establish the right hand inequality. % in \eqref{eq2.green}. 
  We recall the standard upper and lower Sobolev exponents in $\re^d$, namely, 
  with $d=n+1$, 
  \[
  2^*=2d/(d-2)=2(n+1)/(n-1)\,,\quad \,\,  2_*=2d/(d+2)=2(n+1)/(n+3)\,.
  \]
   
First note that by Lemma \ref{lemmaPDbounds}, if $\eps_0$ is small enough,
then for any given $F\in Y^{-1,2}$, there is a unique
$w\in Y^{1,2}_0(\Omega)$ solving $Lw=F$ in the weak sense. Moreover, $w$ 
satisfies the bound \eqref{wfbound3}, so in particular
 \begin{equation}\label{2starbound}
 \|w\|_{L^{2^*}(\Omega)} \lesssim  \|F\|_{L^{2_*}(\Omega)}\,,
 \end{equation}
 with implicit constants depending only on allowable parameters.
 
 On the other hand, by 
 \cite[Lemma 4.2, and Propositions 5.3 and 6.13]{KS}, given $F\in W^{-1,2}(\Omega)$,
   there is a unique $w_1\in W^{1,2}_0{\Omega}$ such that $Lw_1 = F$ in the weak sense. 
  By Sobolev embedding, $w_1\in Y^{1,2}_0(\Omega)$, hence by uniqueness, for
   $F\in L^\infty(\Omega)$, we have $w_1=w$, the solution constructed in Lemma
   \ref{lemmaPDbounds}, which therefore satisfies \eqref{2starbound}.  Moreover,
   for $F\in L^\infty(\Omega)$, we have the formula
   \[
   w(X) =\iint_\Omega G(X,Y) F(Y) dY\,.
   \]
   Suppose now that $r:= |X-Y|\leq \delta(Y)/4$.  Let $B_1:= B(X,r/4)$,
   $B_2=B(Y,r/4)$, and set $F:=1_{B_2}$.  
   Then by Harnack's inequality \eqref{eqHarnack}, and \eqref{2starbound},
   \[
 |B_1|^{1/2^*}|B_2|\, G(X,Y)\, \lesssim \left( \iint_{B_1}\Big(\iint_{B_2}G(Z,W)\,dW\Big)^{2^*}dZ\right)^{1/2^*} \,\lesssim \, 
  |B_2|^{1/2_*}\,,
   \]
   i.e, in ambient dimension $n+1$,
   \[
   G(X,Y) \lesssim |B_1|^{-1/2^*} |B_2|^{-1/2^*} \approx r^{1-n} = |X-Y|^{1-n}\,.
   \]
   Using Harnack again, we may extend the estimate from the regime $|X-Y|\leq \delta(Y)/4$ to
   $|X-Y|\leq \delta(Y)/2$.

   We now turn to the left hand inequality in \eqref{eq2.green}.  We follow the argument in \cite{HL}.
   Note that by Lemma \ref{hmexistssmall}, the elliptic measure for $L$ exists. 
   % and by Lemma \ref{Bourgainhm}, it satisfies \eqref{eq2.Bourgain1}.
   Consequently, for $\Phi\in C_0^\infty(\ree)$, with $\Phi(X)=0$,
   \eqref{eq2.14} becomes
   \begin{equation}  \label{eqRiesz1}
% \int_{\partial\Omega} \Phi\,d\omega^X 
\int_{\pom} \Phi \, d\hm^X \,
=\,
-\iint_\Omega
\big[A^{\tt T}(Y)\nabla_Y G(X,Y) \,+\, \Bb(Y) G(X,Y)\big]\cdot\nabla\Phi(Y)\, dY\,,
\end{equation}
Now set $s=|X-Y|\leq\delta(Y)/2$, and choose 
$\Phi\in C_0^\infty(\ree)$ with $0\leq \Phi\leq 1$,
$\Phi\equiv 0$ in $B(X,s/4)$, and $\Phi \equiv 1 $ on 
$\overline{\Omega}\setminus B(X,s/2)$,
with $|\nabla \Phi|\lesssim 1/s$.
Since the left hand side of \eqref{eqRiesz1} equals 1, and 
since $\delta(Z) \gtrsim s$ in the support of $\nabla\Phi$, we have
\[
1\lesssim \frac1s \iint_{s/4 \leq |X-Z|\leq s/2}\left( |\nabla_Z G(X,Z)| 
+ s^{-1} G(X,Z) \right) dZ
\lesssim G(X,Y) s^{n-1}\,,
\]
by Harnack's inequality, where we have used Cauchy-Schwarz and
Caccioppoli's inequality to handle the gradient term.
The left hand inequality in \eqref{eq2.green} follows.
\end{proof}

In the next lemma, and in the sequel, we use the notation
\begin{equation}\label{DXdef}
\Delta_Y:=B\big(Y,10\delta(Y)\big)\cap\pom\,,\qquad Y\in\Omega\,.
\end{equation}
Note that $\Delta_Y$ is essentially a surface ball.  Indeed, if $\hat{y}\in \pom$ is chosen so that
$|Y-\hat{y}|=\delta(Y)$, then $\Delta(\hat{y},9\delta(Y))\subset \Delta_Y\subset 
\Delta(\hat{y},11\delta(Y))$.

\begin{lemma}[``CFMS" estimates]\label{l2.10}
Suppose that $\Omega$ is a bounded 1-sided CAD,
and let $L=-\dv A \nabla +\Bb\cdot\nabla$,
where $\Bb$ satisfies \eqref{smallptwise},
and is {\tt qualitatively} bounded.  
If $\eps_0$ is sufficiently small, depending only
on $n$, ellipticity of $A$, and the 1-sided CAD constants, then for all
$X,Y\in \Omega$ such that $|X-Y|\ge \delta(Y)/4$ we have
\begin{equation}\label{eqn:right-CFMS}
\frac{G(X,Y)}{\delta(Y)}\,
\lesssim
\,\frac{\hm^X( \Delta_Y)}{\sigma( \,\Delta_Y)}\,,
\end{equation}
and for all $X,Y\in \Omega$ such that $X\in \Omega \setminus B(Y, 50\kappa_0\delta(Y))$
\begin{equation}\label{eqn:left-CFMS}
\frac{\hm^X( \Delta_Y)}{\sigma( \,\Delta_Y)}\,
\lesssim\, \frac{G(X,Y)}{\delta(Y)}\,,
\end{equation}
%$\Delta_Y=B(Y,10\delta(Y))\cap \pom$, and 
where the implicit constants depend only on allowable parameters, and $\kappa_0$ is the constant fixed in Lemma \ref{localHardy}.
% with $\hat{y}\in\pom$ such that $|Y-\hat{y}|=\delta(Y)$.
\end{lemma}

\begin{proof}
We follow the argument in \cite[p.~10]{K}.  Given $Y\in\Omega$, 
and using the right hand inequality in \eqref{eq2.green},we see that for
$|X-Y|=\delta(Y)/4$,
\begin{equation*}
\frac{G(X,Y)}{\delta(Y)}\,\lesssim \, \delta(Y)^{-n}
\lesssim
\,\frac{\hm^X( \Delta_Y)}{\sigma( \,\Delta_Y)}\,,
\end{equation*}
by Lemma \ref{Bourgainhm} (and Harnack's inequality), and ADR.
Estimate \eqref{eqn:right-CFMS} now follows by the maximum principle.

To prove \eqref{eqn:left-CFMS}, we set $R:= \delta(Y)$, and 
fix $Y\in \Omega$ and 
$X\in \Omega \setminus B(Y, 20R)$. 
Choose a smooth bump function
$\Phi\in C_0^\infty (B(Y, 11R))$, with $0\leq \Phi\leq 1$,
$\Phi\equiv 1$ on $B(Y, 10R)$, and $|\nabla \Phi|\lesssim R^{-1}$.  
Choose $\hat{y}\in\pom$ such that $|Y-\hat{y}|=\delta(Y) = R$, and for any
$r>0$ set $\Omega_{r}:=\Omega\cap B(\hat{y},r)$.
Note that 
$\Phi(X)=0$, so that by \eqref{eq2.14} and Lemma \ref{hmexistssmall},
\begin{multline*}
\hm^X(\Delta_Y)\leq
\int_{\pom}\Phi\, d\hm^X 
=
-\iint_\Omega
\big[A(Z)\nabla_Z G(X,Z) \,+\, \Bb(Z) G(X,Z)\big]\cdot\nabla\Phi(Z)\, dZ
\\[4pt]
\lesssim\, \frac1R\iint_{\Omega_{12R}}
\left(|\nabla_Z G(X,Z)| \,+\, \frac{ G(X,Z)}{\delta(Z)}\right) dZ 
\\[4pt]
\lesssim R^{(n-1)/2} \left(\iint_{\Omega_{12\kappa_0 R}}|\nabla_Z G(X,Z)|^2 dZ\right)^{1/2}
\\[4pt]
\lesssim R^{(n-1)/2}R^{-1} \left(\iint_{\Omega_{24\kappa_0 R}} G^2(X,Z) dZ\right)^{1/2}
\lesssim R^{n-1} G(X,Y)\,,
\end{multline*}
where we have used Cauchy-Schwarz and the local Hardy inequality 
(Lemma \ref{localHardy}), then Caccioppoli at the boundary (Lemma \ref{Cacc})
and finally the Carleson estimate (Lemma \ref{proppde}, specifically
 \eqref{carleson}).  Since $R=\delta(Y)$, the ADR property yields \eqref{eqn:left-CFMS}.
\end{proof}

\begin{lemma}\label{lemmadouble} 
Under the same hypotheses as in Lemma \ref{l2.10}, for any ball $B=B(x,r)$, centered on
$\pom$, setting $\Delta =B\cap\pom$, we have
\begin{equation*} % \label{eq3.hmdouble}
\hm^X(2\Delta) \,\lesssim\, \hm^X(\Delta)\,,\qquad X\in\Omega\setminus 4B\,,
\end{equation*}
where the implicit constant depends only on the allowable parameters.
\end{lemma}

\begin{proof} The lemma will be an easy corollary of Lemma \ref{l2.10}, and the proof 
is the standard one.  Let 
$\kappa_0$ be the fixed constant in Lemma \ref{localHardy} and Lemma \ref{l2.10}, and let
$M$ be a sufficiently large number to be fixed momentarily. Given a ball $B=B(x,r)$,
let $X\in\Omega\setminus 4B$, let
$Y$ be a corkscrew point relative to the ball $B(x,r/M)$, and cover $2\Delta$ by a collection
of surface balls $\mathcal{C}=\{\Delta_k\}_k$, such that $\Delta_k=B_k\cap\pom$, and
each $B_k$ is centered on $\pom$,
and has radius $r_{B_k}= r/(M^2\kappa_0)$. By the ADR property of $\pom$,
we may do this in such a way that $\mathcal{C}$ has 
cardinality $\#(\mathcal{C})\approx (M^2\kappa_0)^n$.  
For each $k$, let $Y_k$ be a corkscrew point relative to $MB_k$.
Choosing $M$ large enough, depending only on the constants in the Corkscrew condition,
we have $\Delta_Y\subset \Delta$, and also 
$\Delta_k \subset \Delta_{Y_k}$, and furthermore, 
$X\in \Omega \setminus B(Y_k, 50\kappa_0\delta(Y_k))$.
Note that, since $\kappa_0$ and $M$ have been fixed, we have
\[
\delta(Y)\approx\delta(Y_k) \approx r\,, \quad \forall\, k\,,
\]
and thus also, by ADR, 
\[
\sigma(2\Delta)\approx \sigma(\Delta)
\approx \sigma(\Delta_{Y_k})\approx \sigma(\Delta_Y) \approx r^n\,.
\]
By Lemma \ref{l2.10}, and then Harnack's inequality, we then have
\[
\frac{\hm^X( \Delta_{Y_k})}{r^n}\,
\lesssim\, \frac{G(X,Y_k)}{r}\,\lesssim\,\frac{G(X,Y)}{r} \,.
\]
Consequently, since $2\Delta$ is covered by the collection $\mathcal{C}$, we see that
\[
\frac{\hm^X( 2\Delta)}{r^n}\,\leq\, \sum_k \,\frac{\hm^X( \Delta_{Y_k})}{r^n}\,
\lesssim\,  \sum_k  \frac{G(X,Y)}{r}\, \lesssim \,\frac{\hm^X( \Delta_Y)}{r^n}\,
\leq \, \frac{\hm^X( \Delta)}{r^n}\,,
\]
where in the last two steps we have used \eqref{eqn:right-CFMS} and the fact 
that the cardinality of
$\mathcal{C}$ is bounded by a universal constant, and that $\Delta_Y\subset \Delta$.
\end{proof}

Finally, we record two results that hold for the homogeneous second order operator 
$L_0=-\dv A\nabla$.
These results are well known.  We refer the reader to, e.g., \cite[pp. 10-11]{K}: the proofs given there carry over routinely to our setting, and we omit the details. 

\begin{lemma}\label{lemmaBHP}(Comparison Principle/Boundary Harnack Principle)
Let $\Omega$ be a 1-sided CAD, let
$\Delta=B\cap\pom$ be a surface ball centered on $\pom$, and let $T_\Delta$ be the associated Carleson tent, as in Proposition \ref{sawtoothprop} (8).  Suppose that $u,v$ are two non-negative $W^{1,2}$ solutions of the equation $L_0 u=L_0v=0$ in the tent $T_{2\Delta}$,
vanishing in the trace sense on $\partial T_{2\Delta}\cap\pom$ 
(and hence continuously on $\partial T_{\Delta}\cap\pom$, by Lemma \ref{proppde} and properties of the tents).   Let $X_B$ be a corkscrew point relative to the ball $B$.  Then
\[
\sup_{X\in T_\Delta}\, \frac{u(X)}{v(X)} \,\lesssim \,\frac{u(X_B)}{v(X_B)}\,,
\]
where the implicit constant depends only on ellipticity, dimension, 
and the 1-side CAD constants.
\end{lemma}

\begin{lemma}\label{lemmapolechange}(Pole Change formula)
Let $\Omega$ be a 1-sided CAD, let $B$ be a ball centered on $\pom$,  
and let $X_B$ be a corkscrew point relative to
$B$.  Let $\hm_0$ denote elliptic measure for $L_0$.
Then for every Borel set 
$E\subset \Delta:=B\cap\pom$, we have
\[
 \frac{\omega_0^X(E)}{\omega_0^X(\Delta)}\,\approx\,\omega_0^{X_B}(E)\,,
 \qquad \forall X\in \Omega\setminus 2B\,.
\]
\end{lemma}

Lemma \ref{lemmapolechange} follows from 
Lemma \ref{l2.10}, along with
Lemma \ref{lemmaBHP} applied to the Green function $G_{L_0}(X,Y)$ in the adjoint variable $Y$. 
We remark that we do not know whether Lemma
\ref{lemmapolechange}  holds for $L$ (as opposed to $L_0$), even in the small constant case.  The obstacle is that we do not know whether Lemma
\ref{lemmaBHP} holds for non-negative solutions of the adjoint equation $L^*u=0$ (in particular, for the Green function for $L$, in the adjoint variable).

\section{Perturbing $L_0$ in a small constant case}\label{Ssmall2}

In this section, we prove an auxiliary result in a suitable ``small constant" case,
% i.e., we assume that the constant 
% $M_0$ in \eqref{driftsize} and \eqref{driftCarleson} satisfies $M_0\leq \eps_0$, with 
% $\eps_0>0$ sufficiently small, 
analogous to \cite[Theorem 2.5]{FKP}, in which we % obtain $L^2$ solvability 
show that elliptic measure for $L$ belongs to $A_\infty$ with respect to 
the elliptic measure for $L_0$.
Before stating the precise result, let us record some notation and a couple of definitions.

In this section, we assume that $\Omega$ is a {\em bounded} 1-sided CAD, with
$\diam(\Omega)=: r_0$.  We let $X_0\in\Omega$ be a fixed corkscrew point with respect to
the entire boundary, which we can view as $\Delta_{2r_0}=\pom$, where $\Delta_{2r_0}$ is any surface ball on $\pom$ of radius $2r_0$, and arbitrary center on $\pom$.  Thus
\[
\delta(X_0) \approx r_0\,.
\]
Set $L_0:= -\dv A\nabla$, $L:= L_0 + \Bb\cdot\nabla$,
let $G_0$ and $G$ denote, respectively, the Green functions for 
$L_0$ and $L$ in $\Omega$, and let $\hm^X_0$ and $\hm^X$ denote the corresponding
elliptic measures, at any point $X\in \Omega$.  In the special case that
$X=X_0$, we shall often simply write $\hm_0:=\hm_0^{X_0}$, and $\hm:=\hm^{X_0}$.
We ensure the existence of $G$ and $\hm$, by first working with truncations of 
$\Bb$, as in Subsection \ref{approxhm}.  After establishing appropriate quantitative
estimates for elliptic measure that hold uniformly for all the truncations, we may then pass to the limit.  We shall return to the latter point shortly, in more detail (see the beginning of the proof of
Theorem \ref{Tauxsmall}).

\begin{definition}\label{epsmall}{\bf ($\eps$-smallness)} We say that 
$\Bb$ is ``$\eps$-small with respect to $\hm_0$" if
\begin{equation}\label{eqepsmall1}
|\Bb(Y)| \leq \frac{\eps}{\delta(Y)}\,,\quad \text{a.e. } Y\in\Omega\,,
\end{equation}
and if for every $\Delta=B\cap\pom$, 
 with $B=B(x,r)$, $x\in\pom$, $0<r<r_0$,
 \begin{equation}\label{eqepsmall2}
 \iint_{B\cap\Omega} |\Bb(Y)|^2\, G_0(X_0,Y) \,dY\, \leq\, \eps^2 \hm_0(\Delta)
 \end{equation}
 \end{definition}
 
Recall that, associated to any point 
$X\in\Omega$, we use the notation
\[
 \Delta_X:=B\big(X,10\delta(X)\big)\cap\pom\,,
 \] 
 which essentially defines a 
surface ball on $\pom$.
 
 \begin{definition}\label{defample} {\bf ($(\ttm,c_0,\theta)$-ampleness)}
 Let $\ttm$ be a doubling Borel measure on $\pom$, % ${\tt m}$ $\mathbbm{m}$ $\ttm$ $\bbm$
 and let $c_0,\theta\in (0,1)$.  We say that an
 elliptic measure $\hm$ has the $(\ttm,c_0,\theta)$-ampleness property if for all $X\in \Omega$, and for any Borel set $F\subset \Delta_X$,
 \begin{equation}\label{eqample}
 \ttm(F) \geq (1-\theta)\,\ttm(\Delta_X) \,\implies\, \hm^X(F) \geq c_0\,.
 \end{equation}
 \end{definition}
 
 We shall require the following lemma, which is a routine extension of \cite[Lemma 2.2]{BL}.
 For the purposes of Lemma \ref{BLlemma2}, we suppose that
 $L$ is an elliptic operator (possibly with a drift term)
 defined on $\Omega$, for which the continuous Dirichlet problem (Definition \ref{Dc})
is solvable, the weak maximum principle holds,
 and for which positive solutions satisfy Harnack's inequality.
 
 \begin{lemma}\label{BLlemma2} Suppose that $\Omega$ is an open set with an ADR boundary.    Let $L$ be an elliptic operator as in the preceeding paragraph, and let
  $\hm$ denote the elliptic measure associated to $L$. 
Let $\ttm$ be a doubling Borel measure on $\pom$, with doubling constant $N_{db}$, 
and suppose that
 $\hm$ has the $(\ttm,c_0,\theta)$-ampleness property for some constants $c_0,\theta\in(0,1)$.
 Then, given $s\in (0,1)$, there is a constant $\eta=\eta(\theta,s,N_{db})$ such that
 for any surface ball $\Delta=B\cap\pom$ with radius less than $\frac12 \diam(\pom)$, and for any Borel set $F\subset \Delta$, we have the implication
 \begin{equation}\label{eqBL}
 \ttm(F)\geq (1-\eta)\, \ttm(\Delta) \,\implies\, \hm^Z\big(\frac12 \Delta\big)
 \leq C_{\!s} \hm^Z(F) + s \hm^Z(\Delta)\,,\quad \forall\, Z\in \Omega\setminus 2B\,.
 \end{equation}
  \end{lemma}
  
  We omit the proof of Lemma \ref{BLlemma2}, which originally appeared as \cite[Lemma 2.2]{BL}, in the special case that $L$ is the Laplacian, and $\ttm =\sigma$, the surface measure on
  $\pom$.  The proof carries over to the present setting essentially unchanged:  in fact, the argument in \cite{BL}
  requires only that $L$ and $\ttm$ satisfy the properties stated above.
  
  Our main result in the section is the following analogue of \cite[Theorem 2.5]{FKP}. 
  We use the notation and terminology discussed above, in the preamble to this section.

 \begin{theorem}\label{Tauxsmall} Let $\Omega$ be a bounded 1-sided CAD, with
 $\diam(\Omega) =r_0<\infty$.  Then there is a constant $\eps_0>0$, depending on dimension, ellipticity of the coefficient matrix $A$, and the 1-sided CAD constants, such that if
 $\Bb$ is $\eps$-small with respect to $\hm_0:=\hm_0^{X_0}$, with $\eps\leq \eps_0$,
 then $\hm\in A_\infty(\hm_0)$ in the sense of Definition \ref{deflocalAinfty} with $\ttm=\hm_0$,
 and all constants implicit in the $A_\infty$ condition
 depend only on the same allowable parameters.  In particular,
 $\hm^{X_0}\in A_\infty(\hm_0,\pom)$.
 \end{theorem}

 \begin{proof} We shall work with the operator $L_k$, with truncated drift coefficient
 $\Bb_k$, as in Subsection \ref{approxhm}. Note that the smallness conditions of
 Definition \ref{epsmall} hold uniformly in $k$.  Then elliptic measure exists, by
 Lemma \ref{hmexistssmall}. 
 Also, $\Bb_k$ is (qualitatively) bounded (depending on $k$), so the conclusions of 
 Lemmas \ref{lemma2.green}, \ref{l2.10}, and \ref{lemmadouble} are all valid, and
all the quantitative bounds (explicit and implicit) in those Lemmas hold uniformly in $k$.
Moreover, the hypotheses of Lemma \ref{hmklimit} also hold.
Note further that since $L_0=L_k$ near the boundary, $\hm_0$ and
 $\hm_k$ are (qualitatively) mutually absolutely continuous, hence
 $h_k^X:= d\hm_k^X/d\hm_0$ exists, for each $X\in\Omega$.  
 Recall that $\hm_0:= \hm_0^{X_0}$. We shall show that
 $h^X_k \in L^2(\hm_0)$, uniformly in $k$, with an $L^2$ bound 
 that in turn implies the claimed
 $A_\infty$ property.  We may then invoke Lemma \ref{hmklimit} to deduce that
 $h^X :=d\hm^X/d\hm_0$ exists, and belongs to $L^2(\hm_0)$, with the same bound
 that held uniformly for all $h_k^X$, whence the conclusion of Theorem \ref{Tauxsmall}
 follows.
 
 % In addition, by \eqref{eqepsmall1}, Lemma \ref{proppde} holds uniformly for all $k$, thus

For the rest of the proof, we therefore assume that $\Bb=\Bb_k$ for some $k\geq 1$, but we suppress the subscript $k$ to simplify the notation.  We may then invoke Lemmas \ref{lemma2.green}, \ref{l2.10}, and \ref{lemmadouble} as needed.
  Of course, all of the constants appearing in our estimates will be independent of the truncation, which allows passage to the limit.   
 
 By Lemma \ref{lemmadouble}, there is a constant $N_1$
 such that for $\Delta=B\cap \pom$, with
 $B$ centered on $\pom$, we have the doubling estimate
 \begin{equation}\label{4.double}
 \hm^Z(\Delta) \,\leq \, N_1 \, \hm^Z\big(\frac12 \Delta\big)\,,\qquad \forall \,
 Z\in \Omega \setminus 2B\,,
 \end{equation}
 uniformly for all $\eps\leq \eps_0$, for a sufficiently small $\eps_0>0$.
 
 % \steve{\Bl Note to myself:  check dependence of 
 % the constant in the claim, then state explicitly.
 % I guess only depends on dimension, CAD, and constants for $L_0$}
 \noindent{\bf Claim}.
 We claim that for all $\eps\leq \eps_0$, with $\eps_0$ sufficiently small,
 \begin{equation}\label{eqclaim}
 \|h^X\|_{L^2(\pom,\hm_0)}\, \lesssim \, \hm_0(\Delta_X)^{-1/2} + \eps\,,
 \end{equation}
 for some uniform implicit constant depending only on $n$, the 1-sided CAD constants, and
 ellipticity of $A$.  
 
 Let us take the claim for granted momentarily.  Fix $X\in \Omega$, and let $E\subset\Delta_X$ be a Borel set satisfying
 \begin{equation}\label{eq4.11}
 \hm_0(E) \,\leq\, \theta \hm_0(\Delta_X)\,,
 \end{equation}
 where $\theta\in (0,1)$ will be chosen momentarily.
  Then for all $\eps\leq \eps_0$,
 \begin{multline*}
\hm^X(E) =\int_{E} h^x d\hm_0 \,\leq\, \hm_0(E)^{1/2}\,  \|h^X\|_{L^2(\pom,\hm_0)}
\, \lesssim \,\left(\frac{\hm_0(E)}{\hm_0(\Delta_X)}\right)^{1/2} +\,\hm_0(E)^{1/2}\eps
\\[4pt]
\lesssim \, \theta^{1/2} +\eps\, \lesssim \, \left(\theta^{1/2} +\eps_0\right) \hm^X(\Delta_X)\,,
\end{multline*}
by \eqref{eqclaim}, \eqref{eq4.11}, and then Lemma \ref{Bourgainhm}.  Choosing $\theta$ 
and $\eps_0$ sufficiently small, we obtain that $\hm^X(E) \leq \frac12 \hm^X(\Delta_X)$, 
whence it follows that for some uniform $c_0>0$,
\[
\hm_0(F) \geq (1-\theta)\,\hm_0(\Delta_X) \,\implies\, \hm^X(F) \geq \frac12 \hm^X(\Delta_X)\geq c_0\,,
\]
again by Lemma \ref{Bourgainhm}, provided that $\eps\leq \eps_0$; i.e.,
$\hm$ is $(\hm_0,c_0,\theta)$-ample.  Fixing this choice of $\theta$, and recalling that $N_1$ denotes the doubling constant for $\hm$ (see \eqref{4.double}),
we may therefore apply Lemma \ref{BLlemma2}, with
$\ttm=\hm_0$, and with $s=1/(2N_{1})$, to deduce that there is a uniform choice of
$\eta\in(0,1)$ such that for 
any surface ball $\Delta=B\cap\pom$ with radius less than $\frac12 \diam(\pom)$, 
for any Borel set $F\subset \Delta$, and for every $Z\in\Omega\setminus 2B$,
 \begin{multline*}
 \hm_0(F)\geq (1-\eta)\, \hm_0(\Delta) \,\implies\, 
 \frac1{N_1}\hm^Z(\Delta)\,\leq\, \hm^Z\big(\frac12 \Delta\big)
 \,\leq\, C_{\!s} \hm^Z(F) + s \hm^Z(\Delta)
 \\[4pt]
 =\, C_{\!s} \hm^Z(F) + \frac1{2N_1} \hm^Z(\Delta)\,.
 \end{multline*}
Thus, hiding the last term, we have
 \begin{equation}\label{eqAinfty}
 \hm_0(F)\geq (1-\eta)\, \hm_0(\Delta) \,\implies\, \hm^Z(F) \gtrsim \hm^Z(\Delta)
 \,,\quad \forall\, Z\in \Omega\setminus 2B\,,
\end{equation}
where the implicit constant is uniform. Since \eqref{eqAinfty} holds also 
with $\Delta$ replaced by $\Delta' =B'\cap\pom$, 
whenever $B'\subset B$, we see that $\hm^Z\in A_\infty(\Delta,\hm_0)$
for each $Z\in \Omega\setminus 2B$.  In particular, taking $Z=X_0$, we obtain the conclusion 
of Theorem \ref{Tauxsmall}, modulo \eqref{eqclaim}.

It therefore remains only to verify the claimed estimate \eqref{eqclaim}.
To this end, let $f\in L^2(\pom,\hm_0)$ be non-negative, 
with $\|f\|_{L^2(\pom,\hm_0)}\leq 1$, and set
\[
u_0(X)=u_{0,f}(X):= \int_{\pom} f d\hm_0^X\,,\quad u(X)
=u_{f}(X):= \int_{\pom} f d\hm^X=
\int_{\pom} f h^X d\hm_0\,.
\]
By Lemma \ref{lemmapolechange},
\[
\frac{d\hm_0^X}{d\hm_0} \approx \hm_0(\Delta_X)^{-1}\,,
\]
 and therefore, since $\hm_0^X$ is a probability measure,
 \begin{equation}\label{u0est}
 u_0(X) \leq \left(\int_{\pom} f^2 d\hm_0^X\right)^{1/2} \lesssim
 \|f\|_{L^2(\hm_0)}\, \hm_0(\Delta_X)^{-1/2} \lesssim \hm_0(\Delta_X)^{-1/2}\,.
 \end{equation}
 In particular, if $X=X_0$, we simply have
  \begin{equation}\label{u0est2}
   u_0(X_0) \leq \left(\int_{\pom} f^2 d\hm_0\right)^{1/2} \leq 1\,.
   \end{equation}
 Observe also that
 \[
  \|h^X\|_{L^2(\hm_0)} =\sup u_{f}(X)\,,
 \]
 where the supremum runs over all non-negative $f\in L^2(\hm_0)$ with
 $\|f\|_{L^2(\hm_0)}\leq 1$.
 Consequently, to prove \eqref{eqclaim}, it is more than enough to show that
  \begin{equation}\label{Fest}
  |u(X)-u_0(X)|\,\lesssim \, % \hm_0(\Delta_X)^{-1/2} + 
  \eps u_0(X) + \eps  \|h^X\|_{L^2(\hm_0)} +\eps
   \end{equation}

Since $u_0-u = 0$ on $\pom$, we may follow the idea of \cite{FKP} to write, 
using \eqref{eq2.14},
 \begin{multline*}
F(X):=\, u_0(X)-u(X) \,= \iint_\Omega G(X,Y) \Bb(Y)\cdot\nabla u_0(Y)\, dY
\\[4pt]
=\, \iint_{|X-Y| < \delta(Y)/4} ... \,\, dY % G(X,Y) \Bb(Y)\cdot\nabla u_0(Y)\, dY
\,+\,
\iint_{\Omega  \cap \{|X-Y| \geq \delta(Y)/4\}} ... \,\, dY % G(X,Y)\Bb(Y)\cdot\nabla u_0(Y)\, dY
\,=: \,I + II\,,
 \end{multline*}
 With $X$ fixed, set $r=\delta(X)$.  Observe that in term $I$,
 $|X-Y| < \delta(Y)/4$ implies that $\delta(Y)/4 < \delta(X)/3$.  Thus, 
by \eqref{eqepsmall1} and 
the pointwise estimate for the Green function (in ambient dimension $\ree$),
 \begin{multline*}
|I| \, \lesssim \,\eps r^{-1} \iint_{|X-Y| < r/3} |X-Y|^{1-n}\,|\nabla u_0(Y)|\, dY
\\[4pt]
\,\lesssim \, \eps  
\sum_{k= 1}^{\infty} 2^{-k}\iint_{2^{-k-1} r\,\leq|X-Y| < 2^{-k} r} (2^{-k} r)^{-n}\,|\nabla u_0(Y)|\, dY
\,=:\, \eps   \sum_{k= 1}^{\infty} \,2^{-k} I_k\,.
 \end{multline*}
By the inequalities of Cauchy-Schwarz, Caccioppoli, and Harnack,
$I_k \lesssim u_0(X)$,
and summing the geometric series, we then find that 
\[
|I| \lesssim\, \eps u_0(X)\,.
\]

We consider now term $II$.  For notational convenience, let $\Omega^X$ denote the domain of integration in term $II$, i.e., 
$$
\Omega^X:=  \{Y\in\Omega: |X-Y| \geq \delta(Y)/4\}\,,
$$
and set
$$
\Omega_0:= \{Y\in\Omega: |X_0-Y| \geq \delta(X_0)/4\}\,,\quad
\Omega^{X,X_0}:=  \Omega^X\cap \Omega_0\,.
$$
We then make the further splitting
\[
II\,=\, \iint_{\Omega^{X,X_0} } ... \,\, dY
\,+\, \iint_{\Omega^X  \cap \{|X_0-Y| < \delta(X_0)/4\}} ... \,\, dY\,=:\, II_1 + II_2\,.
\]
Observe that $\delta(Y) \approx \delta(X_0)$ in term $II_2$, and thus also
$|X-Y| \gtrsim \delta(X_0)$.  Hence, by Lemma \ref{lemma2.green} 
and \eqref{eqepsmall1}, 
\begin{multline*}
|II_2|\, \lesssim\, \eps \delta(X_0)^{-n} \iint_{|X_0-Y| < \delta(X_0)/4}|\nabla u_0(Y)|\, dY\,
\\[4pt]
\lesssim\, \eps \left(\delta(X_0)^{-n-1}\iint_{|X_0-Y| < \delta(X_0)/2}| u_0(Y)|^2\, dY\right)^{1/2}\,
\lesssim\,\eps u_0(X_0)\lesssim \eps\,,
\end{multline*}
where in the last three steps we have used
the inequalities of Cauchy-Schwarz and Caccioppoli, then Harnack, and finally \eqref{u0est2}.

By Lemma \ref{l2.10}, setting $\mathcal{G}_0(Y):= G_0(X_0,Y)$, we see that
\begin{multline*}
| II_1|\,\lesssim \,
\iint_{\Omega^{X,X_0}  } 
\,\frac{\hm^X(\Delta_Y)}{\hm_0(\Delta_Y)}\,|\Bb(Y)| \,|\nabla u_0(Y)|\, \mathcal{G}_0(Y)\,dY
\\[4pt]
\lesssim\, \left(
\iint_{\Omega} 
H^X(Y)^2\,|\Bb(Y)|^2 \, \mathcal{G}_0(Y)\,dY\right)^{1/2} \left(
\iint_{\Omega_0} %{\Omega^{X,X_0}} 
|\nabla u_0(Y)|^2\, \mathcal{G}_0(Y)\,dY\right)^{1/2}
\\=:\, {\tt A}  \cdot {\tt B}\,,
\end{multline*}
where $H^X(Y):= \frac{\hm^X(\Delta_Y)}{\hm_0(\Delta_Y)}$. 
Recalling Definition \ref{conentmaxdef}, we observe that for $z\in\pom$,
\[
N_*H^X(z) \approx \mathcal{M}_{\hm_0} h^X(z)\,
\]
where $\mathcal{M}_{\hm_0}$ is the usual Hardy-Littlewood maximal operator, with respect to the doubling measure $\hm_0$.
Then by \eqref{eqepsmall2} and Carleson's embedding 
lemma (with respect to
$\hm_0$),
\[
{\tt A}\, \lesssim \,\eps\, \|\mathcal{M}_{\hm_0} h^X\|_{L^2(\hm_0)}\,
\lesssim \, \eps\, \| h^X\|_{L^2(\hm_0)}\,.
\]
Moreover, since $u_0 = f$ on $\pom$, we have as in \cite{DJK} that
\[
{\tt B} \lesssim \|f\|_{L^2(\hm_0)}\leq 1\,.
\]
Thus, $|II_1| \lesssim \eps\, \| h^X\|_{L^2(\hm_0)}$, so adding together our estimates for
$I,\,II_1$ and $II_2$, we see that \eqref{Fest} holds.  In turn, as noted above, 
this proves the claim \eqref{eqclaim}, and therefore also Theorem \ref{Tauxsmall}.
 \end{proof}

 % \steve{\Bl need  remarks explaining relation between $(D)_p$ and $RH_{p'}$.}
 % \steve{\Bl want to give precise definition of $A_\infty$ and weak $A_\infty$ and 
 % dyadic $A_\infty$.}

  % \steve{\Bl Need to define N.T. max and dyadic cones}

\section{Two key lemmas}\label{section2key}  % {Preliminary arguments for the proof of Theorem \ref{Tmain}}
In this section, we prove two lemmas that are central to our proof
of Theorem \ref{Tmain}.
The first is a stopping-time lemma 
that plays a key role in the proof of Theorem \ref{Tmain} in Section \ref{Sgeneral}, specifically, 
in our use of the method of
``extrapolation of Carleson measures", a bootstrapping technique to lift the Carleson measure constant,
developed by J. L. Lewis \cite{LM}, and based on
the corona construction of Carleson \cite{Car} and Carleson and Garnett \cite{CG}.
This technique was also used in \cite{HL}
(see in addition \cite{AHLT}, \cite{AHMTT}, \cite{HM-TAMS}, \cite{HM-I}).

We shall require some additional notation.  Throughout this section, $M$ is a fixed constant, eventually to be 
chosen depending only on allowable parameters.
Given $M\geq 2$, and a cube $Q\in\dd(\pom)$, we set
\[
\dd^M(Q):= \left\{Q': \frac1M\leq \frac{\ell(Q')}{\ell(Q)}\leq M\, \text{ and }
\dist(Q,Q') \leq M \diam(Q)\right\}\,,
\]
and define the $M$-fattened Whitney region
\begin{equation}\label{UQfat}
U_Q^M=\bigcup_{Q'\in \dd^M(Q)} U_{Q'}\,.
\end{equation}
% as well as the $M$-fattened Carleson box
% \[
% R_Q^M :=  \interior\left(\bigcup_{Q'\subset Q} U^M_{Q'}\right).
% \]
Note that the fattened Whitney regions enjoy the bounded overlap property
\[
\sum_Q 1_{U_Q^M}(Y) \lesssim_M 1\,, \qquad Y\in \Omega\,.
\]
Recall that $\dd_Q:=\{Q'\in\dd(\pom): Q'\subset Q\}$.
For $x\in\pom$, define the ``dyadic cone with aperture $M$, truncated at height
$\ell(Q)$" by
\begin{equation}\label{dyadiccone}
\Upsilon^M_{Q}(x):= \bigcup_{Q'\in\dd_Q:\, x\in Q'} U_{Q'}^M\,.
\end{equation}

We further define a ``truncated dyadic-conical square function" by
\begin{equation}\label{dy-conesqfunction}
\mathcal{A}^M_{Q}({\bf F})^2(x)\,:=\, 
\sum_{Q'\in\dd_Q}1_{Q'}(x)\iint_{U_{Q'}^M} |{\bf F}(Y)|^2 \delta(Y)^{1-n}\,dY \,,
\end{equation}
and its ``shortened" version
\begin{equation}\label{dy-conesqfunctionshort}
\mathcal{A}^{M,short}_{Q}({\bf F})^2(x)\,:=\, 
\sum_{Q'\in\dd_Q\setminus \{Q\}}1_{Q'}(x)\iint_{U_{Q'}^M} |{\bf F}(Y)|^2 \delta(Y)^{1-n}\,dY \,,
\end{equation}
Set
\begin{equation}\label{dy-conesqfunction2}
\mathfrak{A}_Q^M({\bf F})^2(x)\,:=\, \iint_{\Upsilon^M_{Q}(x)} |{\bf F}(Y)|^2 \delta(Y)^{1-n}\,dY\,,
\end{equation}
and observe that by the bounded overlap property, 
\begin{equation}\label{compare_sqfunctions}
\mathcal{A}^M_{Q}({\bf F})(x) \approx_M \mathfrak{A}_Q^M({\bf F})(x) \,,\qquad x\in \pom\,,
\end{equation}
(of course, the direction $\mathfrak{A}_Q^M({\bf F})(x)\leq \mathcal{A}^M_{Q}({\bf F})(x)$ 
is trivial).

Given a tree $\sbf$ with maximal element $Q(\sbf)=Q$, we also define the dyadic square function restricted to the tree by
\begin{equation}\label{dy-treesqfunction}
\mathcal{A}^M_{\sbf}({\bf F})^2(x)\,:=\, 
\sum_{Q'\in\sbf}1_{Q'}(x)\iint_{U_{Q'}^M} |{\bf F}(Y)|^2 \delta(Y)^{1-n}\,dY \,,
\end{equation}

Let $x_Q$ denote the ``center" of $Q$ as in \eqref{cube-ball}, and
set $B^M_Q:= B(x_Q, CM\diam(Q))$, with $C$ large enough that
$\cup_{Q'\in\dd_Q} U_{Q'}^M \subset B_Q^M$. Then
in the special case that ${\bf F} =\Bb_*$, we note for future reference that by definition,
\begin{multline}\label{eqAMCM}
\fint_Q \mathcal{A}^M_{Q}({\bf B_*})^2 d\sigma \,=\,
\fint_Q \sum_{Q'\in\dd_Q}1_{Q'}(x)\iint_{U_{Q'}^M} |\Bb_*(Y)|^2 \delta(Y)^{1-n}\,dY
\\[4pt]
\approx_M\,\frac1{\sigma(Q)} % \sigma(Q)^{-1}
\sum_{Q'\in\dd_Q} \iint_{U_{Q'}^M} |\Bb_*(Y)|^2 \delta(Y)\,dY
\\[4pt]
\lesssim_M\,\frac1{\sigma(Q)} \iint_{B_Q^M\cap\Omega} |\Bb_*(Y)|^2 \delta(Y)\,dY\,
\leq \, CM_1=:M_2\,,
\end{multline}
uniformly for every $Q\in\dd(\pom)$,
where $M_1$ is the constant in \eqref{driftmax}, and $C$ depends on $M$ and the various underlying constants.

We turn now to the stopping time lemma.

\begin{lemma}\label{lemmastoptime} Fix $M\geq 2$. 
% There is a constant $C=C(n,M,ADR)$ such that,
Let $Q\in \dd(\pom)$, and consider constants $a\geq 0,\, b>0$. Suppose that
\begin{equation} \label{aplusb}
\fint_{Q} \mathcal{A}^M_{Q}({\bf B_*})^2 \,d\sigma \,\leq\, (a+b)\,,
\end{equation}
then either 
\begin{equation} \label{sq-functionshortbound}
\fint_{Q} \mathcal{A}^{M,short}_{Q}({\bf B_*})^2 \,d\sigma \,\leq\, a\,,
\end{equation}
or, there is a tree $\sbf$ with maximal cube $Q(\sbf)=Q$ 
such that 
% with the following properties.
% family $\F=\{Q_j\}\subset\dd_{Q}$
% of pairwise disjoint cubes, and 
\begin{equation} \label{sq-functionsmall}
\mathcal{A}^M_{\sbf}({\bf B_*})^2(x)
\leq 2b,\qquad \sigma\text{ \em - a.e. } \, x\in Q\,,
\end{equation}
\begin{equation}
\label{Corona-bad-cubes}
\sum_{\F_{bad}} \sigma(Q_j) 
\,\leq\, \frac{a+b}{a+2b}\, \sigma(Q)\,,
\end{equation}
where 
 $\F_{bad}:=
\{Q_j\in\F:\,\fint_{Q_j}\mathcal{A}^{M,short}_{Q_j}({\bf B_*})^2d\sigma>\,a\}$, and 
$\F=\{Q_j\}_j$ denotes the collection of sub-cubes of $Q=Q(\sbf)$ that are maximal with respect to the property that $Q_j\notin\sbf$.
\end{lemma}

\begin{proof} For convenience of notation, we set 
\[
\mathcal{H}_Q^M
:= \mathcal{A}^M_{Q}({\bf B_*})^2 \,,\qquad \mathcal{H}_Q^{M,short}
:= \mathcal{A}^{M,short}_{Q}({\bf B_*})^2\,,
\]
and similarly with $Q$ replaced by any other cube $Q'\in\dd_Q$.
We also set
\[
\xi_{Q'}:= \iint_{U_{Q'}^M} |\Bb_*(Y)|^2 \delta(Y)^{1-n}\,dY\,, \quad Q'\in\dd_Q
\]
so that, e.g.,
$
\mathcal{H}_Q^M (x)\,=\,
 \sum_{Q'\in\dd_Q}1_{Q'}(x) \,\xi_{Q'}
$.  Note in particular that
\begin{equation}\label{HQMsplit}
\mathcal{H}_Q^M(x) = \mathcal{H}_Q^{M,short}(x) + \xi_Q\,,\qquad \forall \, x\in Q\,.
\end{equation}
For $Q''\subset Q$, we also set
\[
\mathcal{H}_{Q'',Q}^M(x):=  \sum_{Q': Q''\subset Q'\subset Q}1_{Q'}(x) \,\xi_{Q'}\,,
\]
so that, in particular, $\mathcal{H}_{Q,Q}^M(x)=\xi_Q$ for $x\in Q$.
We consider two cases.

\noindent{\bf Case 1}:  $\xi_Q\geq b$.  In this case, \eqref{aplusb} and \eqref{HQMsplit} imply that
\[
\fint_Q \mathcal{H}_Q^{M,short} d\sigma \leq a\,,
\]
i.e., \eqref{sq-functionshortbound} holds.
% In this case, we set $\F:=\{Q\}$, and observe that the tree $\sbf$ and the bad collection
% $\F_{bad}$ are both empty.  Thus, \eqref{sq-functionsmall} and 
% \eqref{Corona-bad-cubes} are both trivially true.

\noindent{\bf Case 2}:  $\xi_Q< b$.  In this case, we subdivide dyadically, to select a family
$\F:=\{Q_j\}_j\subset \dd_Q$ 
of cubes that are maximal with respect to the property that
\begin{equation}\label{stoppingbound}
 \sum_{Q': Q_j\subset Q'\subset Q} \,\xi_{Q'}\,=\,\fint_{Q_j} \mathcal{H}_{Q_j,Q}^M \,d\sigma >2b\,.
\end{equation}
Set $\sbf:=\{Q'\in\dd_Q: Q' \text{ is not contained in any } Q_j\in\F\}$.  If $x\in Q\setminus
\cup_{\F}Q_j$, then 
\begin{equation}\label{prestoppingbound}
\fint_{Q''} \mathcal{H}_{Q'',Q}^M \,d\sigma \leq 2b\,,
 \quad \forall \, Q''\ni x\,,
\end{equation}
so letting $Q''\downarrow x$, we see that
\eqref{sq-functionsmall} holds for a.e.~ such $x$.  

On the other hand, suppose now that $x\in Q_j$ for some $Q_j\in\F$.  In this case, 
by definition \eqref{dy-treesqfunction},
\begin{equation}\label{ASMidentity}
\mathcal{A}^M_{\sbf}({\bf B}_*)^2(x) = \mathcal{H}_{Q^*_j,Q}^M(x) =
 \sum_{Q': Q^*_j\subset Q'\subset Q} \,\xi_{Q'}\,,\quad\forall\, x\in Q_j\,,
\end{equation}
where $Q_j^*$ denotes the dyadic parent of $Q_j$. Moreover,
$\mathcal{H}_{Q^*_j,Q}^M(x)$ is constant on $Q_j^*$ (indeed,
 $\mathcal{H}_{Q^*_j,Q}^M(x)=
 \sum_{Q': Q^*_j\subset Q'\subset Q} \,\xi_{Q'}$ on $Q_j^*$),
 and therefore
\[
\mathcal{H}_{Q^*_j,Q}^M(x) = \fint_{Q_j^*} \mathcal{H}_{Q^*_j,Q}^M\,d\sigma \leq 2b\,,
\] 
by maximality of $Q_j$,
for all $x\in Q_j^*$, in particular, for all $x\in Q_j$.  Combining the latter estimate with the identity
in \eqref{ASMidentity}, we obtain \eqref{sq-functionsmall}.

It remains to verify \eqref{Corona-bad-cubes}. To this end, note that 
by the definition of $\F_{bad}$, along with the stopping time criterion \eqref{stoppingbound},
we see that for each $Q_j\in\F_{bad}$,
\begin{multline*}
(a+2b) \sigma(Q_j) < \int_{Q_j}\left(\mathcal{A}^{M,short}_{Q_j}({\bf B_*})^2\,+\,
 \mathcal{H}_{Q_j,Q}^M\right) d\sigma
 \\[4pt]
=  \int_{Q_j}\Big(\sum_{Q'\in\dd_{Q_j}\setminus \{Q_j\}}1_{Q'}(x) \,\xi_{Q'}\,+\,
 \sum_{Q': Q_j\subset Q'\subset Q}1_{Q'}(x) \,\xi_{Q'}\Big) d\sigma
 \\[4pt]
 =\,\int_{Q_j} H_Q^M\,d\sigma \,=\,\int_{Q_j} \mathcal{A}^M_{Q}({\bf B_*})^2\,d\sigma\,.
\end{multline*}
Summing over $Q_j\in\F_{bad}$, we obtain
\[
(a+2b)\sum_{\F_{bad}}\sigma(Q_j) \leq \sum_{\F} 
\int_{Q_j} \mathcal{A}^M_{Q}({\bf B_*})^2\,d\sigma
\leq \int_{Q} \mathcal{A}^M_{Q}({\bf B_*})^2\,d\sigma\leq (a+b)\sigma(Q)\,,
\]
by hypothesis, and \eqref{Corona-bad-cubes} follows.
\end{proof}

The next lemma provides control of a square function in a sawtooth region, in terms of a dyadic
square function in the original domain.  If $\Omega_{\sbf}$ denotes the sawtooth domain, set
\begin{equation}\label{eqsawtoothconedef}
\Gamma_\star(z) := \{Y\in\Omega_\sbf: |z-Y|< 10\delta_\star(Y)\}\,,\quad z\in\pom_{\sbf}\,,
\end{equation}
where $\delta_\star(Y):= \dist(Y,\pom_{\sbf})$, 
and define the corresponding conical square function
\begin{equation}\label{Astardef}
\mathcal{A}_\star(\Bb)^2 (z) := \iint_{\Gamma_\star(z)}|\Bb(Y)|^2\, \delta_\star^{1-n}(Y)\, dY\,.
\end{equation}

Let $M$ denote the parameter in the definition of the dyadic conical square function, as above.
\begin{lemma}\label{lemmaAstarbound} 
There is a choice of $M$ large enough, depending only on the 
allowable parameters, such that for every $Q_0\in\dd(\pom)$, 
and for any tree $\sbf$ with top cube
$Q(\sbf)=Q_0$, we have
\begin{equation}\label{squaretosquare}
\esssup_{z\in \pom_{\sbf}}\mathcal{A}_\star(\Bb)(z) \,\lesssim \,
\|\mathcal{A}_{\sbf}^M(\Bb)\|_{L^\infty(Q_0)}\,+\esssup_{Y\in\Omega_{\sbf}} \delta(Y)|\Bb(Y)|\,,
\end{equation}
where the implicit constant also depends only on the 
allowable parameters, and where 
$\Omega_{\sbf}=\interior\left(\cup_{Q\in \sbf} \,U_Q\right)$
is the usual ``sawtooth domain" associated to $\sbf$.
\end{lemma}

\begin{proof}
% Since $\mathcal{A}_\star(\Bb)$ is continuous on $\pom_{\sbf}$, it is enough to verify
% \eqref{squaretosquare} for $\sigma_\star$-a.e. $z\in\pom_{\sbf}$, where
% $\sigma_\star:=\mathcal{H}^n\lfloor_{\pom_{\sbf}}$ denotes the usual surface measure on
% $\pom_{\sbf}$.

We note that since $\Omega_{\sbf}\subset\Omega$, we have the trivial inequality
\[
\delta_\star(Y) \leq \delta(Y)\,,
\]
which we shall use repeatedly without further explanation.  We further note that by
Proposition \ref{sawtoothprop} (7), 
\[
\delta(Y) \leq N \ell(Q_0)\,,\qquad \forall \, Y\in R_{Q_0}\,, % \Omega_{\sbf}\,,
\]
for some uniform constant $N$, hence in particular for all $Y\in \Omega_{\sbf}$.

Suppose first that $z\in \pom\cap\pom_{\sbf}$.  
Let $\F:=\{Q_j\}_j$ denote the collection of sub-cubes of $Q_0$ that are maximal with 
respect to the property that $Q_j\notin\sbf$.  Observe that  
$Q_0\setminus (\cup_{\F} Q_j) \subset \pom\cap\pom_{\sbf}$, and
$\pom\cap\pom_{\sbf} \setminus \left(Q_0\setminus (\cup_{\F} Q_j)\right)$ 
has $\mathcal{H}^n$ measure zero, by Propositions \ref{prop:sawtooth-contain} and \ref{prop5.0a}.
Thus, we may assume that $z\in Q_0\setminus (\cup_{\F} Q_j)$.
Let $Y\in\Gamma_\star(z)$ so that 
\begin{equation}\label{deltacompare}
\delta(Y)\leq |z-Y|<10\delta_\star(Y) \leq 10\delta(Y)
\end{equation}
(where in the first inequality we have used that $z\in\pom$).
Choose $Q\subset Q_0$, with $z\in Q$, and $\delta(Y)\leq N\ell(Q)< 2\delta(Y)$.
Then $Y\in U_{Q}^M$ for $M$ chosen large enough (depending on $N$).  Since there is such a cube $Q$ for each $Y\in \Gamma_\star(z)$, we see that
$\Gamma_\star(z)\subset \Upsilon^M_{Q_0}(z)$.  Moreover, $\delta_\star(Y) \approx\delta(Y)$,
by \eqref{deltacompare}.  Consequently,  
\[
\mathcal{A}_\star(\Bb)(z) \,\lesssim \, \mathfrak{A}_{Q_0}^M({\Bb})(z)\,\leq\,
\mathcal{A}_{\sbf}^M(\Bb)(z)\,,
\] 
as desired, since $z\in Q_0\setminus (\cup_{\F} Q_j)$ implies that
$Q\in\sbf$ for every $Q$ with $z\in Q\subset Q_0$.

Next, suppose that $z\in \pom_{\sbf}\setminus \pom$.  Then 
$z\in \pom_{\sbf}\cap\Omega$, so $\delta(z) >0$.  Fix $Z_1\in\Omega_{\sbf}$ such that
$|Z_1-z|<\delta(z)/2000$.  Then by definition of $\Omega_{\sbf}$, 
$Z_1\in U_{Q_1}$ for some $Q_1\in\sbf$, hence
\begin{equation}\label{zQ1compare}
\ell(Q_1)\approx\delta(Z_1)\approx \delta(z)\approx \dist(z,Q_1)\,.
\end{equation}

Let $Y\in\Gamma_\star(z)$.  We consider three cases. Let us 
first summarize some geometric information in each case, which we will then use
to deduce \eqref{squaretosquare}.

\noindent{\bf Case 1}: $|z-Y|<\delta(z)/2$.  Then $\delta(Y)\approx \delta(z)\approx
\ell(Q_1)$, by \eqref{zQ1compare}.

\noindent{\bf Case 2}: 
$|z-Y|\geq \delta(z)/2$, and $\delta(Y) \leq 4\delta(z)$.  By definition of $\Gamma_\star(z)$, this means that
\[
\frac18\delta(Y)\leq \frac12\delta(z) \leq |z-Y| <10\delta_\star(Y) \leq 10\delta(Y)\,.
\]
Thus, $\delta(Y) \approx\delta_\star(Y)$, and 
$|z-Y|\approx \delta(z)\approx \delta(Y)$, 
hence $Y\in U_{Q_1}^M$ for $M$ large enough.

\noindent{\bf Case 3}: 
$|z-Y|\geq \delta(z)/2$, and $\delta(Y) > 4\delta(z)$.
In this case, $\delta(Y) <\frac43|z-Y|$, and therefore
\[
 \delta(Y) \lesssim |z-Y| < 10\delta_\star(Y) \leq 10\delta(Y)\,,
\]
since $Y\in\Gamma_\star(z)$.  In particular, $\delta(Y) \approx\delta_\star(Y)$.
Choose $Q$ such that $Q_1\subset Q\subset Q_0$, with
$\delta(Y)\leq N\ell(Q)< 2\delta(Y)$.  Then using \eqref{zQ1compare}, we see that
\[
\delta(Y)\leq \dist(Y,Q)\leq \dist(Y,Q_1)\leq |Y-z| + \dist(z,Q_1) + \diam(Q_1)
\lesssim_N\, \delta(Y)\,.
\]
Thus, $Y\in U_Q^M$ for $M$ chosen large enough.

We now write
\[
\Gamma_\star(z) = \Gamma_1(z) \cup \Gamma_2(z) \cup \Gamma_3(z)\,,
\]
where $\Gamma_1, \Gamma_2$ and $\Gamma_3$ are the subsets of 
$\Gamma_\star$ corresponding to Cases 1, 2 and 3, respectively.  This induces a corresponding splitting
\[
\mathcal{A}_\star(\Bb)^2 (z) = \sum_{i=1}^3 \mathcal{A}_i(\Bb)^2 (z)
:= \sum_{i=1}^3 \iint_{\Gamma_i(z)}|\Bb(Y)|^2\, \delta_\star^{1-n}(Y)\, dY\,.
\]

Observe that for $Y\in \Gamma_2(z) \cup \Gamma_3(z)$ (thus, as in Case 2 or Case 3), 
we have shown that $\delta_\star(Y) \approx \delta(Y)$ and that $Y\in U_Q^M$ for some
$Q$ with $Q_1\subset Q\subset Q_0$, where $Q_1\in \sbf$.  Consequently,
for any $x\in Q_1$,
\begin{multline*}
\mathcal{A}_2(\Bb)^2 (z) + \mathcal{A}_3(\Bb)^2 (z)
\\[4pt]
\lesssim \sum_{Q: Q_1\subset Q\subset Q_0} \iint_{U_Q^M}|\Bb(Y)|^2\, \delta^{1-n}(Y)\, dY
\,\leq \,
 \mathcal{A}^M_{\sbf}(\Bb)^2 (x)
\, \leq \,\|\mathcal{A}_{\sbf}^M(\Bb)\|^2_{L^\infty(Q_0)}\,, % \,,\quad \forall\, x\in Q_1\,.
\end{multline*}
as desired.

It remains to consider $\mathcal{A}_1(\Bb) (z)$.  
Set $\Xi:= \esssup_{Y\in\Omega_{\sbf}} \delta(Y)|\Bb(Y)|$.
In Case 1, we have $|z-Y|<\delta(z)/2$, and $\delta(Y)\approx \delta(z)\approx
\ell(Q_1)$.  Moreover,
for $Y\in \Gamma_\star(z)$, we have $|z-Y|\approx \delta_\star(Y)$.
Therefore,
\begin{multline*}
\mathcal{A}_1(\Bb)^2 (z)\, =\, 
\iint_{\Gamma_\star(z)\cap\{|z-Y|<\,\delta(z)/2\}}|\Bb(Y)|^2\, \delta_\star^{1-n}(Y)\, dY
\\[4pt]
\lesssim \, \frac{\Xi^2}{\ell(Q_1)^2}\iint_{|z-Y|\,\lesssim \,\ell(Q_1)} |z-Y|^{1-n}\,dY
\,\lesssim\,\Xi^2\,,
\end{multline*}
since our ambient dimension is $n+1$.  We conclude by summing our estimates for 
$\mathcal{A}_i(\Bb)(z)$, $ i = 1,2,3$.
\end{proof}

In the following remark and corollary, we shall continue to use the same notation as in Lemma \ref{lemmaAstarbound}.
% In addition, we fix a point $X_0\in U_{Q_0}$ such that 
% $\delta(X_0) \approx \delta_{\star}(X_0) \approx \ell(Q_0)$, 
% so that $X_0$ is a corkscrew point for $\Omega_\sbf$ at the scale $r_0:= \diam \Omega_{\sbf}
% \approx \ell(Q_0)$.

\begin{remark}\label{x0fix} Fix a cube $Q_0\in\dd(\pom)$, let $\sbf$ be a tree with top cube
$Q(\sbf)=Q_0$, and consider the associated sawtooth domain $\Omega_\sbf$ and Carleson region $R_{Q_0}$ as in Proposition \ref{sawtoothprop}.
By the properties of the Whitney regions $U_Q$, and of the sawtooth and Carleson regions,
% domain $\Omega_\sbf$,
we may (and do) fix a point $X_0\in U_{Q_0}\subset \Omega_\sbf \subset R_{Q_0}$, 
such that 
% \begin{multline*}
\[\dist(X_0,\partial U_{Q_0})\approx 
\delta_\star(X_0) 
\approx \dist(X_0,\partial R_{Q_0}) \approx \delta(X_0)
% \\[4pt]
 \approx \ell(Q_0)
\approx \diam(\Omega_\sbf)\approx \diam(R_{Q_0})\,.
\]
% \end{multline*}
Thus, $X_0$ is a corkscrew point for $\Omega_\sbf$ at the 
scale $r_0:= \diam (\Omega_{\sbf})\approx \ell(Q_0)$, with respect to any center
$z\in\pom_{\sbf}$.
\end{remark}  

With the point $X_0$ fixed, we let $\hm_{0,\star}:= 
\hm^{X_0}_{L_0,\Omega_{\sbf}}$ and $\mathcal{G}_{0,\star} := G_{L_0,\Omega_{\sbf}} (X_0,\cdot)$
denote the elliptic measure and Green function for the operator $L_0=-\dv A\nabla$, in the domain
$\Omega_{\sbf}$, with pole at $X_0$.

With $M$ now fixed as in Lemma \ref{lemmaAstarbound}, we shall simply write
$\mathcal{A}_{\sbf}(\Bb):= \mathcal{A}_{\sbf}^M(\Bb)$.

\begin{corollary}\label{corsawtooth}
Let $Q_0$, $\sbf$, and $\Omega_{\sbf}$ be as in Lemma \ref{lemmaAstarbound}, 
and suppose that $\Bb$ is truncated in $\Omega$, 
as in Subsection \ref{approxhm}.
Let $X_0$ be the corkscrew point fixed in Remark \ref{x0fix}, and let $\hm_{0,\star}$ and $\mathcal{G}_{0,\star}$ be the elliptic measure and Green function as described above.
Suppose that 
\begin{equation}\label{eqbetamall0}
\|\mathcal{A}_{\sbf}(\Bb)\|_{L^\infty(Q_0)}\,+\esssup_{Y\in\Omega_{\sbf}} \delta(Y)|\Bb(Y)|\,\leq \,\beta\,,
\end{equation}
for some $\beta>0$.  Then
\begin{equation}\label{eqbetamall1}
|\Bb(Y)| \leq \frac{\beta}{\delta_\star(Y)}\,,\quad \text{\em a.e. } Y\in\Omega_{\sbf}\,,
\end{equation}
and for every $\Delta_\star=B\cap\pom_{\sbf}$, 
 with $B=B(z,r)$, $z\in\pom_{\sbf}$, $0<r<r_0$,
 \begin{equation}\label{eqbetamall2}
 \iint_{B\cap\Omega_{\sbf}} |\Bb(Y)|^2\, \mathcal{G}_{0,\star}(Y) \,dY\, \leq\, C\beta^2 \hm_{0,\star}(\Delta_\star)\,,
 \end{equation}
where $C\geq 1$ depends only on the allowable parameters;
i.e., $\Bb$ is ``$\sqrt{C}\beta$-small with respect to $\hm_{0,\star}$" in the sense of Definition \ref{epsmall}, in the domain $\Omega_{\sbf}$.
\end{corollary}

\begin{proof}
Note that since $\Bb$ has been truncated in $\Omega$
as in Subsection \ref{approxhm}, and hence is qualitatively bounded,
then by construction, 
$\mathcal{A}_\star(\Bb)$ is continuous on $\pom_{\sbf}$ (see \eqref{Astardef}).
Consequently, \eqref{squaretosquare} self-improves to give
\begin{equation}\label{squaretosquare2}
\sup_{z\in \pom_{\sbf}}\mathcal{A}_\star(\Bb)(z) \,\lesssim \,
\|\mathcal{A}_{\sbf}(\Bb)\|_{L^\infty(Q_0)}\,+\esssup_{Y\in\Omega_{\sbf}} \delta(Y)|\Bb(Y)|
\,\lesssim \beta\,,
\end{equation}
i.e., the essential supremum on the left hand side has now been replaced by a supremum; the upper bound on the right hand side is simply the hypothesis \eqref{eqbetamall0}.

The bound \eqref{eqbetamall1} is trivial, since $\delta_\star(Y) \leq \delta(Y)$ for all $Y\in \Omega_{\sbf}$.

To establish \eqref{eqbetamall2}, we write 
 \begin{multline*}
 \iint_{B\cap\Omega_{\sbf}} |\Bb(Y)|^2\, \mathcal{G}_{0,\star}(Y) \,dY
 \\[4pt]
 =\,
 \iint_{B\cap\Omega_{\sbf}\cap\{|X_0-Y|>\delta(Y)/4\}} ...\, dY +
  \iint_{B\cap\Omega_{\sbf}\cap \{|X_0-Y|\leq\delta(Y)/4\}} ...\, dY\,=: I + II\,.
 \end{multline*}
 (It may be that $II$ is vacuous, unless $B$ has radius $r\approx r_0$).
 By standard Green function estimates in $\ree$, and \eqref{eqbetamall0},
 \[
 II\,\lesssim \, \frac{\beta^2}{r_0^2} \iint_{|X_0-Y|\leq r_0} |X_0-Y|^{1-n} dY \,\lesssim \beta^2\,,
 \]
since $r_0:= \diam (\Omega_{\sbf}) \approx \delta(X_0)\approx \delta(Y)$
in the regime $|X_0-Y|\leq\delta(Y)/4$.
Furthermore, by Lemma \ref{l2.10} applied in $\Omega_{\sbf}$, 
setting $\Delta_{\star,Y}:= B(Y,10\delta_\star(Y))\cap\pom_{\sbf}$, and using
\eqref{eqsawtoothconedef} and \eqref{Astardef}, we see that
 \begin{multline*}
I \lesssim  \iint_{B\cap\Omega_{\sbf}} 
|\Bb(Y)|^2\, \hm_{0,\star}(\Delta_{\star,Y})\, \delta_*^{1-n}(Y) \,dY
\\[4pt]
=\,
\int_{\pom_{\sbf}} \iint_{\Omega_{\sbf}\cap\{|z-Y|<10\delta_\star(Y)\}}
 |\Bb(Y)|^2 \, \delta_*^{1-n}(Y) \, dY \, d\hm_{0,\star} (z)
 \\[4pt]
 =\, 
 \int_{\pom_{\sbf}} \mathcal{A}_\star(\Bb)^2 (z) \,d\hm_{0,\star} (z)\,\lesssim\, \beta^2\,,
\end{multline*}
by \eqref{squaretosquare2}, % Lemma \ref{lemmaAstarbound}, 
since $\hm_{0,\star}$ is a probability measure.
\end{proof}

\section{Proof of Theorem \ref{Tmain}:  the extrapolation argument}
\label{Sgeneral}

We now arrive at the heart of the matter.  
In this section, we give the proof of Theorem \ref{Tmain}.  As mentioned above in Section \ref{section2key}, our approach will be based on the method of ``extrapolation of Carleson measures".

We first require some preliminaries.  For $X\in\Omega$, let $\hm^X$ denote the 
elliptic measure at $X$ for the operator 
$L=-\dv A\nabla +\Bb\cdot \nabla$ in $\Omega$.

An important tool for us will be 
the following result of \cite{BL}, which provides a (necessary and sufficient)
criterion for elliptic measure to belong to
weak-$A_\infty$
(recall that the latter property yields
the conclusion of Theorem \ref{Tmain};
see Remark \ref {Dp=Ainfty}).  We recall the notation of \eqref{DXdef}:  for $X\in\Omega$, we set
$\Delta_X:=B\big(X,10\delta(X)\big)\cap\pom$.

% In the next lemma, and in the sequel,we use the notation
% \begin{equation} % \label{DXdef}
% \Delta_X:=B\big(X,10\delta(X)\big)\cap\pom\,,\qquad X\in\Omega\,.
% \end{equation}

\begin{lemma}[\cite{BL}]\label{BLlemma}   Let $L:=-\dv A\nabla +\Bb\cdot\nabla$, 
where $\Bb$ satisfies \eqref{driftsize}, and suppose that the continuous Dirichlet problem
(Definition \ref{Dc})
is solvable for $L$, and that the solutions satisfy the weak maximum principle (thus, elliptic measure exists).
For $X\in\om$, define
$\Delta_X$ as in \eqref{DXdef}.
Assume that $\pom$ is ADR, and that there are uniform
positive constants $\eta$ and $\gamma$
such that for each $X\in\om$ with $\delta(X)< 2 \diam(\pom)$,
the following implication holds for the
Borel subsets of  $\Delta_X$:
\begin{equation}\label{eq3.1b}
E\subset \Delta_X\,,\quad
\sigma(E)\geq (1-\eta) \sigma(\Delta_X)  \, \, \implies \,\,  \hm^X(E) \geq \gamma\,.
\end{equation}
Then elliptic measure for $L$ is locally in weak-$A_\infty$ in the sense of Definition
\ref{deflocalAinfty}, with uniform constants depending only on $\eta$, $\gamma$, 
$n$, $\lambda$, $\Lambda$, ADR, and the constant in \eqref{driftsize}.  
% In particular, the weak-$A_\infty$ constants do {\tt not} depend upon the qualitative 
% $L^\infty$ assumption, nor, in fact, on the corkscrew and Harnack chain conditions.
\end{lemma}

\begin{remark}\label{remarkunbound}
Clearly, the assumption $\delta(X)< 2 \diam(\pom)$ imposes a meaningful restriction 
only in the case that $\Omega$ is unbounded and $\pom$ is bounded.
\end{remark}

Although not stated explicitly in this form, Lemma \ref{BLlemma} is proved in \cite{BL}:  it follows from the
combination of \cite[Lemma 2.2]{BL} and its proof,
and \cite[Lemma 3.1]{BL}.  The results in \cite{BL} are stated only for the case that
$L$ is the Laplacian, but in fact those results may be applied much more generally:  the proof in \cite{BL} requires only Harnack's inequality and the weak maximum principle. 
 In our setting, we recall that Harnack's inequality is available to us 
 by Lemma \ref{Moser}, and 
 if we truncate $\Bb$ as in Subsection \ref{approxhm}, then
 the weak maximum principle, solvability of the continuous Dirichlet problem, and existence of elliptic measure are also available, by Lemma \ref{hmkexists}.  
 
 Thus, we consider truncations $\Bb_k$ and the corresponding operators $L_k$
as in Subsection \ref{approxhm}. 
We shall show that \eqref{eq3.1b}
holds, with the constants $\eta$ and $\gamma$ uniform in $k$ (that is, independent of the truncation), depending only on allowable parameters.   In particular, taking $E=\Delta_X$ 
in \eqref{eq3.1b}, we see that  $\hm^X_k(\Delta_X) \geq \gamma$, uniformly in $k$, and 
therefore for any $x\in \pom$, $r<\diam{\pom}$, and for all
$Y\in \Omega \cap B(x,r/2)$, we also have $\omega_k^{Y} (\Delta(x,r)) \geq c\gamma$,
again uniformly in $k$.
Thus, by a standard iteration argument,
the conclusion of Lemma \ref{proppde} holds for $L_k$ (even without
 \eqref{smallptwise}), uniformly in $k$.  
We may therefore apply Lemma \ref{hmklimit} to deduce that
$\hm^X$, the elliptic measure for $L$ exists, and also satisfies \eqref{eq3.1b}.
By Lemma \ref{BLlemma}, 
 the weak-$A_\infty$ property then holds for $\hm$, whence 
the conclusion of Theorem \ref{Tmain} follows, by
Remark \ref {Dp=Ainfty}.

% corresponding to the truncated operator $L_k$, with constants that 
% are uniform in $k$ and depend only on allowable parameters. 

% In our setting, as discussed above, we may assume, by a limiting argument, 
% that $\Bb$ is bounded qualitatively, to ensure {\it a priori} existence of the elliptic 
% measure.  The quantitive bounds that we prove are uniform, 
% and do not depend on $\|\Bb\|_{L^\infty}$.

In the sequel, we henceforth implicitly assume that $\Bb$ has 
been truncated near the boundary, but we  
suppress the subscript $k$, for simplicity of notation.  Our quantitative bounds will
depend only on allowable parameters, and we may then pass to the limit as described above.

% To set the stage for the induction procedure, 
% let us begin by making some preliminary reductions.
By Lemma \ref{BLlemma}, 
 it suffices to show that there are uniform positive constants $\eta$ and $\gamma$
such that for each $X\in\om$ with $\delta(X)<2 \diam(\pom)$ 
(see Remark \ref{remarkunbound}), and for
$\Delta_X$ as in \eqref{DXdef}, we have the implication \eqref{eq3.1b}
for any Borel subset $E\subset \Delta_X$.
% \begin{equation}\label{eq5.1}
% \sigma(E)\geq (1-\eta) \sigma(\Delta_X)  \, \, \implies \,\,  \hm^X(E) \geq \gamma\,.
% \end{equation}

We begin by reducing matters to a dyadic version of \eqref{eq3.1b}. % \eqref{eq5.1}.
To this end,
 let $X\in \om$, %and as in Definition \ref{def1.johnpoint},
%set $\Delta_X:= \Delta(\hat{x},\delta(X))$,
and let $\hat{x}\in\pom$ be a touching point for $X$, i.e., $|X-\hat{x}|=\delta(X)$. 
Choose a cube $Q\in \dd(\pom)$, with $\hat{x}\in Q$, and such that 
$\ell(Q)\approx \delta(X)$ and $Q\subset \Delta_X$.
By ADR, $\sigma(\Delta_X)\approx \sigma(Q)$, so 
matters are reduced to verifying that there are uniform positive constants
$\eta$ and $\gamma$ such that for Borel subsets of any cube $Q\in \dd(\pom)$,
\begin{equation}\label{eq5.1a}
E\subset Q\,,\quad \sigma(E)\geq (1-\eta) \sigma(Q)  \, \, \implies \,\,  \hm^X(E) \geq \gamma\,.
\end{equation}
Indeed, \eqref{eq5.1a} for some $\eta$ and $\gamma$, with $Q\subset \Delta_X$ 
as above, implies \eqref{eq3.1b} % \eqref{eq5.1} 
for a somewhat smaller $\eta$ (on the order of $c\eta$, with $c=c(n,ADR)$), 
and for the same $\gamma$.

\begin{remark}\label{XinUQ}
In the construction in \cite{HM-I} of the Whitney regions $U_Q$ (see Proposition \ref{sawtoothprop}), one may choose the associated parameters in such a way
that for any $X\in \Omega$, and for $Q\subset \Delta_X$ with $\ell(Q)\approx \delta(X)$,
as in the preceding
discussion, we have $X\in U_Q$.
\end{remark}

Set $\mathcal{A}_{Q}(\cdot):= \mathcal{A}^M_{Q}(\cdot)$, where the latter
 is defined in \eqref{dy-conesqfunction}, and $M$ has been fixed as in Lemma \ref{lemmaAstarbound}.
 Set $\Theta(Q):= \fint_Q \mathcal{A}_{Q}({\bf B_*})^2 d\sigma$, and
recall that by \eqref{eqAMCM}, 
\begin{equation}\label{ThetaQdef}
\sup_{Q\in \dd}\Theta(Q)\,:= \,\sup_{Q\in \dd}\,\fint_Q \mathcal{A}_{Q}({\bf B_*})^2 
d\sigma \leq M_2\,,
\end{equation}
where $M_2 = CM_1$, with $M_1$ as in \eqref{driftmax}, and with $C$ depending only on allowable parameters.

The proof will proceed by induction. 
The inductive parameter will be
a number $a\in[0,M_2)$.
% For the sake of notational convenience, we set
% \begin{equation}\label{ThetaQdef}
% \Theta(Q)\,:= \,\fint_Q \mathcal{A}_{Q}({\bf B_*})^2 d\sigma\,.
% \end{equation}

\medskip

\null\hskip.1cm \fbox{\rule[8pt]{0pt}{0pt}$H[a]$}\hskip4pt \fbox{\ \parbox[c]{.75\textwidth}{%
%\smallskip
%
%\noindent {\bf [H(a,\theta)].} 
\rule[10pt]{0pt}{0pt}\it  
There exist positive numbers $\gamma_{a}$ and $\eta_{a}$ such that
for any given $Q\in\dd(\pom)$,  if
\begin{equation}\label{Mabound}
% \!\!\!\!\!\!\!\!\!\!\!\!\!\!\!\!\!\!\!\!\!\!\!\!\!\!\!\!\!\!\!\!\!\!\mut(\dd_Q)
\!\!\!\!\!\!\!\!\!\!\!\!\!\!\!\!\!\!\!\!\!\!\!\!\!\!\!\!\!\!\!\!\!\!\!\!\!\!\!\!\!\! % \fint_Q \mathcal{A}_{Q}({\bf B_*})^2 d\sigma 
\Theta(Q)\,\le \, a \,,
\end{equation} 
and if $E\subset Q$ is a Borel set, then 
\begin{equation}\label{eq5.8}
\!\!\!\!\!\!\!\! \!\!\!\sigma(E)\geq (1-\eta_{a}) \sigma(Q)\ \   
\implies\ \  \inf_{X\in U_Q} \hm^X(E)
\geq \gamma_{a}\,,
\end{equation}
} }

\medskip

Our proof strategy is as follows.  We first verify $H[0]$ (this will be easy),
and then show that there is a uniform number $b>0$ such that 
$H[a]\!\!\! \implies \!\!\!H[a+b]$.  After iterating this procedure 
on the order of $M_1/b\approx M_2/b$ times, we then 
find that $H[M_2]$ holds, hence by
\eqref{ThetaQdef}, we see that the conclusion \eqref{eq5.1a} holds for all $Q\in\dd(\pom)$
with $\eta=\eta_{M_2}$, and with $\gamma =\gamma_{M_2}$, for all $X\in U_Q$.  Consequently, as observed above,
using Remark \ref{XinUQ}, we find that \eqref{eq3.1b} % \eqref{eq5.1} 
holds for all $X\in \Omega$ (for a somewhat smaller, but still fixed and uniform value of $\eta$).  In turn, Lemma \ref{BLlemma} then yields the desired weak-$A_{\infty}$ property.

Thus, it remains to verify $H[0]$, and that
$H[a] \!\!\implies \!\!H[a+b]$, for each $a\in [0,M_2)$, and for some uniform $b>0$.

% \steve{\Bl Note to myself:  when defining $\Delta_Q$, set parameters so that 
% $2\Delta_Q \subset Q$.  Or maybe $M\Delta_Q$ for $M$ 
% large enough that $T_{\Delta_Q}\subset MB_Q$?}

% \steve{\Bl do I want to make $R_Q$ open (i.e., take interior)?}

% To this end, let us first make the following observation.

\smallskip
\noindent{\em Proof of} $H[0]$.  
Fix $Q_0$ such that
$\Theta(Q_0) = 0$. Then by definition, $\Bb_* \equiv 0$ in the ``Carleson box"
$R_{Q_0}$. 
Consequently, $L:= L_0$ in $R_{Q_0}$.  Recall (see Proposition \ref{sawtoothprop} (6)), that
 there is a ball $B_{Q_0}=(x_{Q_0},r_{Q_0})$, with $r_{Q_0}\approx \ell({Q_0})$, such that
 $2B_{Q_0}\cap\Omega\subset R_{Q_0}$, 
 where as usual $2B_{Q_0}$ denotes the double of $B_{Q_0}$. 
 Let $X_0$ be the corkscrew point fixed in Remark \ref{x0fix}, and
 let $ \hm^{X_0}_{L_0}$ and $\widetilde{\hm}^{X_0}_{L_0}$ denote elliptic measure
 at the point $X_0$ for $L_0$, in the domains $\Omega$ and $R_{Q_0}$ respectively.
Similarly, let $G_{L_0}$ denote the Green function for $L_0$ in $\Omega$, and let
 $\widetilde{G}_{L_0}$ denote the Green function for $L_0$ in the domain $R_{Q_0}$.  
 By the usual comparison principle, 
 aka ``boundary Harnack principle" (i.e., Lemma \ref{lemmaBHP}), 
along with Harnack's inequality and the Harnack Chain property,
 \[
 G_{L_0}(X_0,Y) \approx \widetilde{G}_{L_0}(X_0,Y)\,, 
 \quad \forall  \, Y\in  \frac32 B_{Q_0}\cap\Omega\,.
 \]
 Hence, for any Borel set
 $F\subset \Delta_{Q_0}:= B_{Q_0}\cap\pom$, by the CFMS estimates 
 for $L_0$ (Lemma \ref{l2.10})
 we have
 \begin{equation}\label{eqhmcompare}
 \hm^{X_0}_{L_0}(F) \approx  \widetilde{\hm}^{X_0}_{L_0}(F)\,.
 \end{equation}
 
 Let $E\subset Q_0$ with $\sigma(E)>(1-\eta_0)\sigma(Q_0)$.
 If we choose $\eta_0$ small enough, then by the $A_\infty$ property for
$\hm_{L_0}$ and Lemma \ref{Bourgainhm} 
(and Harnack's inequality and the Harnack Chain condition), 
using \eqref{eqhmcompare} we see that
\[
\widetilde{\hm}_{L_0}^{X_0}(E\cap \Delta_{Q_0})
\,\approx\, \hm_{L_0}^{X_0}(E\cap \Delta_{Q_0})\,\gtrsim\, \hm_{L_0}^{X_0}(Q_0)\,
 \gtrsim 1\,. % \qquad \forall \,X\in U_Q\,,
\]
Let $\hm_L$ denote elliptic measure for $L$ in $\Omega$.
Since $L=L_0$ in $R_{Q_0}$, by Harnack's inequality, the Harnack Chain condition, and the maximum principle we have
\[
\inf_{X\in U_{Q_0}}\hm_L^{X} (E)\,\gtrsim\,
\hm_L^{X_0} (E) \,\geq \,\widetilde{\hm}_{L_0}^{X_0}(E\cap \Delta_{Q_0}) \,\gtrsim \,1\,. 
\]
 Thus, $H[0]$ holds.
\hfill\(\Box\)

The next (and main) step is to show that $H[a]$ implies $H[a+b]$, for a sufficiently small (but uniform) choice of $b>0$.  
% To this end, let us first record a few preliminary observations.

\smallskip

\noindent  $H[a]\!\! \implies \!\!H[a+b]$.
Let $a\geq 0$, let $b$ be a sufficiently small positive number to be chosen, and 
consider a cube $Q_0$ such that
\begin{equation}\label{eqQ0}
\Theta(Q_0) \leq (a+b) \,\sigma(Q_0)\,,
\end{equation}
with $\Theta(Q)$ defined as in \eqref{ThetaQdef}.  
We apply Lemma \ref{lemmastoptime} to $Q_0$. 
Then either \eqref{sq-functionshortbound} holds, or both 
 \eqref{sq-functionsmall} and \eqref{Corona-bad-cubes} hold.
 The former case is easy to handle.

 Indeed, if \eqref{sq-functionshortbound} holds, then by definition of $\mathcal{A}_Q^{short}$
 (recall that we have suppressed the parameter $M$, which has now been fixed), there is at least one dyadic child of $Q$, call it $Q'$, such that
 $\Theta(Q') \leq a$.  Consequently, we may apply the induction hypothesis $H[a]$ to $Q'$.
 Suppose that $E\subset Q$, with $\sigma(E)\geq (1-\eta)\sigma(Q)$, where $\eta>0$ 
 is chosen small enough, depending only on dimension and ADR, so that
 $\sigma(E\cap Q')\geq (1-\eta_a)\,\sigma(Q')$.  Applying $H[a]$ in $Q'$, and using Harnack's inequality, we see that
 \[
 \inf_{X\in U_Q} \hm^X(E) \gtrsim  \inf_{X\in U_{Q'}} \hm^X(E\cap Q') \geq \gamma_a\,,
 \]
 thus, we obtain $H[a+b]$ in the present case with $\eta_{a+b} = \eta$, and
 $\gamma_{a+b} = c\gamma_a$, for some $c>0$ depending only on 
 allowable parameters.

In order to treat the main case that both
 \eqref{sq-functionsmall} and \eqref{Corona-bad-cubes} hold, we shall need to
discuss some preliminary matters.

Recall (see Lemma \ref{lemmastoptime}) that $\F =\{Q_j\}_j$ is the family of sub-cubes of
$Q_0$ that are maximal with respect to the property that $Q_j \notin \sbf$.  As in \cite{HM-I},
we define a projection operator $\P_{\F}$ with respect to the 
family $\F$ as follows: given a Borel measure $\ttm$ defined on $\pom$, 
and a Borel subset $F\subset\pom$, we set
\begin{equation}\label{pfdef}
\P_\F\ttm(F):= \ttm\left(F\setminus (\cup_\F Q_j)\right) + 
\sum_\F \frac{\sigma(F\cap Q_j)}{\sigma(Q_j)}\, \ttm(Q_j)\,.
\end{equation}

We note that in particular, for any dyadic cube $Q$ that is {\em not} contained in any 
$Q_j\in\F$, we have
\begin{equation}\label{mu=pmuoncubes}
\P_\F\ttm(Q) = \ttm(Q)\,,\quad Q\not\subset Q_j\,,\, \text{ for any } Q_j\in\F\,.
\end{equation}

\begin{lemma}({\em Essentially}, \cite[Lemma 6.15]{HM-I}) {\em (Dyadic sawtooth lemma for projections).}\label{lemma:DJK-dyadic-proj}
Suppose that $\Omega$ is a 1-sided chord-arc domain, and let
$L_0:=-\dv A\nabla$ be uniformly elliptic in $\Omega$.
%, and that $\Omega$ and its Carleson boxes and sawtooth sub-domains (as well as the Carleson boxes $T_{\td}$ with respect to sawtooths $\Omega_{\F,Q}$) also satisfy the qualitative exterior corkscrew condition.
Fix $Q_0\in\dd(\pom)$,  let $\F=\{Q_j\}\subset \dd_{Q_0}$ be a family of pairwise disjoint dyadic cubes that are maximal with respect to the property that they are do not belong to a stopping time tree $\sbf$ with top cube $Q(\sbf)=Q_0$.
Let $\P_\F$ be the corresponding projection operator, and 
fix $X_0 \in U_{Q_0}$ as in Remark \ref{x0fix}.
Let $\hm_0=\hm_{0}^{X_0}$ and $\hm_{0,\star}=\hm_{0,\star}^{X_0}$ denote the respective
elliptic measures (for the operator $L_0$) for the domains $\Omega$ and $\Omega_{\sbf}$, 
with fixed pole at $X_0$.  
Let $\nu=\nu^{X_0}$ be the measure on $Q_0$ defined by
\begin{equation}\label{defi:nu-bar}
\nu(F)
=\hm_{0,\star}\left(F\setminus(\cup_{\F} Q_j)\right) +
\sum_{Q_j\in\F} \frac{\hm_0(F\cap Q_j)}{\hm_0(Q_j)}\,
\hm_{0,\star}(P_j),
\qquad F\subset Q_0
\end{equation}
where $P_j\subset \pom_{\sbf}$ is the surface ball satisfying \eqref{eqpj}, 
whose existence is guaranteed by Proposition \ref{prop:Pj}
(that is, by \cite[Lemma 6.7]{HM-I}).
Then % $\P_\F\nu$ depends only on $\tom$ and not on $\hm$.  More precisely,
\begin{equation}\label{defi-proj-nubar}
\P_\F \nu(F)
=
\hm_{0,\star}\left(F\setminus(\cup_{ \F} Q_j)\right)+\sum_{Q_j\in\F} \frac{\sigma(F\cap Q_j)}{\sigma(Q_j)}\,
\hm_{0,\star}(P_j),
\qquad F\subset Q_0.
\end{equation}
Moreover, there exists $\rho>0$ 
such that for all  $Q\in\dd_{Q_0}$ and $F\subset Q$, we have
\begin{equation}\label{dyadic-DJK:proj}
\left(\frac{\P_\F \hm_0(F)}{\P_\F \hm_0(Q)}\right)^{\rho}
\lesssim
\frac{\P_\F \nu(F)}{\P_\F \nu(Q)}
\lesssim
\frac{\P_\F \hm_0(F)}{\P_\F \hm_0(Q)}.
\end{equation}
\end{lemma}
\noindent
We remark that in \eqref{defi:nu-bar} and in \eqref{defi-proj-nubar}, we
are using that $Q_0\setminus (\cup_\F Q_j) \subset \pom \cap\pom_\sbf$ 
(see Proposition \ref{prop:sawtooth-contain}), so that 
$\hm_{0,\star}(F\setminus (\cup_\F Q_j))$ makes sense.

The preceeding lemma is a dyadic version of the ``Main Lemma" of \cite{DJK}, from which 
its proof is adapted.  
In the present form, it has essentially been proved:  see \cite[Lemma 6.15]{HM-I}.  
The latter result was stated only for the Laplacian, but the proof carries 
through essentially verbatim to the case
stated above:  one requires only that solutions to the operator $L_0u=0$ satisfy the 
standard estimates in the Lemmata \ref{proppde}, \ref{Bourgainhm}, \ref{lemma2.green}, 
\ref{l2.10}, \ref{lemmadouble}, \ref{lemmaBHP}, and \ref{lemmapolechange}.
We mention that there is an extraneous qualitative hypothesis in  
the statement of \cite[Lemma 6.15]{HM-I}, namely
that $\Omega$ has exterior corkscrew points in a {\em qualitative} sense (i.e., 
only at small scales, with no quantitative limitation on how small the scales might 
be): in \cite{HM-I}, the lemma was 
applied in approximating domains which had the qualitative property, but 
the quantitative estimates were uniform for all the approximating domains, and
did not depend on this qualitative assumption; in any case this hypothesis is not needed in the proof,
as the interested reader may readily verify by consulting \cite{HM-I}.

 We now turn to the main case, in which 
 both \eqref{sq-functionsmall} and \eqref{Corona-bad-cubes} hold.
For $\sbf$ as in Lemma \ref{lemmastoptime}, 
by \eqref{sq-functionsmall} and the definition \eqref{dy-treesqfunction} of 
$\mathcal{A}_{\sbf}(\cdot)=\mathcal{A}^M_{\sbf}(\cdot)$, 
we have that whenever $x\in Q \in \sbf$,
\begin{equation*}
\iint_{U_{Q}^M} |{\Bb_*}(Y)|^2 \delta(Y)^{1-n}\,dY\,\leq\,
\mathcal{A}_{\sbf}({\bf \Bb_*})^2(x)\,\leq 2b\,, %\qquad x\in Q \in \sbf\,.
\end{equation*}
and thus also, by the trivial containment $U_Q\subset U_Q^M$,
\begin{equation}\label{UQbbound}
\iint_{U_{Q}} |{\Bb_*}(Y)|^2 \delta(Y)^{1-n}\,dY\, \leq\, 2b\,,\qquad \forall \, Q\in\sbf\,.
\end{equation}
Since $U_Q$ is a union of fattened Whitney cubes, we may construct a covering of $U_Q$ by a collection $\mathcal{C}$ of
$(n+1)$-dimensional Euclidean cubes, each of side length
$\ell(I) =\ell(Q)/N$, so that $U_Q =\cup_{I\in\mathcal{C}}\,I$, and 
where we fix $N$ large enough that $|Z-Y|<\delta(Y)/4$
for any pair of points $Y,Z\in I$;
here, we have used that
$\delta(Y) \approx \ell(Q)$, for all $Y\in U_Q$.

 With this choice of $N$ now fixed,
using \eqref{UQbbound},
we see that for each such $I$,
\[
\frac1{|I|} \iint_I \,|\delta(Y) \Bb_*(Y)|^2 dY \,\lesssim\, b
\]
In particular, there is some fixed $Y_I\in I$ such that 
$\delta(Y_I) |\Bb_*(Y_I)| \lesssim \sqrt{b}$.  By definition \eqref{bstardef}
of $\Bb_*$, we see that
\[
\delta_\star(Y)|\Bb(Y)|\leq \delta(Y) |\Bb (Y)| \lesssim \sqrt{b}\,,\qquad \text{a.e. } Y\in I\,,
\]
hence for a.e.~$Y\in U_Q$, and hence for a.e.~$Y\in \Omega_{\sbf}$, where as above,
$\delta_\star(Y):=\dist(Y,\pom_{\sbf})$.
Moreover, by \eqref{sq-functionsmall} and Corollary \ref{corsawtooth}, choosing
$b=c\eps^2$ for some small enough but fixed constant $c>0$, 
we see that $\Bb$ is $\eps$-small with respect to $\hm_{0,\star}$ in $\Omega_S$,
in the sense of Definition \ref{epsmall}, where as above
$\hm_{0,\star}:= \hm^{X_0}_{L_0,\Omega_{\sbf}}$ 
denotes the elliptic measure for the operator $L_0=-\dv A\nabla$, in the domain
$\Omega_{\sbf}$, with pole at the fixed point $X_0\in U_{Q_0}$ in Remark \ref{x0fix}.
Thus, by Theorem \ref{Tauxsmall}, applied in the domain $\Omega_{\sbf}$, with
$\eps\leq\eps_0$, we find that 
\begin{equation}\label{sawtoothainfinity}
\hm^{X_0}_{L,\Omega_{\sbf}} =: \hm_{\star}\in A_\infty(\pom_{\sbf}, \hm_{0,\star})\,,
\end{equation}
where $\hm^{X}_{L,\Omega_{\sbf}}$ is elliptic measure for $L=L_0+\Bb\cdot \nabla$ in
the domain $\Omega_{\sbf}$.

\begin{remark}\label{remarkhm} In agreement with 
our previous usage (as in \eqref{sawtoothainfinity} for example),
we formalize some notation to distinguish elliptic measures in different domains, and for different operators.  We let $\hm^X:=\hm^X_{L,\Omega}$, and
$\hm_\star^X:=\hm^X_{L,\Omega_{\sbf}}$ denote elliptic measure for $L$ at the point $X$, in the domains $\Omega$ and $\Omega_{\sbf}$ respectively.  Similarly
$\hm_0^X:=\hm^X_{L_0,\Omega}$, and
$\hm_{0,\star}^X:=\hm^X_{L_0,\Omega_{\sbf}}$ will
denote elliptic measure for $L_0$ at the point $X$, in the domains $\Omega$ and $\Omega_{\sbf}$.
In the special case that $X=X_0$, we simply write
$\hm=\hm^{X_0}$, $\hm_\star = \hm_\star^{X_0}$, 
$\hm_0=\hm_0^{X_0}$ and $\hm_{0,\star} =\hm_{0,\star}^{X_0}$,
where $X_0\in U_{Q_0}$ has been fixed as in Remark \ref{x0fix}.
\end{remark}

Proceeding with the proof of Theorem \ref{Tmain}, we
recall that we are assuming that $H[a]$ holds, 
and we seek to establish the conclusion
of $H[a+b]$, for the cube $Q_0$ satisfying \eqref{eqQ0}, 
where as above, we have now fixed $b=c\eps^2$ 
with $\eps\leq \eps_0$.  

% Recall further that we have applied 
% Lemma \ref{lemmastoptime} to $Q_0$, and it 
% remains to treat the main case that both 
%  \eqref{sq-functionsmall} and \eqref{Corona-bad-cubes} hold.

Suppose then that
$E\subset Q_0$, such that
\begin{equation}\label{eqEample}
\sigma(E)\geq (1-\eta)\sigma(Q_0)\,.
\end{equation}

As above, let $X_0\in U_{Q_0}$ be as in Remark \ref{x0fix}.
Our goal is to show that there is a uniform number $\gamma>0$ such that
\begin{equation}\label{BLforhm}
\hm(E):=\hm^{X_0}(E)\geq \gamma\,,
\end{equation}
provided that
$\eta =\eta_{a+b}>0$ is small enough.  By the Harnack chain condition and Harnack's inequality, \eqref{BLforhm} immediately implies the analogous estimate with
$X_0$ replaced by an arbitrary $X\in U_{Q_0}$, and with $\gamma$ replaced by 
$c\gamma=:\gamma_{a+b}$, for some uniform constant $c$, and $H[a+b]$ will then follow.

We now turn to the proof of \eqref{BLforhm}.  Recall that we have applied Lemma 
\ref{lemmastoptime} to $Q_0$, so that, by \eqref{Corona-bad-cubes}, since $a<M_2$
(see \eqref{ThetaQdef}),
\begin{equation*}
\sigma \bigg(\bigcup_{\F_{bad}}Q_j\bigg)
\leq \frac{a+b}{a+2b}\, \sigma(Q_0)\,\leq \,\frac{M_2+b}{M_2+2b}\, \sigma(Q_0)\,=:
(1-\theta)\,\sigma(Q_0)\,.
\end{equation*}
Consequently, either

\smallskip
\noindent {\bf Case 1}:
\begin{equation}\label{case1}
 \sigma \big(\cup_{\F_{good}}Q_j\big) \geq \frac{\theta}{2} \sigma (Q_0)\,,
  % \leqno(\bf \text{\bf Case } 1)
\end{equation}
where $\F_{good}:= \F \setminus \F_{bad}$, and $\theta = \theta(M_2,b)$ is a uniform positive number, or

\smallskip
\noindent {\bf Case 2}:
\begin{equation}\label{case2}
 \sigma \left(Q_0\setminus \left(\cup_\F Q_j\right) \right) \geq \frac{\theta}{2} \sigma (Q_0)\,.
 % \leqno(\bf \text{\bf Case } 2)
\end{equation}

\noindent{\em Proof of $H[a] \!\!\implies \!\!H[a+b]$ in Case 1}.
 Suppose that \eqref{eqEample} and \eqref{case1} hold.  Note that by definition,
\[
 \F_{good}:=
\{Q_j\in\F:\,\fint_{Q_j}\mathcal{A}^{short}_{Q_j}({\bf B_*})^2d\sigma\leq\,a\}\,.
\]
Then for $\eta>0$ small enough, by pigeon-holing and the definition of
$\mathcal{A}^{short}_{Q_j}$ (see \eqref{dy-conesqfunctionshort}), 
there is an ``extra good" subset 
$\F_{eg}\subset \F_{good}$ such that there is a child $Q_j'$ of each $Q_j\in\F_{eg}$,
satisfying the following three estimates:
\begin{equation}\label{feg}
 \sigma \big(\cup_{Q_j\in \F_{eg}}Q'_j\big) \gtrsim_{a,\theta,\eta} \sigma(Q_0)\,,
\end{equation}
\begin{equation}\label{feg2}
\fint_{Q'_j}\mathcal{A}_{Q'_j}({\bf B_*})^2d\sigma \, \leq \, a\,,
\end{equation}
and 
\begin{equation}\label{feg3}
\sigma(E\cap Q'_j) \geq (1-\eta_a) \sigma(Q_j')\,.
\end{equation}
Note that we may therefore apply $H[a]$ to such $Q'_j$, to deduce that
\begin{equation}\label{feg4}
\inf_{X\in U_{Q'_j}}\hm^X(E\cap Q'_j) \geq \gamma_a\,.
\end{equation}

By a covering lemma argument, we may extract a further
``extra good separated" subcollection $\F_{egs}\subset \F_{eg}$, such that
for some sufficiently large constant  $N_0\gg 1$ to be chosen momentarily (depending as usual only on allowable parameters),
\begin{equation}\label{eqseparate}
\dist(Q_{j_1},Q_{j_2}) \geq N_0 \max(\diam Q_{j_1},\diam Q_{j_2})
\end{equation}
for distinct $Q_{j_1},Q_{j_2} \in \F_{egs}$, and
\begin{equation}\label{eqegs}
\sum_{Q_j\in \F_{egs}} \sigma(Q_j) =
 \sigma \left(\cup_{Q_j\in \F_{egs}}Q_j\right) 
 \geq
\sum_{Q_j\in \F_{egs}} \sigma(Q_j') =
 \sigma \left(\cup_{Q_j\in \F_{egs}}Q'_j\right) \gtrsim_{a,\theta,\eta,N_0} \sigma(Q_0)\,,
\end{equation}
where as above, $Q_j'$ is a child of $Q_j$ satisfying \eqref{feg2} and \eqref{feg3}.

We recall the surface balls $P_j\subset \pom_\sbf$ given in Proposition \ref{prop:Pj}. 
Note that we may (and do) choose $N_0$ large enough, 
depending on the implicit constants in \eqref{eqpj}, such that
for any distinct $Q_{j_1},Q_{j_2} \in \F_{egs}$, we have 
\begin{equation}\label{eqpjseparate}
\dist(P_{j_1},P_{j_2}) \,\geq \, \max(\diam Q_{j_1},\diam Q_{j_2})\,.
\end{equation}

We recall that 
$E\subset Q_0$ is a Borel set satisfying \eqref{eqEample}, and 
we remind the reader of the notation in Remark \ref{remarkhm}.
Since $\Omega_\sbf\subset \Omega$,
\begin{multline}
\hm(E) \,\geq\, \sum_{Q_j\in \F_{egs}} \hm(E\cap Q_j')
\\[4pt]=\, \sum_{Q_j\in \F_{egs}} \int_{\pom_\sbf} \hm^{Y}(E\cap Q_j') \,
d\hm_{\star}(Y)
\,\geq\,
\sum_{Q_j\in \F_{egs}} \int_{P_j} \hm^{Y}(E\cap Q_j') \,
d\hm_\star(Y),
\end{multline}
where in the last step we have used \eqref{eqpjseparate}. 
Note that by \eqref{feg4}, \eqref{eqpj} and Harnack's inequality,
\[
\hm^{Y}(E\cap Q_j') \gtrsim \gamma_a\,.
\]
Hence,
\begin{equation}\label{hmotohmos}
\hm(E) \gtrsim \sum_{Q_j\in \F_{egs}} \hm_\star(P_j)
= \hm_\star\left(\cup_{Q_j\in \F_{egs}} P_j\right)=:
\hm_\star(E_\sbf)\,,
\end{equation}
where $E_\sbf:= \cup_{Q_j\in \F_{egs}} P_j$.

We have therefore reduced matters, in Case 1, to proving
\begin{equation}\label{eq5.37}
\hm_\star(E_\sbf) \gtrsim 1\,.
\end{equation}

A remark is in order concerning the proof of \eqref{eq5.37}.
Even though $\Bb$ satisfies a small constant version of \eqref{driftmax} in 
$\Omega_{\sbf}$,
we cannot simply prove a ``small constant" analogue of Theorem \ref{Tmain} 
in $\Omega_{\sbf}$ (which would yield \eqref{eq5.37} directly)
because we do not know that $L_0$ is ``good" in 
$\Omega_\sbf$, i.e., it is not clear that $\hm_{0,\star}$ enjoys the $A_\infty$ 
property with respect to surface measure on $\pom_{\sbf}$.
By hypothesis, such a property holds in 
$\Omega$, but it seems unlikely that it
can be transferred, in general, to the subdomain $\Omega_{\sbf}$.

Instead, our strategy is to use Lemma \ref{lemma:DJK-dyadic-proj},
that is, the dyadic version of the sawtooth lemma 
(the ``Main Lemma") of \cite{DJK}, but here we encounter another obstacle, 
namely that it is not clear that Lemma \ref{lemma:DJK-dyadic-proj} may be applied 
directly to $L$.  Thus, we do something a bit more indirect, inspired by ideas
in \cite{FKP}:  we shall 
 use \eqref{sawtoothainfinity} (which we had deduced 
 from Theorem \ref{Tauxsmall}, our analogue of \cite[Theorem 2.5]{FKP}), 
 to compare $\hm_\star$ to $\hm_{0,\star}$, and then we apply 
 Lemma \ref{lemma:DJK-dyadic-proj} to the operator $L_0$, 
 to compare $\hm_0$ and $\hm_{0,\star}$.

% \steve{\Bl Do we need this?}
% To this end, let $\sigma_\sbf:= \mathcal{H}^n\lfloor_{\,\pom_\sbf}$ 
% denote surface measure on
% $\pom_\sbf$, 
% and observe that
% \begin{equation}\label{eq5.38}
% \sigma_\sbf (E_\sbf) = \sum_{Q_j\in \F_{egs}} \sigma_\sbf(P_j)\, \approx 
%  \sum_{Q_j\in \F_{egs}} \sigma(Q_j) \gtrsim \sigma (Q_0)\,,
% \end{equation}
% where in the last two steps we have used that $\pom_\sbf$ is ADR
% (by Proposition \ref{sawtoothprop} (1)), and \eqref{eqegs}.

Let us now proceed to establish \eqref{eq5.37}.  The first step is now easy.
We view $\pom_{\sbf}$ as a surface ball of radius 
$r=2\diam(\Omega_{\sbf}) \approx \ell(Q_0)$, so that by
\eqref{sawtoothainfinity}, there is a uniform exponent $\alpha>0$ such that
\begin{equation}\label{ainfinityforhmstar}
\hm_\star(E_\sbf) \gtrsim 
\left(\frac{\hm_{0,\star}(E_\sbf)}{\hm_{0,\star}(\pom_{\sbf})}\right)^{\!\alpha}
\hm_{\star}(\pom_{\sbf})\,.
\end{equation}
Since $\hm_\star$ and $\hm_{0,\star}$ 
are probability measures on $\pom_{\sbf}$, it now suffices to prove that
\begin{equation}\label{hm0starbound}
\hm_{0,\star}(E_\sbf) \gtrsim 1\,,
\end{equation}
which we shall do, in turn, using Lemma \ref{lemma:DJK-dyadic-proj}.

To this end, define $\P_\F$ as in \eqref{pfdef}, let
$\nu$ be as in \eqref{defi:nu-bar}, so that $\P_{\F}\nu$ is given by 
\eqref{defi-proj-nubar}. Set
\begin{equation}\label{eq5.45}
E_1:= E \cap \left(\cup_{Q_j\in\F_{egs}} Q_j'\right)\,.
\end{equation}
% \steve{\Bl do we need this?}
% so that by \eqref{feg3}, \eqref{eqegs}, and the 
% pairwise disjointness of the collection $\{Q_j'\}_{Q_j\in\F_{egs}}$, we have
% \begin{equation}\label{eq5.46}
% \sigma(E_1) \gtrsim \sigma(Q_0)\,.
% \end{equation} 
We then have that
\begin{equation*}%\label{nuequal}
\P_\F\nu(E_1) = 
\sum_{Q_j\in\F_{egs}} \frac{\sigma(E\cap Q'_j)}{\sigma(Q_j)}\, 
\hm_{0,\star}(P_j) \approx 
\sum_{Q_j\in\F_{egs}} 
\hm_{0,\star}(P_j) = 
\hm_{0,\star}(E_\sbf)\,,
\end{equation*}
where in the next-to-last step, we have used that, in particular, \eqref{feg3} holds for each
$Q_j\in\F_{egs}$, while the last step is just the definition of $E_\sbf$, along with the separation property \eqref{eqpjseparate}.
Thus, to prove \eqref{hm0starbound}, it is equivalent to show that
\begin{equation*}%\label{eq5.47}
\P_\F\nu(E_1) \gtrsim 1.
\end{equation*}
To this end, we first note that by \eqref{dyadic-DJK:proj}, applied with $Q=Q_0$,
\[
\P_\F\nu(E_1) \gtrsim
\left(\frac{\P_\F \hm_0(E_1)}{\P_\F 
\hm_0(Q_0)}\right)^{\rho} \P_\F\nu(Q_0)\,.
\]
It therefore suffices to verify the following pair of claims.

\smallskip
\noindent{\bf Claim 1}:
\begin{equation}\label{eq5.47}
\frac{\P_\F \hm_0(E_1)}{\P_\F 
\hm_0(Q_0)} \,\gtrsim\, 1\,,
\end{equation}
and 

\smallskip
\noindent{\bf Claim 2}:
\begin{equation}\label{eq5.48}
\P_\F\nu(Q_0) \gtrsim 1\,,
\end{equation}
where in each case the implicit constants depend only on allowable parameters.

\begin{proof}[Proof of Claim 1]
By hypothesis,
$\hm_0\in A_\infty(\sigma,Q_0)$. 
% hence in particular in dyadic $A_\infty$, so that
By definition of $E_1$ \eqref{eq5.45}, and of $\P_\F$ \eqref{pfdef}, 
\begin{multline*}
\P_\F \hm_0(E_1) \,=\,
\sum_{Q_j\in\F_{egs}} \frac{\sigma(E\cap Q'_j)}{\sigma(Q_j)}\,
\hm_0(Q_j)
\\[4pt]
\approx \sum_{Q_j\in\F_{egs}} \hm_0(Q_j)\, =\,
\hm_0(\cup_{\F_{egs}} Q_j) \,\gtrsim\, \hm_0(Q_0)
\,=\, \P_\F\hm_0(Q_0)\,,
\end{multline*}
where in the second line of the display, we have used
\eqref{feg3}, then \eqref{eqegs} and the $A_\infty$ property of 
$\hm_0$,
 and finally \eqref{mu=pmuoncubes}.  
Thus, Claim 1 \eqref{eq5.47} holds.
\end{proof}

\begin{proof}[Proof of Claim 2] 
By the formula for $\P_\F\nu$ in \eqref{defi-proj-nubar},
\begin{multline} \label{eq5.54}
\P_\F\nu(Q_0)
= \hm_{0,\star}\left(Q_0\setminus (\cup_\F Q_j)\right) + 
\sum_\F  \hm_{0,\star}(P_j)
\\[4pt]
\gtrsim \,
\hm_{0,\star}\left(Q_0\setminus (\cup_\F Q_j)\right) + 
\sum_\F  \hm_{0,\star}(M P_j)\,,
\end{multline}
where in the last line we have used the doubling property of 
$\hm_{0,\star}$ to replace
the surface ball $P_j$ 
by the concentric dilate
$M P_j$ defined in Proposition \ref{prop:Pj2}, 
where $M$ is the dilation factor in Corollary \ref{corPj}.
Let $\Delta_\star$ be the surface ball on $\pom_\sbf$ given by 
Proposition \ref{prop:Pj2}.  
Then by \eqref{eq5.54} and Corollary \ref{corPj},
\begin{equation}\label{eq5.55a}
\P_\F\nu(Q_0) \gtrsim \hm_{0,\star}(\Delta_\star) \gtrsim 1\,,
\end{equation}
where in the second inequality we have used
Lemma \ref{Bourgainhm}, 
Harnack's inequality and the Harnack Chain property,
all applied in $\Omega_\sbf$.
% here, the last inequality follows from Lemma \ref{Bourgainhm} applied in $\Omega_\sbf$, Harnack's inequality and the Harnack Chain property.  
Of course, \eqref{eq5.55a} yields Claim 2 \eqref{eq5.48}. 
\end{proof}

We have now
completed the proof of the implication $H[a]\implies H[a+b]$ 
in the scenario of Case 1, i.e., \eqref{case1}. \hfill\(\Box\)

\smallskip

\noindent{\em Proof of $H[a] \!\!\implies \!\!H[a+b]$ in Case 2}.
We now turn to Case 2 \eqref{case2}.  Recall that $E$ satisfies \eqref{eqEample}.
Hence, for $\eta=\eta(\theta)$ small enough, \eqref{case2} yields
\begin{equation}\label{Ecase2}
 \sigma(E'):=\sigma \left(E\setminus \left(\cup_\F Q_j\right) \right) 
 \geq \frac{\theta}{4} \sigma (Q_0)\,,
\end{equation}
where $E':=E\setminus \left(\cup_\F Q_j\right)$.  
Since $E'\subset \pom\cap\pom_{\sbf}$, by the maximum principle and then \eqref{sawtoothainfinity}, we find that for some uniform exponent $\alpha>0$,
\begin{equation}\label{ainfinityforhmstar2}
\hm(E)\geq \hm_\star(E') \gtrsim 
\left(\frac{\hm_{0,\star}(E')}{\hm_{0,\star}(\pom_{\sbf})}\right)^{\!\alpha}
\hm_{\star}(\pom_{\sbf})\,,
\end{equation}
Since $\hm_\star$ and $\hm_{0,\star}$ 
are probability measures on $\pom_{\sbf}$, it now suffices to prove that
\begin{equation}\label{hm0starbound2}
\hm_{0,\star}(E') \gtrsim 1\,.
\end{equation}

To this end, note that by definition of $\P_\F\nu$ \eqref{defi-proj-nubar}, and then 
\eqref{dyadic-DJK:proj} applied with $Q=Q_0$,
\[
\hm_{0,\star}(E') = \P_\F\nu(E') \gtrsim
\left(\frac{\P_\F \hm_0(E')}{\P_\F 
\hm_0(Q_0)}\right)^{\rho} \P_\F\nu(Q_0)\,.
\]
By definition of $\P_\F$ \eqref{pfdef}, and \eqref{mu=pmuoncubes},
we have
\[
\frac{\P_\F \hm_0(E')}{\P_\F 
\hm_0(Q_0)} = \frac{\hm_0(E')}{\hm_0(Q_0)} \approx 1\,,
\]
where in the last step we have used \eqref{Ecase2} and that
by hypothesis, $\hm_0\in A_\infty(Q_0,\sigma)$.
Moreover, exactly as in \eqref{eq5.55a},
$\P_\F\nu(Q_0) \gtrsim 1$, so that \eqref{hm0starbound2} follows.
This concludes the proof that $H[a]\!\!\implies \!\!H[a+b]$ in Case 2. \hfill\(\Box\)
 
Thus, $H[a]\!\!\implies \!\!H[a+b]$ in both cases, and the proof of Theorem \ref{Tmain} is 
complete.

\section{The doubling property, and the proof of Theorem \ref{T2}}\label{appdoubl}

We show that for $L_0$ and $L$ as in Theorem \ref{Tmain}, the doubling property holds for 
the corresponding elliptic-harmonic measure $\hm_L$.  We 
follow the argument in \cite[Ch. III, Section 4]{HL}, adapted to our setting.  
% Once this has been done, we deduce a certain converse to our main theorem.
Given a domain $\Omega$, 
we suppose throughout this section that
\[
L=-\dv A \nabla +\Bb\cdot\nabla\,,
\]
where 
\begin{equation*} % \label{eq7.driftsize}
|\Bb(X)|\,\leq\, \sqrt{M_0} \,\delta(X)^{-1}\,,\qquad \text{a.e. } X\in \Omega\,.
\end{equation*} 
Eventually, we shall assume that $\Bb$ satisfies additional properties, to be specified as required.
We let $\hm=\hm_L$ denote the elliptic-harmonic measure for $L$, in the domain $\Omega$ under consideration. Recall that
 \begin{equation}\label{7.DXdef}
\Delta_Y:=B\big(Y,10\delta(Y)\big)\cap\pom\,,\qquad Y\in\Omega\,.
\end{equation}

We first record extensions of \eqref{eq2.14} and Lemma \ref{l2.10}.

\begin{lemma}\label{lemma7.CFMS}
Let $\Omega\subset \ree,\, n\geq 2$, be a 1-sided CAD,
and assume that $\Bb$ is 
compactly supported in $\Omega$.  
Then for $\Phi \in Y^{1,2}(\Omega) \cap C(\overline{\Omega})$,
 % C(\overline{\Omega})\cap Y^{1,2}(\Omega)$, 
and $X\in\Omega$, we have
\begin{equation}\label{eq7.rieszformula}
% \int_{\partial\Omega} \Phi\,d\omega^X 
u(X) \,-
\, \Phi(X)
=
-\iint_\Omega
\big[A^{\tt \!T}(Z)\nabla_Z G(X,Z) \,+\, \Bb(Z) G(X,Z)\big]\cdot\nabla\Phi(Z)\, dZ\,,
\end{equation}
where $u$ is the % $W^{1,2}$ (or $Y^{1,2}$ if $\Omega$ is unbounded)
solution of the Dirichlet problem $Lu=0$ in $\Omega$, $u\in 
Y^{1,2}(\Omega) \cap C(\overline{\Omega})$, with 
$u\lfloor_{\pom}=\Phi\lfloor_{\pom}$.
Next, suppose further
that there is a uniform constant $\gamma >0$
such that 
\begin{equation}\label{eq7.Bourgain1}
\omega^{Y} (\Delta_Y)  
\geq \gamma \,, \qquad \forall\, Y\in\Omega\,.
% \qquad \forall \, X,Y \in \Omega
% \text{\em \, such that } |X-Y|\le \delta(Y)/4\,. %\Omega_{r/2}:= \Omega \cap B(x,r/2)\,.
\end{equation}
% on $n$, ellipticity of $A$, and the 1-sided CAD constants, then 
Then for all $X,Y\in \Omega$ such that $|X-Y|\ge \delta(Y)/4$, we have
\begin{equation}\label{7.eqn:right-CFMS}
\frac{G(X,Y)}{\delta(Y)}\,
\lesssim_\gamma
\,\frac{\hm^X( \Delta_Y)}{\sigma( \,\Delta_Y)}\,,
\end{equation}
% \steve{\Bl note to myself:  specify $N=20$? 30?}
and there is a positive constant $N$ such that
for all $X,Y\in \Omega$ with $|X-Y|>N\delta(Y)$,
\begin{equation}\label{7.eqn:left-CFMS}
\frac{\hm^X( \Delta_Y)}{\sigma( \,\Delta_Y)}\,
\lesssim\, \frac{G(X,Y)}{\delta(Y)}\,,
\end{equation}
%$\Delta_Y=B(Y,10\delta(Y))\cap \pom$, and 
where $N$ and the implicit constants depend only on allowable parameters, but not on the qualitative assumption that $\Bb$ has compact support in $\Omega$.
% and $\kappa_0$ is the constant fixed in Lemma \ref{localHardy}.
% with $\hat{y}\in\pom$ such that $|Y-\hat{y}|=\delta(Y)$.
\end{lemma}
\begin{remark}
Observe that \eqref{eq7.rieszformula}
is a mild extension % simply a restatement 
of \eqref{eq2.14}, % where we now permit a slightly more general class of $\Phi$, 
where we now allow unbounded $\Omega$ 
and a slightly more general class of $\Phi$, % to be unbounded.
and that \eqref{eq7.Bourgain1} is essentially a restatement of \eqref{eq2.Bourgain1}.  In 
Lemma \ref{lemma7.CFMS}, \eqref{eq7.Bourgain1} replaces the smallness 
condition \eqref{smallptwise} of Lemma \ref{l2.10}. 
\end{remark}

\begin{proof}[Proof of Lemma \ref{lemma7.CFMS}]
Given the qualitative assumption that $\Bb$ has compact support in $\Omega$, we obtain \eqref{eq7.rieszformula}
directly from the results of either \cite{S} or \cite{KS} (in the case that $\Omega$ is bounded),
or from \cite{M} (in the case that $\Omega$ is unbounded).  
Let us now turn to the quantitative estimates.

\begin{proof}[Proof of \eqref{7.eqn:right-CFMS}]  Note that $\sigma(\Delta_Y)\approx \delta(Y)^n$
by the ADR property of $\pom$, and that
\begin{equation}\label{eq7.Bourgain1a}
\omega^{X} (\Delta_Y)  \gtrsim \gamma\,,
\qquad \forall\, X,Y \in \Omega
\text{ such that } |X-Y|\le \delta(Y)/4\,,
\end{equation}
 by \eqref{eq7.Bourgain1} and Harnack's inequality.
Then by a standard argument, it suffices to prove that
\begin{equation}\label{eq7.ptwisegreen}
G(X,Y) \lesssim_\gamma |X-Y|^{1-n}\,,\qquad |X-Y|= \delta(Y)/4\,,
\end{equation}
where the implicit constant may depend on $\gamma$, and on the usual allowable parameters. 
Indeed, viewing $Y$ as fixed,
we obtain \eqref{7.eqn:right-CFMS} by combining \eqref{eq7.ptwisegreen} with
\eqref{eq7.Bourgain1a}, and then using the maximum principle in 
$\Omega\setminus \overline{B(Y,\delta(Y)/4)}$
(see, e.g., \cite[Proof of Lemma 1.3.3]{K}).

To verify \eqref{eq7.ptwisegreen}, 
we follow closely the argument in \cite[Ch.~III, Proof of Lemma 4.3]{HL}.  Fix
$Y\in \Omega$, set $R:= \delta(Y)$, let $\xi \in C_0^\infty (B(Y, 12 R))$, with 
$0\leq \xi\leq 1$ and $\xi\equiv 1$ on $B(Y, 10R)$,
and set $v(Z):= \int_{\pom}\xi(x) \,d\hm^Z(x)$.  Define
\[
U_\eta(Y):= \left( B\big(Y,MR\big) \setminus
B\big(Y,20R\big)\right) \cap \big\{Z\in \Omega:  \eta R/M\leq \delta(Z)\leq \eta R\big\},
\]
where $M$ is simply chosen large enough, depending on the corkscrew constants for 
$\Omega$, to ensure that $U_\eta(Y)$ is non-empty. 
Then by \eqref{eq7.Bourgain1}
and an iteration argument, there exists a uniform exponent $\alpha>0$ such that
\[
v(Z) \lesssim \big(\delta(Z)/R\big)^\alpha 
\lesssim \eta^\alpha\,,
\qquad \forall\, Z\in U_\eta(Y)\,.
\]
Note also that $v(Z)\geq \hm^Z(\Delta_Y)$, by construction.
 Choosing
$\eta$ small enough depending on $\gamma$, fixing a point $Z_0\in U_\eta(Y)$ and a point
$X_0$ with $|X_0-Y|= R/2$,
and then using \eqref{eq7.Bourgain1}, Harnack's inequality
and Poincare's inequality, we see that for an appropriate closed set $U^*_\eta(Y)$,
\begin{equation}\label{eq7.lowerbound1}
\gamma \lesssim v(X_0) - v(Z_0) \lesssim_\eta \, 
R \left(R^{-1-n}\iint_{U^*_\eta(Y)} |\nabla v(Z)|^2dZ\right)^{1/2}\,,
\end{equation}
where $U^*_\eta(Y)$ satisfies
\[
B(X_0,R/8)\cup\{Z_0\}\subset U^*_\eta(Y)\subset \Omega \setminus B(Y,R/3)\,,
\qquad \diam(U^*_\eta(Y))\lesssim R\,.
\] 
Consider any $X$ with $|X-Y|= R/4$,
and set $m:= \min_{Z\in U^*_\eta(Y)} G(X,Z)$.  Squaring the inequality in
\eqref{eq7.lowerbound1}, and using that $Lv=0$, hence
$A\nabla v\cdot\nabla v =-(1/2) L(v^2)$ in the weak sense, we then have
\begin{multline*}
\gamma^2 m \,\lesssim \, R^{1-n} \iint_{U^*_\eta(Y)} G(X,Z) \, |\nabla v(Z)|^2dZ % \\[4pt]
\,\lesssim\, R^{1-n} \int_\Omega G(X,Z)\, A(Z)\nabla v(Z)\cdot \nabla v(Z) dZ
\\[4pt] 
= \, -\frac12 R^{1-n} \iint_\Omega
\big[A^{\tt \!T}(Z)\nabla_Z G(X,Z) \,+\, \Bb(Z) G(X,Z)\big]\cdot\nabla v^2(Z)\, dZ
\\[4pt] 
= \frac12 R^{1-n}\left(u(X) - v^2(X)\right) \lesssim R^{1-n}\,,
\end{multline*}
where we have used
\eqref{eq7.rieszformula} with $\Phi=v^2$, and $u$ is the solution of the Dirichlet problem for 
$L$, with data $u\lfloor_{\pom} = v^2\lfloor_{\pom} = \xi^2\lfloor_{\pom}$.
Let $Z_1\in U^*_\eta(Y)$ be a point where the minimum value is achieved, i.e.,
$G(X,Z_1)=m$.  By the Harnack chain condition, we may choose a collection of at most
$C(\eta)$ balls connecting $Z_1$ to $Y$, 
each of radius $r\gtrsim \eta R$, and each missing the ball $B(X,R/100)$.  Since $\eta$ depends only on $\gamma$, using Harnack's inequality and the definition of $R$, we obtain 
\eqref{eq7.ptwisegreen}.  This completes the proof of \eqref{7.eqn:right-CFMS}.
\end{proof}

\begin{proof}[Proof of \eqref{7.eqn:left-CFMS}] 
With \eqref{7.eqn:right-CFMS} in hand,
we now follow the argument in \cite[Section 3]{HM-I}; see also the earlier argument in  
\cite[Chapter III, proof of Lemma 4.6]{HL}, which is similar in spirit.  

It will be convenient to reformulate \eqref{7.eqn:left-CFMS} slightly.
For each ball $B$ centered on $\pom$, of radius $r_B<5\diam(\pom)$, 
and each $a\in(0,1)$, set
\begin{equation}\label{Uabdef}
U^a_B:=\{Y\in \Omega\cap B: \delta(Y) >a r_B\}.
\end{equation}
Note that we may specify a particular choice of the parameter $a$, call it 
$a_0$, depending only on the constant in the Corkscrew condition, such that 
$U^{a_0}_B$ is non-empty.  In the sequel, with this 
value of $a_0$ now fixed, we set $U_B:= U_B^{a_0}$.
Define
\begin{equation}\label{defi-upsilon}
\Xi :=\sup_{X\in\Omega\setminus 4B,\, Y\in U_B} \frac{r_B^{1-n} \omega^X(\Delta)}{G(X,Y)}\,,
\end{equation}
where the supremum runs over all $B$ 
centered on $\pom$, of radius $r_B<5\diam(\pom)$, and all $X,Y$ as specified, and where
$\Delta:= B\cap\pom$.   One may readily check that, 
to establish \eqref{7.eqn:left-CFMS},
it suffices to show that $\Xi\leq C$, for some finite constant $C$
depending only on allowable parameters.
Observe that $\Xi<\infty$, by our qualitative assumption
that $\Bb$ has compact support in $\Omega$.
Thus, it will suffice to show that 
\begin{equation}\label{eqXiest}
\Xi \,\leq\, C_1+ \frac12 \,\Xi,
\end{equation}
where $C_1$ depends only on allowable parameters.  We now fix a ball $B$ centered on
$\pom$, with radius $r_B<5\diam(\pom)$, a point 
$X\in\Omega\setminus 4B$, and a point 
$Y\in U_B$, such that
\begin{equation}\label{XiBXY}
\Xi \,\leq\, 2\,\frac{r_B^{1-n} \omega^X(\Delta)}{G(X,Y)}\,,
\end{equation}
where as above, $\Delta:=B\cap\pom$.

Let $\epsilon >0$ be a sufficiently small number to be chosen, and
define $U^\epsilon_{2B}$ as in \eqref{Uabdef} with $a=\epsilon$, and with $2B$ in place of $B$.  Set
$L^\epsilon_{2B}:= (2B\cap\Omega)\setminus U^\epsilon_{2B}$.
For convenience of notation, set $R:= r_B$. 
Taking $\Phi \in C_0^\infty(2B)$, with $0\leq\Phi\leq 1$,  $\Phi(Y)\equiv 1$
on $B$, and $\|\nabla \Phi\|_\infty \lesssim 1/R$, and using that $X$ lies outside
$\supp (\Phi)$, we deduce 
from \eqref{eq7.rieszformula} that
\begin{multline*} % \label{eq2.15}
\omega^X(\Delta)\leq \int_{\pom}\Phi d\hm^X \lesssim 
\frac1R \iint_{\Omega\cap 2B}
\big(\left|\nabla_Z G(X,Z)\right| +|\Bb(Z)| G(X,Z)\big)dZ
\\[4pt]
=\,\frac1R \iint_{U^\epsilon_{2B}}\,...\,\, dZ \,+\,\frac1R \iint_{L^\epsilon_{2B}}\,...\, \,dZ\,
=:\,{\tt A}\,+\, {\tt B}\,.
\end{multline*}

We further split the terms ${\tt A}$ and ${\tt B}$ as follows:
\[
{\tt A} \,=\, \frac1R \iint_{U^\epsilon_{2B}}
|\nabla_Z G(X,Z)|\,dZ
\,+\, \frac1R \iint_{U^\epsilon_{2B}}
|\Bb(Z)| G(X,Z)\,dZ\,=:\,{\tt A}_1\,+\,{\tt A}_2 \,,
\]
and 
\[
{\tt B} \,=\, \frac1R \iint_{L^\epsilon_{2B}}
|\nabla_Z G(X,Z)|\,dZ
\,+\, \frac1R \iint_{L^\epsilon_{2B}}
|\Bb(Z)| G(X,Z)\,dZ\,=:\,{\tt B}_1\,+\,{\tt B}_2 \,.
\]

We treat term ${\tt A}_1$ first.
Let $\W^\epsilon_{2B}$ be a covering of $U^\epsilon_{2B}$ by Whitney cubes,
and note that we may construct this covering so that
\[
\bigcup_{I\in \W^\epsilon_{2B}} 2I \subset \Omega\cap 3B\cap \{\delta(Z)>\epsilon R/2\}\,,
\]
and such that the doubled Whitney cubes have bounded overlaps and
satisfy $\diam(I) \approx \delta(Z)$ for all $Z\in 2I$. 
Using Cauchy-Schwarz, then the Whitney covering and
Caccioppoli's inequality, and finally Harnack's inequality and the definition of $R$,
we see that
\[
{\tt A}_1\,\lesssim \, R^{-1+(n+1)/2}\left(\sum_{I\in \W^\epsilon_{2B}}
\iint_{2I} \left(\frac{G(X,Z)}{\delta(Z)}\right)^2 dZ\right)^{1/2}
\lesssim_{\epsilon}\, r_B^{n-1} G(X,Y)\,,
\]
where $Y\in U_B$ is the point specified in \eqref{XiBXY}.
A similar (but simpler) argument, using that $|\Bb(Z)|\lesssim \delta(Z)^{-1}$, yields a bound 
on the same order for ${\tt A}_2$, and therefore
\[
\frac{r_B^{1-n}}{G(X,Y)} \,\,{\tt A}\, \leq\, C_\epsilon\,.
\]
Once we have fixed $\epsilon$, we will set $C_1:=2C_\epsilon$ and use \eqref{XiBXY}
to obtain the first part of the estimate in \eqref{eqXiest}.

Next, we turn to term ${\tt B}_1$.  Let $\W_{\epsilon,2B}$
be a covering of $L^\epsilon_{2B}$ by Whitney cubes,
satisfying
\[
\bigcup_{I\in \W_{\epsilon, 2B}} 2I \subset \Omega\cap 3B\cap \{\delta(Z)\leq 2\epsilon R\}\,,
\]
along with the usual Whitney properties as above.  
For each Whitney cube $I\in \W_{\epsilon,2B}$, let $\Delta_I$ be a surface ball on $\pom$,
such that $\diam(\Delta_I) \approx \ell(I)\approx  \delta(Z)$ and $\Delta_Z\subset \Delta_I$,
for all $Z\in 2I$.   

For each integer $k$, set
\[
\W^k_{\epsilon, 2B}\,:= \,\{I\in \W_{\epsilon, 2B}: \ell(I)= 2^{-k}\}\,. 
\]
We may assume that $\W^k_{\epsilon, 2B}$ is non-empty only if both
$2^{-k} \leq 10 \epsilon R$, and (provided that we take $\epsilon$ sufficiently small)
$\Delta_I \subset 3\Delta:= 3B\cap\pom$. 
Note that for each $k$ fixed, we have the bounded overlap property
\[
\sum_{I\in \W^k_{\epsilon, 2B}} 1_{\Delta_I} (Z) \,\lesssim\, 1\,,\qquad Z\in\Omega\,.
\]
Therefore, using first the Whitney covering, then Cauchy-Schwarz and
Caccioppoli's inequality, 
and finally \eqref{7.eqn:right-CFMS}
and the preceeding observations, we see that
\begin{multline*}
{\tt B}_1 \lesssim \,R^{-1} \!\!
\sum_{k:\, 2^{-k}\leq 10\epsilon R}\,\sum_{I\in \W^k_{\epsilon, 2B}}
\ell(I)^{(n+1)/2}
\left(\iint_{2I} \left(\frac{G(X,Z)}{\delta(Z)}\right)^2 dZ\right)^{1/2}
\\[4pt]
\lesssim \, R^{-1}\!\!
\sum_{k:\, 2^{-k}\leq 10\epsilon R}\,\sum_{I\in \W^k_{\epsilon, 2B}}
\ell(I)^{(n+1)/2}
\left(\iint_{2I} \left(\frac{\hm^X( \Delta_Z)}{\delta(Z)^n}\right)^2 dZ\right)^{1/2}
\\[4pt]
\lesssim \, R^{-1}\!\!
\sum_{k:\, 2^{-k}\leq 10\epsilon R} 2^{-k}
\sum_{I\in \W^k_{\epsilon, 2B}}
\hm^X( \Delta_I) \,\lesssim \, \epsilon \hm^X(3\Delta)\,.
\end{multline*}
A similar (but simpler) argument yields a bound 
on the same order for ${\tt B}_2$.   We now cover
$3\Delta$ by a collection of balls $\{B_i\}_{i=1}^m$, each centered in $3\Delta$, 
with $r_{B_i}= r_B/8$ for each $i$, 
and where the cardinality $m$
is uniformly bounded depending only on dimension and the ADR constants.
Note that $X\in\Omega\setminus 4B$ implies $X\in\Omega\setminus 4B_i$ for each $i$, by construction.
We then have
\[
{\tt B}\,\lesssim \,\epsilon \sum_{i=1}^m \hm^X(\Delta_i) \,
\lesssim \, \epsilon \hm^X(\Delta_{i_0})\,,
\]
where for some particular $i_0$, the last inequality follows by pigeon-holing, since 
$m\approx 1$.

Recall that we have fixed $Y\in U_B$, hence 
$G(X,Y) \approx G(X,Y')$, uniformly for all $Y'\in U_{B_{i_0}}$,
by Harnack's inequality.
Therefore,
\[
\frac{r_B^{1-n}}{G(X,Y)}\,\,{\tt B} \, \leq\, C\epsilon\,\Xi \,\leq \,\frac14 \Xi\,,
\]
provided that we choose $\epsilon$ small enough.  Estimate
\eqref{eqXiest}, and thus also \eqref{7.eqn:left-CFMS}, now follow.
\end{proof}
This concludes the proof of Lemma \ref{lemma7.CFMS}.  
\end{proof}

\begin{corollary} Suppose that $\Omega\subset \ree,\, n\geq 2$, is a 1-sided CAD,
let $L_0$ and $L$ be as in Theorem \ref{Tmain}, and let $\hm$ be
the elliptic-harmonic measure corresponding to $L$ (which exists, as shown in the proof of
Theorem \ref{Tmain}:  see the discussion following Lemma \ref{BLlemma}).  
Then $\hm$ satisfies the doubling property: % \eqref{eq3.hmdouble} holds for $\hm$.
\begin{equation}  \label{eq7.hmdouble}
\hm^X(2\Delta) \,\lesssim\, \hm^X(\Delta)\,,\qquad X\in\Omega\setminus 4B\,,
\end{equation}
for any $\Delta =B\cap\pom$, where $B$ is a ball centered on $\pom$, of radius
$r_B<\diam{\pom}$.  The implicit constant in \eqref{eq7.hmdouble} depends only on the allowable parameters for Theorem \ref{Tmain}.
\end{corollary}
\begin{proof}
In the proof of Theorem \ref{Tmain}, we showed that $\hm_k$, the elliptic-harmonic measure corresponding to the operator $L_k$ defined in Subsection \ref{approxhm},
satisfies the Bennewitz-Lewis criterion \eqref{eq3.1b}, uniformly in $k$.  
In particular, taking $E=\Delta_X$ in 
\eqref{eq3.1b}, and then using Harnack's inequality,
we obtain \eqref{eq7.Bourgain1}, again uniformly in $k$.  
We may then invoke Lemma \ref{lemma7.CFMS}
to see that the conclusion of 
Lemma \ref{l2.10} holds, so in turn Lemma \ref{lemmadouble} (the doubling property) 
remains valid for $\hm_k$, and once again the quantitative bounds are all
uniform in $k$.  We may then invoke the weak convergence result in 
Lemma \ref{hmklimit} to obtain \eqref{eq7.hmdouble}.
\end{proof}

\begin{remark}
In \cite{P}, a counter-example is presented for which
the Green function 
estimate \eqref{eq7.ptwisegreen} fails. 
On the other hand, while the counter-example of \cite{P} is valid, nonetheless
\eqref{eq7.ptwisegreen} does hold {\em in the presence of \eqref{eq7.Bourgain1}}, as was established in \cite{HL}.  Indeed, the proof of \eqref{eq7.ptwisegreen}
given here follows that of \cite[Ch.~III, Lemma 4.3]{HL}.  
\end{remark}

Finally, with the doubling property in hand, we are ready to give the 
short proof of 
Theorem \ref{T2}.

\begin{proof}[Proof of Theorem \ref{T2}]
Given a subdomain $\Omega'\subset\Omega$,
let $\hm',\hm_0'$, and $\sigma'$ denote, respectively,
elliptic measure for $L$ and for $L_0$ in $\Omega'$, and 
surface measure on $\pom'$.
Under the hypotheses of Theorem \ref{T2}, by the result of \cite{AHMMT}
(see also Remark \ref{remarkcompare}), $\Omega$ 
satisfies an interior big pieces of
chord-arc subdomains condition (Definition \ref{def1.ibp}).  
Thus, by Lemma \ref{BLlemma} (the result of \cite{BL}), and, e.g.,
\cite[Proposition 13]{H}\footnote{\cite[Proposition 13]{H} 
essentially originates in \cite{JK}, and is also implicit in \cite{BL}.}, 
it suffices to show that for each chord-arc subdomain
$\Omega'\subset \Omega$, we have 
$\hm'\in A_\infty(\sigma')$\footnote{We remark that doubling is essential here : 
to carry out the ``big pieces" argument, one needs the 
full strength of $A_\infty$, not just weak-$A_\infty$.}
 in the sense of
Definition \ref{deflocalAinfty}, and that \eqref{eq7.Bourgain1} holds for 
$\hm'$ in $\Omega'$, in each case 
with uniform constants depending only on the structural constants in 
Theorem \ref{T2} and on the chord-arc constants for $\Omega'$.
In turn, to establish the desired properties of $\hm'$, it
suffices, by Theorem \ref{Tmain} and the arguments of the present section,
to show that the hypotheses of Theorem \ref{Tmain}
hold (uniformly) for $L_0$ and $L$ in $\Omega'$.  
To see that these hypotheses are indeed verified, note first that
$\delta'(X) :=\dist(X,\pom')\leq \delta(X)$, for $X\in\Omega'$, and therefore
\eqref{driftmax}, as well as the DKP estimates for $|\nabla A|$, hold uniformly in each chord-arc subdomain $\Omega'$.  Recall that by \cite{DJe}, the chord-arc domain $\Omega'$ has ``Interior Big Pieces of Lipschitz Subdomains",
defined as in Definition \ref{def1.ibp}, but where now the subdomains
$\Omega_X$ are Lipschitz subdomains with uniform constants, and again
note that the DKP condition continues to hold uniformly in these Lipschitz subdomains.  Invoking the result of \cite{KP} in the Lipschitz subdomains, and using the same big pieces argument as before (i.e., invoking 
Lemma \ref{BLlemma} and
\cite[Proposition 13]{H}), we find that $\hm_0'\in A_\infty(\sigma')$ on 
$\pom'$ in the sense of 
Definition \ref{deflocalAinfty}.  The hypotheses of Theorem \ref{Tmain} are therefore verified in $\Omega'$.

We conclude by remarking that the DKP condition, {\em per se}, is not required in the preceding argument, but rather only the fact that 
$\hm_0'\in A_\infty(\sigma')$ uniformly in the chord-arc subdomains of $\Omega$.
\end{proof}

\end{document}